\documentclass[11pt]{article}
\usepackage{amsmath,amssymb,amsthm}
\usepackage{bbm,natbib,color,dsfont}
\usepackage[margin=1in]{geometry}
\usepackage{caption,rotating,graphicx}
\usepackage[hang,flushmargin]{footmisc}
\usepackage[colorlinks=true,citecolor=blue]{hyperref}
\usepackage{microtype}
\usepackage{ragged2e}
\newtheorem{theorem}{Theorem}
\newtheorem{lemma}{Lemma}
\newtheorem{proposition}{Proposition}

\newtheorem{definition}{Definition}
\newtheorem{corollary}{Corollary}

\newtheorem{example}{Example}

\newcommand{\R}{\mathbb{R}}

\newcommand{\EE}[1]{\mathbb{E}\left[{#1}\right]}
\newcommand{\EEst}[2]{\mathbb{E}\left[{#1}\  \middle| \ {#2}\right]}
\newcommand{\Ep}[2]{\mathbb{E}_{{#1}}\left[{#2}\right]}
\newcommand{\Epst}[3]{\mathbb{E}_{{#1}}\left[{#2}\  \middle| \ {#3}\right]}
\newcommand{\PP}[1]{\mathbb{P}\left\{{#1}\right\}}
\newcommand{\PPst}[2]{\mathbb{P}\left\{{#1}\  \middle| \ {#2}\right\}}
\newcommand{\Ppst}[3]{\mathbb{P}_{{#1}}\left\{{#2}\  \middle| \ {#3}\right\}}
\newcommand{\Pp}[2]{\mathbb{P}_{{#1}}\left\{{#2}\right\}}

\newcommand{\one}[1]{{\mathds{1}}_{{#1}}}
\newcommand{\iidsim}{\stackrel{\textnormal{i.i.d.}}{\sim}}

\newcommand{\Bcal}{\mathcal{B}}
\newcommand{\Xcal}{\mathcal{X}}
\newcommand{\BX}{\mathcal{B}(\Xcal)}
\newcommand{\BXn}{\mathcal{B}(\Xcal^n)}
\newcommand{\BXinf}{\mathcal{B}(\Xcal^\infty)}
\newcommand{\Pcal}{\mathcal{P}}
\newcommand{\Mcal}{\mathcal{M}}
\newcommand{\Ecal}{\mathcal{E}}
\newcommand{\Ftail}{\mathcal{F}_{\textnormal{tail}}}

\newcommand{\Fcal}{\mathcal{F}}
\newcommand{\Scal}{\mathcal{S}}
\newcommand{\dtv}{\mathsf{d}_{\mathsf{TV}}}

\def\d{\mathsf{d}}
\def\asto{\overset{\textnormal{a.s.}}{\to}}
\def\aseq{\overset{\textnormal{a.s.}}{=}}
\def\th{^{\textnormal{th}}}

\DeclareFontFamily{U}{mathx}{\hyphenchar\font45}
\DeclareFontShape{U}{mathx}{m}{n}{
      <5> <6> <7> <8> <9> <10>
      <10.95> <12> <14.4> <17.28> <20.74> <24.88>
      mathx10
      }{}
\DeclareSymbolFont{mathx}{U}{mathx}{m}{n}
\DeclareFontSubstitution{U}{mathx}{m}{n}
\DeclareMathAccent{\widebar}{0}{mathx}{"73}
\DeclareMathAccent{\widecheck}{0}{mathx}{"71}

\def\dF{\mathsf{dF}}
\def\01{\mathsf{01}}
\def\LLN{\mathsf{LLN}}

\title{De Finetti's theorem and related results for \\ infinite weighted 
  exchangeable sequences}  
\author{
Rina Foygel Barber\thanks{Department of Statistics, University of Chicago}, 
Emmanuel J.~Cand{\`e}s\thanks{Departments of Statistics and Mathematics,
  Stanford University}, 
Aaditya Ramdas\thanks{Departments of Statistics and Machine Learning, Carnegie
  Mellon University} , 
Ryan J.~Tibshirani\thanks{Department of Statistics, University of California,
  Berkeley}}  
\date{}

\begin{document}  
\maketitle

\begin{abstract}
De Finetti's theorem, also called the de Finetti--Hewitt--Savage theorem, is a
foundational result in probability and statistics. Roughly, it says that an
infinite sequence of exchangeable random variables can always be written as a 
mixture of independent and identically distributed (i.i.d.) sequences of random
variables. In this paper, we consider a weighted generalization of
exchangeability that allows for weight functions to modify the individual 
distributions of the random variables along the sequence, provided that---modulo
these weight functions---there is still some common exchangeable base
measure. We study conditions under which a de Finetti-type representation
exists for weighted exchangeable sequences, as a mixture of distributions which
satisfy a weighted form of the i.i.d.\ property. Our approach establishes a
nested family of conditions that lead to weighted extensions of other well-known
related results as well, in particular, extensions of the zero-one law and the
law of large numbers.  
\end{abstract}

\section{Introduction}\label{sec:intro}

Nearly 100 years ago, \citet{definetti1929funzione} established a result that
connects an infinite exchangeable sequence of binary random variables to a
mixture of i.i.d.\ sequences of binary random variables. De Finetti's result
says that $X_1,X_2, \dots \in \{0,1\}$ are exchangeable if and only if a draw 
from their joint distribution can be equivalently represented as:
\[
\textnormal{sample $p \sim \mu$, then draw $X_1,X_2,\dots \iidsim
  \mathrm{Bernoulli}(p)$},
\]
for some distribution $\mu$ on $[0,1]$. In other words, any infinite
exchangeable binary sequence can be represented as a mixture of i.i.d.\
Bernoulli sequences.

This result has been extended well beyond the binary case, to a general space
$\Xcal$, due to work by de Finetti and others, most notably
\citet{hewitt1955symmetric}. The more general result is known as the de
Finetti--Hewitt--Savage theorem, but is also often simply called de Finetti's
theorem. It says that, under fairly mild assumptions on $\Xcal$, any infinite
exchangeable sequence $X_1,X_2, \dots \in \Xcal$ can be represented as a mixture
of i.i.d.\ sequences. See Theorem \ref{thm:definetti} below for a formal
statement, and the paragraphs that follow for more references and discussion of
the history of contributions in this area. De Finetti's theorem is widely
considered to be of foundational importance in probability and statistics. Many
authors also view it as a central point of motivation for Bayesian inference;
e.g., see \citet{schervish2012theory} for additional background on the role de
Finetti's theorem plays in statistical theory and methodology.

Some authors have also studied a \emph{weighted} notion of exchangeability. 
This was considered by \citet{lauritzen1988extremal} in the binary case $\Xcal =
\{0,1\}$, and by \citet{tibshirani2019conformal} for general $\Xcal$ in the
context of predictive inference under distribution shift. For example, in the
case of continuously-distributed and real-valued random variables
$X_1,\dots,X_n$, we say (following \citet{tibshirani2019conformal}) that their
distribution is \emph{weighted exchangeable} with respect to given weight
functions $\lambda_1,\dots,\lambda_n$ if the joint density $f_n$ of
$X_1,\dots,X_n$ factorizes as 
\[
f_n(x_1,\dots,x_n) = \bigg( \prod_{i=1}^n \lambda_i(x_i) \bigg) \cdot
g_n(x_1,\dots,x_n), \quad \textnormal{for $x_1,\dots,x_n \in \R$},
\] 
for a function $g_n$ that is symmetric, i.e., invariant to any permutation of
its arguments. Note that if instead $f_n$ itself is symmetric (equivalently, if
the above holds with all constant weight functions $\lambda_i \equiv 1$), then
this reduces to the ordinary (unweighted) notion of exchangeability. The
definition of weighted exchangeability can be extended to infinite sequences in
the usual way: by requiring that the joint distribution of each finite
subsequence of random variables be weighted exchangeable as per the above. It
also extends beyond continuously-distributed, real-valued random variables. See
Definitions \ref{def:wtd_exch_finite} and \ref{def:wtd_exch} for the formal
details.

In light of what de Finetti's theorem teaches us about exchangeable
distributions, it is natural to ask whether infinite weighted exchangeable
distributions have an analogous property: to put it informally, can an infinite
weighted exchangeable sequence be represented as a mixture of infinite
weighted i.i.d.\ (i.e., weighted exchangeable and mutually independent)
sequences?  

It turns out that the weight functions $\lambda_1,\lambda_2,\dots$ play a
critical role in determining whether or not this is true, and a central question
in this paper is:\footnote{As we will describe later, this question was already
  answered by \citet{lauritzen1988extremal} for the binary case $\Xcal =
  \{0,1\}$.}   
\begin{itemize}
\item[Q1.] Which sequences of weight functions lead to a generalized de Finetti 
  representation?   
\end{itemize}  
Rather than trying to answer the above question on its own, we find it
interesting to embed it into a larger problem of comparing the answers to
\emph{three} questions---Q1, and the following two questions, on weighted
extensions of two other well-known results in the i.i.d.\ case, namely, the
zero-one law (Theorem \ref{thm:01}) and the law of large numbers (Theorem
\ref{thm:lln}):      
\begin{itemize}
\item[Q2.] Which sequences of weight functions lead to a generalized zero-one
  law? 
\item[Q3.] Which sequences of weight functions lead to a generalized law of
  large numbers?
\end{itemize} 
Our main result, in Theorem \ref{thm:main} below, relates questions Q1--Q3. In
short we show that ``answers to Q1'' (weight functions leading to a de
Finetti-type representation) are a subset of ``answers to Q2'', which are
themselves a subset of ``answers to Q3''. We also provide complementary 
necessary and sufficient conditions.  

\paragraph{General assumptions and notation.}

Throughout this paper, we will work in a space $\Xcal$ that we assume is a
separable complete metric space (this certainly includes, for example, any 
finite-dimensional Euclidean space, $\Xcal = \R^d$). We denote the Borel
$\sigma$-algebra of $\Xcal$ by $\BX$. Under these assumptions, we call $\Xcal$, 
equipped with its Borel $\sigma$-algebra $\BX$, a standard Borel
space.

We write $\Xcal^n = \Xcal \times \cdots \times \Xcal$ ($n$ fold) and
$\Xcal^\infty = \Xcal \times \Xcal \times \dots$ for the finite and countably 
infinite product spaces, respectively, and $\BXn$ and $\BXinf$ for their
respective Borel $\sigma$-algebras. We note that since $\Xcal$ is a standard
Borel space, it is second countable, which means it has a countable base. This
implies (among other important facts) that the countable product $\Xcal^\infty$
is also a standard Borel space, and $\BXinf$ is generated by the product of
individual Borel $\sigma$-algebras $\BX \times \BX \times \dots$ (e.g., see
Lemma 6.4.2 part (ii) in \citet{bogachev2007measure}). This same property is of
course also true for finite products.  

For a measure $Q$ on $\Xcal^\infty$ and any $n \geq 1$, we use $Q_n$ for the  
associated marginal measure on $\Xcal^n$, 
\[
Q_n(A) = Q(A \times \Xcal \times \Xcal \times \dots), \quad \textnormal{for
  $A \in \BXn$}. 
\]
We use \smash{$\Mcal_\Xcal$} to denote the set of measures on
$\Xcal$.\footnote{We allow measures on $\Xcal$ to be nonfinite, 
  i.e., a measure $Q \in \Mcal_\Xcal$ may have $Q(\Xcal) = \infty$.}  
This is itself a measure space, with a $\sigma$-algebra generated by sets of the
form \smash{$\{P \in \Mcal_\Xcal : P(A) \leq t\}$}, for $A\in\BX$ and $t \geq
0$. We use \smash{$\Pcal_\Xcal \subseteq \Mcal_\Xcal$} to denote the set of 
distributions on $\Xcal$, that is, \smash{$\Pcal_\Xcal = \{P \in \Mcal_\Xcal :
  P(\Xcal) = 1\}$}. For \smash{$P \in\Pcal_\Xcal$}, we write $P^n = P \times P
\times \cdots \times P$ ($n$ fold) and $P^\infty = P \times P \times \dots$ for
the finite and countably infinite product distributions on $\Xcal^n$ and
$\Xcal^\infty$, respectively; so that writing $(X_1,X_2,\dots) \sim P^\infty$ is
the same as writing \smash{$X_1,X_2,\dots \iidsim P$}, and similarly for
$P^n$. We use $\delta_x$ to denote the distribution defined by a point mass at
any given $x \in \Xcal$. Finally, we use \smash{$\aseq$} and \smash{$\asto$} to
denote almost sure equality and almost sure convergence, respectively.

\section{Exchangeability and weighted exchangeability}

Exchangeability is a property of a sequence of random variables that expresses
the notion, roughly speaking, that the sequence is ``equally likely to appear in
any order''. For a distribution on a finite sequence $(X_1,\dots,X_n) \in
\Xcal^n$, or more generally, for a measure on $\Xcal^n$, we can define this
property in terms of invariance to permutations, as follows.\footnote{Throughout 
  this paper, we will use the terms ``exchangeable'' and ``i.i.d.'' (and later
  on, ``weighted exchangeable'' and ``weighted i.i.d.'') to refer either to a
  measure $Q$ itself, or to a random sequence $X$ drawn from $Q$ (when $Q$ is a
  distribution), depending on the context.}  

\begin{definition}[Finite exchangeability]\label{def:exch_finite}
A measure $Q$ on $\Xcal^n$ is called \emph{exchangeable} provided that, for all 
$A_1,\dots,A_n \in \BX$,     
\[
Q(A_1 \times \dots \times A_n) = 
Q(A_{\sigma(1)} \times \dots \times A_{\sigma(n)}),
\quad \textnormal{for all $\sigma\in\Scal_n$}, 
\]
where $\Scal_n$ is the set of permutations on $[n] = \{1,\dots,n\}$. 
\end{definition}

This definition can be extended to a distribution on an infinite sequence
$(X_1,X_2,\dots) \in \Xcal^\infty$, or more generally, to a measure on  
$\Xcal^\infty$, as follows.  

\begin{definition}[Infinite exchangeability]\label{def:exch}
A measure $Q$ on $\Xcal^\infty$ is called \emph{exchangeable} if for all $n \geq
1$ the corresponding marginal measure $Q_n$ is exchangeable.     
\end{definition}

In this paper, we will study a weighted generalization of exchangeability.
We denote by \smash{$\Lambda = \Lambda_\Xcal$} the set of measurable functions 
from $\Xcal$ to $(0,\infty)$, with $\Lambda^n$ or $\Lambda^\infty$ denoting the   
finite or countably infinite product, respectively, of this space. First, we
define the notion of weighted exchangeability for finite sequences.   

\begin{definition}[Finite weighted exchangeability]\label{def:wtd_exch_finite} 
Given $\lambda = (\lambda_1,\dots,\lambda_n) \in \Lambda^n$, a measure $Q$ on
$\Xcal^n$ is called \emph{$\lambda$-weighted exchangeable} if the measure
\smash{$\widebar{Q}$} defined as  
\[
\widebar{Q}(A) = \int_A \frac{\d{Q}(x_1,\dots,x_n)}{\lambda_1(x_1) \cdots
  \lambda_n(x_n)}, \quad \textnormal{for $A\in\BXn$}
\]
is an exchangeable measure.
\end{definition}

Similar to the unweighted case, we can extend this definition to infinite
sequences.

\begin{definition}[Infinite weighted exchangeability]\label{def:wtd_exch}
Given $\lambda = (\lambda_1,\lambda_2,\dots) \in \Lambda^\infty$, a measure $Q$
on $\Xcal^\infty$ is called \emph{$\lambda$-weighted exchangeable} if for all $n
\geq 1$ the corresponding marginal measure $Q_n$ is 
$(\lambda_1,\dots,\lambda_n)$-weighted exchangeable.      
\end{definition}

\subsection{Background on exchangeable distributions}

This section provides background on (unweighted) exchangeability, and reviews key
properties. Readers familiar with these topics may wish to skip ahead to Section
\ref{sec:background_weighted}.
  
\subsubsection{Mixtures and i.i.d.\ sequences}

An important special case of exchangeability is given by an independent and
identically distributed (i.i.d.) process. For any distribution $P$ on $\Xcal$,
the product $P^n$ is a finitely exchangeable distribution on $\Xcal^n$, whereas
the countable product $P^\infty = P \times P \times \dots$ is an infinitely
exchangeable distribution on $\Xcal^\infty$. More generally, exchangeability is
always preserved under mixtures, as the following result recalls, which we state
without proof.
 
\begin{proposition}\label{prop:exch_mix}
Any mixture of exchangeable measures on $\Xcal^n$ (or $\Xcal^\infty$) is itself
an exchangeable measure on $\Xcal^n$ (or $\Xcal^\infty$). 
\end{proposition}

In particular, this means that any mixture of i.i.d.\ sequences has an
exchangeable distribution.  

\begin{corollary}\label{cor:mix_of_iid}
For any distribution $\mu$ on \smash{$\Pcal_\Xcal$}, the distribution $Q$ on
$\Xcal^n$ defined by
\[
\textnormal{sample $P \sim \mu$, then draw $X_1,\dots,X_n \iidsim P$} 
\]
is exchangeable. Similarly, the distribution $Q$ on $\Xcal^\infty$ defined by
\[
\textnormal{sample $P \sim \mu$, then draw $X_1,X_2,\dots \iidsim P$}
\]
is exchangeable.
\end{corollary}

For the infinite setting, we use \smash{$Q = (P^\infty )_\mu$} to denote this
mixture. That is, for a distribution $\mu$ on \smash{$\Pcal_\Xcal$}, the   
distribution \smash{$Q = (P^\infty )_\mu$} is the infinitely exchangeable
distribution defined by  
\[
Q(A) = \Ep{P \sim \mu}{(P^\infty)(A)}, \quad \textnormal{for $A \in \BXinf$}. 
\]
To unpack this further, note that we can also equivalently define this mixture
distribution via
\[
Q(A_1 \times A_2 \times \cdots) = \int \prod_{i=1}^\infty P(A_i) \, \d\mu(P),  
\quad \textnormal{for $A_1,A_2,\dots \in \BX$},
\]
as $\BXinf$ is generated by $\BX \times \BX \times \cdots$ (which holds because
$\Xcal$ is standard Borel).  

\subsubsection{De Finetti's theorem}\label{sec:deFin_background}

The well-known de Finetti theorem, sometimes called the de
Finetti--Hewitt--Savage theorem, gives a converse to Corollary
\ref{cor:mix_of_iid}, for the infinite setting.  

\begin{theorem}[De Finetti--Hewitt--Savage]\label{thm:definetti}
Let $Q$ be an exchangeable distribution on $\Xcal^\infty$. Then there exists a
distribution $\mu$ on $\Pcal_\Xcal$ such that $Q = (P^\infty)_\mu$.
Moreover, the distribution $\mu$ satisfying this equality is
  unique.
\end{theorem}

This result was initially proved by \citet{definetti1929funzione} for the
special case of binary sequences, where $\Xcal=\{0,1\}$. In this case the mixing
distribution $\mu$ can simply be viewed as a distribution on $p \in [0,1]$,
where $p$ gives the parameter for the Bernoulli distribution $P$ on $\Xcal$, and
then Theorem \ref{thm:definetti} has a particularly simple interpretation: given
any exchangeable distribution $Q$ on infinite binary sequences, there exists a
distribution $\mu$ on $[0,1]$ such that draws from $Q$ can be expressed as
\[
\textnormal{sample $p \sim \mu$, then draw $X_1,X_2,\dots \iidsim
  \mathrm{Bernoulli}(p)$},
\]
as described earlier in Section \ref{sec:intro}. After his pioneering 1929 work,
\citet{definetti1937prevision} extended this theorem to real-valued sequences,
where $\Xcal=\R$, which was also later established independently by
\citet{dynkin1953classes}. \citet{hewitt1955symmetric} generalized de Finetti's
result to a much more abstract setting, which covers what are called Baire
measurable random variables taking values in a compact Hausdorff space.
Further generalizations to broader classes of spaces
can be found in \citet{farrell1962representation, maitra1977integral}, among
others.
Quite recently, \citet{alam2020generalizing} generalized this to cover any
Hausdorff space, under the assumption that the common marginal distribution
$Q_1$ on $\Xcal$ is a Radon measure. In particular, as stated in Theorem 
\ref{thm:definetti}, we emphasize that de Finetti's theorem holds when $\Xcal$
is a standard Borel space---this can be seen a consequence of the general
topological result in \citet{hewitt1955symmetric}, as pointed out by
\citet{varadarajan1963groups}. See also the discussion after Theorem 1.4 in
\citet{alam2020generalizing}, or Theorem 2.1 in \citet{fritz2021finetti}. 

Notably, completeness---inherent to standard Borel spaces---is critical
here because without it de Finetti's theorem can fail. This was shown by
\citet{dubins1979exchangeable}, who constructed an infinite sequence of Borel
measurable exchangeable random variables in a separable metric space that cannot
be expressed as a mixture of i.i.d.\ processes. Furthermore, the assumption of 
infinite exchangeability is also critical, because the result can fail in the
finite setting. As a simple example, let $\Xcal = \{0,1\}$, $n=2$, and let $Q$ 
place equal mass on $(1,0)$ and on $(0,1)$ (and no mass on $(0,0)$ or
$(1,1)$). Then $Q$ is exchangeable, but clearly cannot be realized by mixing the
distributions of i.i.d.\ Bernoulli sequences. 

Lastly, we remark that the above discussion is by no means a comprehensive
treatment of the work on de Finetti's theorem, its extensions and applications,
or alternative proofs. For a broader perspective, see, for example,
\citet{aldous1985exchangeability, diaconis1987dozen, lauritzen1988extremal}.

\subsubsection{Special properties of i.i.d.\ sequences}

Before moving on to discuss our main results, we mention some additional
well-known results that apply to infinite i.i.d.\ sequences (but not to
infinitely exchangeable sequences more generally), so that we can compare to the 
weighted case later on. 

We begin by recalling the Hewitt--Savage zero-one law, which is also due to 
\citet{hewitt1955symmetric}. See also \citet{kingman1978uses}. 

\begin{theorem}[Hewitt--Savage zero-one law]\label{thm:01}
For any distribution \smash{$P \in \Pcal_\Xcal$}, it holds that
\[
P^\infty(A) \in\{0,1\}, \quad \textnormal{for all $A \in \Ecal_\infty$}, 
\]
where $\Ecal_\infty \subseteq \BXinf$ is the sub-$\sigma$-algebra of
exchangeable events. 
\end{theorem}

The sub-$\sigma$-algebra $\Ecal_\infty \subseteq \BXinf$ referenced in the
theorem is defined as  
\begin{multline*}
\Ecal_\infty= \Big\{ A\in\BXinf : \textnormal{for all $x \in A$, $n \geq 1$, and 
  $\sigma \in \Scal_n$}, \\ \textnormal{it holds that
  $(x_{\sigma(1)},\dots,x_{\sigma(n)},x_{n+1},x_{n+2},\dots) \in A$}\Big\}. 
\end{multline*}
This is often called the \emph{exchangeable $\sigma$-algebra}. To give a
concrete example, the event ``we never observe the value zero'' is in
$\Ecal_\infty$, since we can express it as \smash{$A = \{x \in \Xcal^\infty: 
  \sum_{i=1}^\infty \one{x_i = 0} =  0\}$}, and this satisfies the permutation 
invariance condition in the last display.  

Next, we recall the strong law of large numbers,\footnote{To be precise, what   
  we state in Theorem \ref{thm:lln} is actually simply the application of the
  law of large numbers to indicator variables, as this special case is most
  pertinent to our study.}  
due to \citet{kolmogorov1930sur} (whereas weaker versions date back earlier to 
Bernoulli, Chebyshev, Markov, Borel, and others). For simplicity, we will drop
the specifier ``strong'' henceforth, and simply refer to this as the ``law of
large numbers''.     

\begin{theorem}[Law of large numbers]\label{thm:lln}
For any distribution \smash{$P \in \Pcal_\Xcal$}, given \smash{$X_1, X_2,
  \dots \iidsim P$}, write
\[
\widehat{P}_n= \frac{1}{n}\sum_{i=1}^n \delta_{X_i}
\]
to denote the empirical distribution of the first $n$ random variables. Then it
holds that   
\[
\widehat{P}_n(A) \asto P(A), \quad \textnormal{as $n \to \infty$, for all $A \in
  \BX$}. 
\]
\end{theorem}

This is particularly interesting when combined with de Finetti's theorem (recall
Theorem \ref{thm:definetti}). For any exchangeable distribution $Q$ on
$\Xcal^\infty$, by de Finetti's theorem, we can express \smash{$Q =
  (P^\infty)_\mu$} for a distribution $\mu$ on $\Pcal_\Xcal$; then by the law of 
large numbers (which we apply after conditioning on the draw $P \sim \mu$), the
unknown (random) distribution $P$ can be recovered from the observed sequence $X
\in \Xcal^\infty$ by taking the limit of its empirical distribution.  

\subsection{Mixtures in the weighted case, and weighted i.i.d.\ sequences}
\label{sec:background_weighted}

The next result extends Proposition \ref{prop:exch_mix} to weighted exchangeable 
distributions. We omit its proof since the result follows immediately from the
definition of weighted exchangeability.  

\begin{proposition}\label{prop:wtd_exch_mix}
For any $\lambda\in\Lambda^n$ (or $\Lambda^\infty$), any mixture of
$\lambda$-weighted exchangeable distributions on $\Xcal^n$ (or $\Xcal^\infty$)
is itself a $\lambda$-weighted exchangeable distribution on $\Xcal^n$ (or
$\Xcal^\infty$).  
\end{proposition}

We next introduce some notation that will help us concisely represent certain
product distributions. For \smash{$P \in \Mcal_\Xcal$} and $\lambda \in
\Lambda$, we define the distribution \smash{$P \circ \lambda \in \Pcal_\Xcal$}
by 
\[
(P \circ \lambda)(A) = \frac{\int_A \lambda(x) \, \d{P}(x)}
{\int_\Xcal \lambda(x) \, \d{P}(x)}, \quad \textnormal{for $A \in \BX$}. 
\]
Here, we are effectively ``reweighting'' the measure $P$ according
  to the weight function $\lambda$.
Note that $P \circ \lambda$ is well-defined for all $P$ in the set
\smash{$\Mcal_\Xcal(\lambda) = \{ P\in\Mcal_\Xcal : 0 < \int_\Xcal \lambda(x) \,
  \d{P}(x) < \infty \}$}.  For finite or infinite sequences of weight functions,
we will overload this notation as follows: when $\lambda =
(\lambda_1,\dots,\lambda_n) \in \Lambda^n$, we use $P \circ \lambda$ to denote
the distribution on $\Xcal^n$ defined by   
\[
P\circ\lambda = (P\circ\lambda_1) \times \dots \times (P\circ\lambda_n),
\]
and when $\lambda = (\lambda_1,\lambda_2\dots) \in \Lambda^\infty$, we
use $P \circ \lambda$ to denote the distribution on $\Xcal^\infty$ defined by 
\[
P\circ\lambda = (P\circ\lambda_1) \times (P\circ\lambda_2) \times \dots,
\]
and either of these mixtures are well-defined for all $P$ in the set
(overloading notation once more) \smash{$\Mcal_\Xcal(\lambda) =   
  \{ P\in\Mcal_\Xcal : 0 < \int_\Xcal \lambda_i(x) \, \d{P}(x) < \infty, \;
  \textnormal{for all $i$} \}$}. 

With this notation in hand, it is worth noting that for a distribution $Q$ on
$\Xcal^n$ or $\Xcal^\infty$ and any $\lambda$ in $\Lambda^n$ or
$\Lambda^\infty$,   
\begin{multline}
\label{eq:weighted_iid}
\textnormal{$Q$ is both a product distribution and $\lambda$-weighted
  exchangeable $\iff$} \\ \textnormal{$Q = P \circ \lambda$ for some $P \in 
  \Mcal_\Xcal$}. 
\end{multline}
Each direction is simple to check using Definition \ref{def:wtd_exch_finite} or
\ref{def:wtd_exch}. We will use the term \emph{$\lambda$-weighted i.i.d.}\ to   
refer to any distribution $Q$ satisfying \eqref{eq:weighted_iid}. 

Finally, as in the original unweighted case, Proposition \ref{prop:wtd_exch_mix}
has a nice interpretation for weighted i.i.d.\ sequences.

\begin{corollary}\label{cor:wtd_mix_of_iid}
For any $\lambda\in\Lambda^n$ and distribution $\mu$ on
\smash{$\Mcal_\Xcal(\lambda)$}, the distribution $Q$ on $\Xcal^n$ 
defined by
\[
\textnormal{sample $P \sim \mu$, then draw $X_1,\dots,X_n \sim P \circ
  \lambda$}   
\]
is $\lambda$-weighted exchangeable. Similarly, for any
$\lambda\in\Lambda^\infty$ and distribution $\mu$ on
\smash{$\Mcal_\Xcal(\lambda)$}, the distribution $Q$ on $\Xcal^\infty$ defined  
by   
\[
\textnormal{sample $P \sim \mu$, then draw $X_1,X_2,\dots \sim P \circ
  \lambda$} 
\]
is $\lambda$-weighted exchangeable. 
\end{corollary}

For the infinite setting, we will use $Q =  (P\circ\lambda)_\mu$ to denote this  
mixture, that is, for any $\lambda\in\Lambda^\infty$ and any distribution $\mu$ 
on \smash{$\Mcal_\Xcal(\lambda)$}, the distribution \smash{$Q = ( P\circ\lambda  
  )_\mu$} is the $\lambda$-weighted exchangeable distribution given by   
\[
Q(A) = \Ep{P \sim \mu}{(P\circ\lambda)(A)}, \quad \textnormal{for $A \in
  \BXinf$}, 
\]
or equivalently, in less compact notation,
\[
Q(A_1 \times A_2\times \cdots) = \int \prod_{i=1}^\infty (P\circ\lambda_i)(A_i)
\, \d\mu(P), \quad \textnormal{for $A_1,A_2,\dots \in \BX$}.  
\]

\section{Main results}

We now present the central questions and main findings of this work.  In short,
we are interested in examining whether de Finetti's theorem for infinitely
exchangeable sequences can be generalized to the setting of infinite weighted
exchangeability. Along the way, we will also consider whether results for
i.i.d.\ sequences---namely, the Hewitt--Savage zero-one law and the law of large
numbers---can be extended to the weighted case, as well.

\subsection{Framework}

We now define the properties we wish to examine. These are all weighted 
generalizations of the properties in the ordinary unweighted setting, described
above.  

\paragraph{The weighted de Finetti property.} 

We say that a sequence $\lambda\in\Lambda^\infty$ satisfies the \emph{weighted
  de Finetti property} if, for any $\lambda$-weighted exchangeable distribution
$Q$ on $\Xcal^\infty$, there is a distribution $\mu$ on
\smash{$\Mcal_\Xcal(\lambda)$} such that $Q = (P\circ\lambda)_\mu$, or
equivalently, 
\[
Q(A) = \Ep{P \sim \mu}{(P\circ\lambda)(A)}, \quad \textnormal{for all $A \in \BXinf$}.  
\]
In other words, this says that for any $\lambda$-weighted exchangeable 
distribution $Q$, we can represent it by mixing the distributions of
$\lambda$-weighted i.i.d.\ sequences. For our main results that follow, it will
be useful to denote this condition compactly. To this end, we write:
\begin{multline*}
\Lambda_{\dF} =  \Big\{ \lambda\in\Lambda^\infty : 
\textnormal{for all $\lambda$-weighted exchangeable distributions $Q$ on 
  $\Xcal^\infty$}, \\
\textnormal{there exists a distribution $\mu$ on $\Mcal_{\Xcal}(\lambda)$ 
  such that $Q = (P\circ\lambda)_\mu$} \Big\}.
\end{multline*}

\paragraph{The weighted zero-one law.} 

We say that a sequence $\lambda\in\Lambda^\infty$ satisfies the \emph{weighted
  zero-one law} if, for any \smash{$P\in \Mcal_\Xcal(\lambda)$}, each event
$A\in\Ecal_\infty$ is assigned either probability zero or one by
$P\circ\lambda$. To represent this condition compactly, we write:
\[
\Lambda_{\01} = \Big\{ \lambda\in\Lambda^\infty : \textnormal{for all
  $P\in\Mcal_\Xcal(\lambda)$ and $A\in\Ecal_\infty$, it holds that 
  $(P\circ\lambda)(A) \in\{0,1\}$} \Big\}. 
\] 

\paragraph {The weighted law of large numbers.}

We say that a sequence $\lambda\in\Lambda^\infty$ satisfies the 
\emph{weighted law of large numbers} if, for any \smash{$P \in
  \Mcal_\Xcal(\lambda)$}, the following holds: for \smash{$(X_1,X_2,\dots) \sim 
  P\circ\lambda$}, 
\[
\widetilde{P}_{n,i}(A) \asto (P\circ\lambda_i)(A), \quad \textnormal{as $n \to 
  \infty$, for all $i \geq 1$ and $A\in\BX$},
\]
where \smash{$\widetilde{P}_{n,i}$} is a certain weighted empirical distribution
of $X_1,\dots,X_n$, defined by
\begin{multline}\label{eq:wtd_empirical}
\widetilde{P}_{n,i} = \sum_{j=1}^n (w_{n,i}(X_1,\dots,X_n))_j \cdot
\delta_{X_j}, \\ \textnormal{where} \;\, (w_{n,i}(x_1,\dots,x_n))_j = \frac{  
\sum_{\sigma\in\Scal_n : \sigma(i)=j} \prod_{k=1}^n \lambda_k(x_{\sigma(k)})}  
{\sum_{\sigma\in\Scal_n} \prod_{k=1}^n \lambda_k(x_{\sigma(k)})}, \;\,
j=1,\dots,n.  
\end{multline}
 (Later on, in Proposition \ref{prop:Xi_condition_on_Ecalm}, we will
  see that \smash{$\widetilde{P}_{n,i}$} determines the distribution of the
  random variable $X_i$, if we condition on observing the \emph{unordered}  
  collection of values $X_1,\dots,X_n$. To compare this to the unweighted case, 
  observe that we would just have \smash{$(w_{n,i}(X_1,\dots,X_n))_j \equiv
    \frac{1}{n}$} for each $j$, so in this case we can view $X_i$ as a uniform
  random draw from $X_1,\dots,X_n$ after conditioning on this list.)
To represent the above condition compactly, we write:
\begin{multline*}
\Lambda_{\LLN} = \Big\{ \lambda\in\Lambda^\infty : 
\textnormal{for all $P \in \Mcal_\Xcal(\lambda)$, $i \geq 1$, and $A \in \BX$,
  if $X\sim P\circ\lambda$}, \\ \textnormal{then $\widetilde{P}_{n,i}(A) \asto  
  (P\circ\lambda_i)(A)$, for $\widetilde{P}_{n,i}$ as defined in
  \eqref{eq:wtd_empirical}} \Big\}. 
\end{multline*}

\begin{example}\label{example:binary}

To build intuition for these various sets, here we pause to give a simple
example. Consider the binary setting, $\Xcal = \{0,1\}$, and define
$\lambda \in \Lambda^\infty$ as the sequence of functions 
\[
\lambda_i(0) = 1, \quad \lambda_i(1) = 2^{-i},  \quad i \geq 1. 
\]
A straightforward calculation verifies that this sequence does not belong to any
of the three sets \smash{$\Lambda_{\dF}, \Lambda_{\01}, \Lambda_{\LLN}$} defined
above. For instance, to see that \smash{$\lambda \not \in\Lambda_{\dF}$},
consider the following distribution $Q$ on $\Xcal^\infty$: defining
$e_i = (0,\dots,0,1,0,0,\dots) \in \Xcal^\infty$ as the sequence with a 1 in
position $i$ and 0s elsewhere, let  
\[
Q(\{e_i\}) = 2^{-i}, \quad i \geq 1,
\] 
and $Q(\{x\})=0$ for all other $x \in \Xcal^\infty$. We can verify that $Q$ is
$\lambda$-weighted exchangeable, but since the sequence $X = (X_1,X_2,\dots)$
must contain a single 1 almost surely under $Q$, we can see that $Q$ cannot be
written as a mixture of $\lambda$-weighted i.i.d.\ distributions.\end{example}

\subsection{Main theorems}

Our first main result in this paper establishes a connection between these three 
weighted properties. Its proof will be covered in Section \ref{sec:proofs}.   

\begin{theorem}[Embedding of conditions]\label{thm:main}
It holds that
\[
\Lambda_{\dF} \subseteq \Lambda_{\01} \subseteq \Lambda_{\LLN}.
\] 
\end{theorem}

In other words, the above theorem says that for any $\lambda\in\Lambda^\infty$,
the weighted de Finetti property implies the weighted zero-one law, which in turn
implies the weighted law of large numbers.  

Our next main results pertain to necessary and sufficient conditions for 
$\lambda\in\Lambda^\infty$ to lie in these sets. Their proofs are also covered
in Section \ref{sec:proofs}. 

\begin{theorem}[Necessary condition]\label{thm:nec_condition}
If \smash{$\lambda\in\Lambda_{\LLN}$}, then $\lambda$ satisfies 
\begin{multline}\label{eq:nec_condition}
\sum_{i=1}^\infty \min\{(P\circ\lambda_i)(A) ,
  (P\circ\lambda_i)(A^c)\} = \infty, \\ \textnormal{for all
  $P\in\Mcal_\Xcal(\lambda)$, $A\in\BX$ with $P(A),P(A^c)>0$}. 
\end{multline}
\end{theorem}

\begin{theorem}[Sufficient condition]\label{thm:suff_condition}
If $\lambda\in\Lambda^\infty$ satisfies
\begin{equation}\label{eq:suff_condition}
\sum_{i=1}^\infty \frac{\inf_{x\in\Xcal} \, \lambda_i(x)/\lambda_*(x)}  
{\sup_{x\in\Xcal} \, \lambda_i(x)/\lambda_*(x)} = \infty \quad 
\textnormal{for some $\lambda_* \in \Lambda$},
\end{equation} 
then \smash{$\lambda \in \Lambda_{\dF}$}.
\end{theorem}

To summarize, combining our main results, we have: 
\begin{equation}\label{eqn:five_sets_subseteq}
\{\textnormal{$\lambda$ satisfying \eqref{eq:suff_condition}}\} \subseteq
\Lambda_{\dF} \subseteq \Lambda_{\01} \subseteq \Lambda_{\LLN}
\subseteq \{\textnormal{$\lambda$ satisfying \eqref{eq:nec_condition}}\}.
\end{equation}
In what follows, we explore whether these set inclusions are strict, or whether
they are equalities. In particular, we will derive precise answers to these
questions for the special case where $\Xcal$ has finite cardinality; in the
infinite case, we will see that open questions remain.

\subsection{Special case: binary random variables}

Before treating the finite case in full generality, it is useful to consider the
case of binary random variables, where $\Xcal$ contains two elements, and we
can take $\Xcal = \{0,1\}$ without loss of generality. In this case, it turns
out that all inclusions above, in \eqref{eqn:five_sets_subseteq}, are
equalities. This was already established by \citet{lauritzen1988extremal}. Below
we show that this can be derived as a consequence of Theorem \ref{thm:main}.  

\begin{theorem}[{Adapted from Sections II.9.1, II.9.2 and Chapter II Theorem 
    4.4 of \citet{lauritzen1988extremal}}]\label{thm:binary_case}
For the binary case where $\Xcal = \{0,1\}$, the five sets defined above are
equal: 
\[
\{\textnormal{$\lambda$ satisfying \eqref{eq:suff_condition}}\} =
\Lambda_{\dF} = \Lambda_{\01} = \Lambda_{\LLN}
= \{\textnormal{$\lambda$ satisfying \eqref{eq:nec_condition}}\}.
\]
Moreover, the common necessary \eqref{eq:nec_condition} and sufficient
\eqref{eq:suff_condition} conditions can equivalently be expressed as:
\[
\sum_{i=1}^\infty \frac{\min\{\lambda_i(0), \lambda_i(1)\}}
{\max\{\lambda_i(0), \lambda_i(1)\}} = \infty. 
\]
\end{theorem}

\begin{proof} 
\citet{lauritzen1988extremal} establishes this result for $\Xcal = \{0,1\}$ 
through the lens of extremal families of distributions. Here we instead give a
simple proof by applying Theorems \ref{thm:main}, \ref{thm:nec_condition},
and \ref{thm:suff_condition}, which hold for a general standard Borel space 
$\Xcal$. It suffices to show that the necessary condition
\eqref{eq:nec_condition} implies the sufficient condition 
\eqref{eq:suff_condition} when $\Xcal = \{0,1\}$.  

Towards this end, let $\lambda$ satisfy \eqref{eq:nec_condition},
thus \smash{$\sum_{i=1}^\infty \min\{(P\circ\lambda_i)(A),
  (P\circ\lambda_i)(A^c)\} = \infty$} holds for all measures
$P\in\Mcal_\Xcal(\lambda)$ and $A\in\BX$ with $P(A),P(A^c)>0$. Fixing  
$P = \mathrm{Bernoulli}(0.5) \in \Mcal_\Xcal(\lambda)$, and $A=\{0\}$, we have
$P(A),P(A^c)>0$, so we can apply \eqref{eq:nec_condition}. We calculate 
\[
(P\circ\lambda_i)(\{0\}) = \frac{\lambda_i(0)}{\lambda_i(1) + \lambda_i(0)},
\quad 
(P\circ\lambda_i)(\{1\}) = \frac{\lambda_i(1)}{\lambda_i(1) + \lambda_i(0)}.  
\]
By \eqref{eq:nec_condition}, then,
\begin{multline*}
\infty
= \sum_{i=1}^\infty \min\{(P\circ\lambda_i)(A),
(P\circ\lambda_i)(A^c)\} 
= \sum_{i=1}^\infty \min\{(P\circ\lambda_i)(\{0\}),
(P\circ\lambda_i)(\{1\})\} \\ 
= \sum_{i=1}^\infty \frac{\min\{\lambda_i(0), \lambda_i(1)\}}
{\lambda_i(1) + \lambda_i(0) }
\leq \sum_{i=1}^\infty \frac{\min\{\lambda_i(0), \lambda_i(1)\}}
{\max\{\lambda_i(0) , \lambda_i(1) \}}    
= \sum_{i=1}^\infty \frac{\inf_{x\in\Xcal} \, \lambda_i(x)/\lambda_*(x)} 
{\sup_{x\in\Xcal} \, \lambda_i(x)/\lambda_*(x)},
\end{multline*} 
where in the last step we define the function $\lambda_* : \Xcal \to
(0,\infty)$ by $\lambda_*(x)\equiv 1$. This proves that, if the necessary
condition \eqref{eq:nec_condition} holds, then the sufficient condition 
\eqref{eq:suff_condition} also holds. Finally, the fact that the condition:
\smash{$\sum_{i=1}^\infty \min\{\lambda_i(0), \lambda_i(1)\} /
  \max\{\lambda_i(0), \lambda_i(1)\} = \infty$} is equivalent to the common
necessary and sufficient condition is a consequence of the last display.
\end{proof}

\subsection{Beyond the binary case}

Now we move beyond the binary case. First we will see that when $|\Xcal| \geq  
3$, it will no longer be the case that the five sets in
\eqref{eqn:five_sets_subseteq} are all equal: specifically, there is always a
gap between the necessary condition \eqref{eq:nec_condition} and the sufficient  
condition \eqref{eq:suff_condition}.  

\begin{proposition}\label{prop:Xcal_size_at_least_3}
If $|\Xcal|\geq 3$, then 
\[
\{\textnormal{$\lambda$ satisfying \eqref{eq:suff_condition}}\} 
\subsetneq \{\textnormal{$\lambda$ satisfying \eqref{eq:nec_condition}}\}, 
\]
that is, the necessary condition \eqref{eq:nec_condition} is strictly weaker
than the sufficient condition \eqref{eq:suff_condition}. 
\end{proposition}

This result is proved with a simple example: we choose a partition $\Xcal =
\Xcal_0\cup \Xcal_1\cup\Xcal_2$ for some nonempty
$\Xcal_0,\Xcal_1,\Xcal_2\in\BX$, and define 
\[
\lambda_i(x) = \begin{cases} 
e^{-i} & x\in \Xcal_{\textnormal{mod}(i,3)},\\ 
1 & x \not\in \Xcal_{\textnormal{mod}(i,3)}, 
\end{cases} \quad i \geq 1.
\]
The full proof is given in Appendix \ref{sec:proof_prop:Xcal_size_at_least_3}. 

Proposition \ref{prop:Xcal_size_at_least_3} implies that, in the sequence of
four set inclusions in \eqref{eqn:five_sets_subseteq}, at least one of these set 
inclusions must be strict whenever $|\Xcal|\geq 3$---but the result does not
specify exactly where this gap might occur. The following theorem resolves this
question for the finite case, where $\Xcal$ has finite cardinality. Its proof is
given in Appendix \ref{sec:proof_thm:finite_case}.   

\begin{theorem}\label{thm:finite_case}
If $|\Xcal|<\infty$, then
\[
\Lambda_{\dF} = \Lambda_{\01} = \Lambda_{\LLN} 
= \{\textnormal{$\lambda$ satisfying \eqref{eq:nec_condition}}\},
\]
that is, the condition \eqref{eq:nec_condition} is in fact both necessary and
sufficient for the sets \smash{$\Lambda_{\dF}, \Lambda_{\01}, \Lambda_{\LLN}$}.
\end{theorem}

For the infinite case, however, no such result is known at present. Combining
all of our results so far, we can now summarize these different regimes as
follows. 

\begin{itemize}
\item \textbf{Binary case.} If $|\Xcal| = 2$, then by
  Theorem \ref{thm:binary_case},  
\[
\{\textnormal{$\lambda$ satisfying \eqref{eq:suff_condition}}\} =
\Lambda_{\dF} = \Lambda_{\01} = \Lambda_{\LLN}
= \{\textnormal{$\lambda$ satisfying \eqref{eq:nec_condition}}\}.
\]
(The same holds trivially for the singleton case, $|\Xcal|=1$.)

\item \textbf{Finite case.} If $2 < |\Xcal| < \infty$, then by Proposition
  \ref{prop:Xcal_size_at_least_3} and Theorem \ref{thm:finite_case}, 
\[
\{\textnormal{$\lambda$ satisfying \eqref{eq:suff_condition}}\} \subsetneq 
\Lambda_{\dF} = \Lambda_{\01} = \Lambda_{\LLN}
= \{\textnormal{$\lambda$ satisfying \eqref{eq:nec_condition}}\}.
\]
That is, condition \eqref{eq:nec_condition} is both necessary and sufficient,
and condition \eqref{eq:suff_condition} is strictly stronger than needed. 

\item \textbf{Infinite case.} If $|\Xcal| = \infty$, then by Proposition
  \ref{prop:Xcal_size_at_least_3}, 
\[
\{\textnormal{$\lambda$ satisfying \eqref{eq:suff_condition}}\} 
\subsetneq \{\textnormal{$\lambda$ satisfying \eqref{eq:nec_condition}}\}.
\]
However, for the infinite case, it is currently an open question to determine
which of the four set inclusions in \eqref{eqn:five_sets_subseteq} are strict. 
\end{itemize}

\section{Discussion}

This work studies and generalizes de Finetti's theorem through the lens of what
we call weighted exchangeability. Our main result shows that if a sequence of  
weight functions $\lambda$ satisfies a weighted de Finetti representation, then
this implies $\lambda$ also satisfies a weighted zero-one law, which in turn
implies $\lambda$ satisfies a weighted law of large numbers. We also present
more explicit sufficient (for the weighted de Finetti theorem representation)
and necessary (for the weighted law of large numbers) conditions.  
 After the initial version of our paper appeared online,
  \citet{tang2023finite} built on our work to derive interesting, approximate de
  Finetti representations for \emph{finite} weighted exchangeable sequences 
  $X_1,\dots,X_n$ (analogous to well-known results by \citet{diaconis1977finite,
    diaconis1980finite} for finite unweighted exchangeable sequences).

A potentially important contribution of this work, which we have not yet
discussed at this point in the paper, lies in the proof of the sufficient
condition---in Section \ref{sec:proof_suff} below. There we establish that, if  
the sufficient condition holds, one can start with an infinite weighted
exchangeable sequence $X$ and construct an \emph{exchangeable} infinite
subsequence \smash{$\check{X}$} by a careful rejection sampling scheme
(described in Section \ref{sec:construct_Xcheck}), whose limiting
\emph{unweighted} empirical distribution matches a weighted empirical
distribution of the original sequence. This is a key to our proof of the
sufficient condition, and may be useful in other problems in which weighted
exchangeability arises.  

We finish by discussing the connection between weighted exchangeability and the
literature on distribution-free statistical inference. Conformal prediction is a
general framework for quantifying uncertainty in the predictions made by
arbitrary prediction algorithms. It does so by acting as a wrapper method,
converting the predictions made by an algorithm into prediction sets which have    
distribution-free, finite-sample coverage properties; see, e.g., 
\citet{vovk2005algorithmic, lei2018distribution} for background. Importantly,
the coverage guarantees for conformal prediction rely on the assumption that 
all of the data---the training samples (fed into the algorithm, to fit the
predictive model) and test sample (at which we want to a form prediction
set)---are exchangeable. 

In previous work \citep{tibshirani2019conformal}, we extended the conformal
prediction framework to a setting where the data is not exchangeable, but
instead weighted exchangeable (precisely as defined in the current paper). This
framework allows conformal prediction to be applied in problems with covariate
shift (where the training and test covariate distributions differ); moreover, it
can be used as a basis for developing new conformal methods in various settings,
such as label shift \citep{podkopaev2021distribution}, causal inference
\citep{lei2021conformal}, experimental design \citep{fannjiang2022conformal},
and survival analysis \citep{candes2023conformalized}. In each of these
examples, the data set in hand is finite (and assumed to be weighted
exchangeable). In contrast, the current paper considers an infinite weighted exchangeable 
sequence; the characterizations that we developed for such sequences (as
mixtures of weighted i.i.d.\ sequences) may be useful for understanding online
(streaming data) problems in nonexchangeable conformal prediction. Developing
such connections, as well as broadly pursuing applications of our
theorems to other problems in statistical modeling and inference, will be topics
for future work.

\subsection*{Acknowledgements} 

We are grateful to Vladimir Vovk for suggesting that we study a de
Finetti-type representation in the weighted setting, and for introducing us to
the results of \citet{lauritzen1988extremal}. We are also grateful to Isaac
Gibbs for helpful feedback on an earlier draft of this manuscript.

We thank the American Institute of Mathematics for supporting and hosting our
collaboration. R.F.B.\ was supported by the National Science Foundation via
grants DMS-1654076 and DMS-2023109, and by the Office of Naval Research via
grant N00014-20-1-2337.  E.J.C.\ was supported by the Office of Naval Research
grant N00014-20-1-2157, the National Science Foundation grant DMS-2032014, the
Simons Foundation under award 814641, and the ARO grant 2003514594.

\section{Proofs of main results}\label{sec:proofs}

\subsection{Important properties}

Before proving our main results, we state a number of basic but important
properties of exchangeability and weighted exchangeability. These results will
not only be helpful later on in our proofs; they will also help build
intuition about (weighted) exchangeability, and some could be of potential
interest in their own right.     

\subsubsection{Equivalent characterizations} 

In the unweighted case, checking exchangeability of a distribution can be
reduced to checking that the joint distribution of $X_1, X_2, \ldots$ is
unchanged by swapping any pair of random variables. For concreteness, we record
this fact in the next result. For a finite sequence or infinite sequence $x$, we
will write \smash{$x^{ij}$} to denote $x$ but with \smash{$i\th$} and
\smash{$j\th$} entries swapped.     

\begin{proposition}\label{prop:exch_equiv}
For any distribution $Q$ on $\Xcal^n$ (or on $\Xcal^\infty$), the following 
statements are equivalent:  
\begin{enumerate}
\item[(a)] $Q$ is exchangeable.
\item[(b)] For all $1 \leq i < j \leq n$ (or all $1 \leq i < j$), and all
  measurable functions $f: \Xcal^n \to [0,\infty)$
  (or all measurable functions $f: \Xcal^\infty \to [0,\infty)$), 
  \[
  \Ep{Q}{f(X)} = \Ep{Q}{f(X^{ij})}.
  \]
\end{enumerate}
\end{proposition}

An analogous result holds in the weighted case: checking $\lambda$-weighted
exchangeability of $X \sim Q$ is equivalent to considering \emph{weighted} swaps
of the entries of $X$. 

\begin{proposition}\label{prop:wtd_exch_equiv}
Fix $\lambda \in \Lambda^n$ (or $\lambda \in \Lambda^\infty$). For any
distribution $Q$ on $\Xcal^n$ (or on $\Xcal^\infty$), the following statements 
are equivalent: 
\begin{enumerate}
\item[(a)] $Q$ is $\lambda$-weighted exchangeable.
\item[(b)] For all $1 \leq i < j \leq n$ (or all $1\leq i <j$), and all
  measurable functions $f: \Xcal^n \to [0,\infty)$ 
  (or all measurable functions $f: \Xcal^\infty \to [0,\infty)$),
  \[
  \Ep{Q}{\frac{f(X)}{\lambda_i(X_i)\lambda_j(X_j)}} =
  \Ep{Q}{\frac{f(X^{ij})}{\lambda_i(X_i)\lambda_j(X_j)}}.
  \]
\end{enumerate}
\end{proposition}

We remark that an analogous result holds more generally when $Q$ is a
$\lambda$-weighted exchangeable measure, but for simplicity we state the result
in Proposition \ref{prop:wtd_exch_equiv} only for distributions. The proof of 
Proposition \ref{prop:wtd_exch_equiv} (which generalizes Proposition
\ref{prop:exch_equiv}) is deferred to Appendix \ref{app:wtd_exch_equiv}.  

\subsubsection{Conditional distributions} 

For a sequence $X$, we denote by \smash{$\widehat{P}_m$} the  empirical
distribution of the first $m$ terms in the sequence,
\[
\widehat{P}_m(A) 
= \frac{1}{m}\sum_{i=1}^m \one{X_i\in A}, \quad \textnormal{for $A \in \BX$}. 
\] 
Define $\Ecal_m$ to be the sub-$\sigma$-algebra containing events that are
symmetric in the first $m$ coordinates, that is, if $B\in\Ecal_m$ then it must 
hold that 
\[
x\in B \iff x^{ij}\in B \quad \textnormal{for all $i \neq j \in [m]$}.
\]
We think of $\Ecal_m$ as a sub-$\sigma$-algebra of either $\BXn$ or $\BXinf$, 
depending on if we are working in the finite or infinite setting; this will be
clear from context. Equivalently, we can define $\Ecal_m$ as the
$\sigma$-algebra generated by \smash{$\widehat{P}_m, X_{m+1},\dots,X_n$} in the  
finite setting, or by \smash{$\widehat{P}_m, X_{m+1},X_{m+2},\dots$} in the
infinite setting. 

The following result on conditional distributions holds for the exchangeable
case.   

\begin{proposition}\label{prop:Xi_condition_on_Ecalm_exch} 
For any exchangeable distribution $Q$ on $\Xcal^n$ (or on $\Xcal^\infty$), and
$X \sim Q$, it holds for any $1 \leq i \leq m \leq n$ (or for any $1 \leq i \leq
m$) that
\[
X_i \mid \Ecal_m \sim \widehat{P}_m.
\]
\end{proposition}

Note that the above conditional law does not depend on $Q$. In other words, the 
distribution of $X_i \mid \Ecal_m$ is the same---it is simply the empirical
distribution on the first $m$ coordinates---for \emph{any} exchangeable
distribution $Q$.  

The next result characterizes the analogous conditional distribution for the  
terms in a weighted exchangeable sequence.   

\begin{proposition}\label{prop:Xi_condition_on_Ecalm}
For any $\lambda\in\Lambda^n$ (or $\lambda\in\Lambda^\infty$), any
$\lambda$-exchangeable distribution $Q$ on $\Xcal^n$ (or on $\Xcal^\infty$), and  
$X \sim Q$, it holds for any $1 \leq i \leq m \leq n$ (or for any $1 \leq i \leq
m$) that
\[
X_i \mid \Ecal_m \sim  \widetilde{P}_{m,i},
\]
where \smash{$\widetilde{P}_{m,i} =  \sum_{j=1}^m (w_{m,i}(X_1,\dots,X_m))_j
  \cdot \delta_{X_j}$} is the weighted empirical distribution defined in
\eqref{eq:wtd_empirical}.  
\end{proposition}

In other words, each $X_i \mid \Ecal_m$ can be viewed as a draw from a 
\emph{weighted} empirical distribution of $X_1,\dots,X_m$. The weights
\smash{$(w_{m,i}(X_1,\dots,X_m))_j$} that define this weighted empirical
distribution, as constructed in \eqref{eq:wtd_empirical}, are determined by the
functions $\lambda_1,\dots,\lambda_m$ but do not otherwise depend on the
original distribution $Q$. This is crucial, and is analogous to the unweighted
case in Proposition \ref{prop:Xi_condition_on_Ecalm_exch}.  The proof of
Proposition \ref{prop:Xi_condition_on_Ecalm} (which generalizes Proposition 
\ref{prop:Xi_condition_on_Ecalm_exch}) is deferred to Appendix
\ref{app:Xi_condition_on_Ecalm}.

\subsection{Proof of Theorem \ref{thm:main}}
\label{sec:proof_main}

We turn to the proof of our first main result, Theorem \ref{thm:main}.

\subsubsection{Proof of $\Lambda_{\01} \subseteq \Lambda_{\LLN}$}
\label{sec:proof_main1}

Let \smash{$\lambda \in \Lambda_{\01}$}, and draw $X \sim P \circ \lambda$ for 
some \smash{$P \in \Mcal_\Xcal(\lambda)$}. Fix $i \geq 1$ and $A \in
\BX$. Define weights \smash{$(w_{n,i}(x_1,\dots,x_n))_j$}, $j \in [n]$ as in
\eqref{eq:wtd_empirical}. First, by Proposition
\ref{prop:Xi_condition_on_Ecalm}, as $P\circ\lambda$ is $\lambda$-weighted
exchangeable,      
\[
\Ppst{P\circ\lambda}{X_i\in A}{\Ecal_n} \aseq
\sum_{j=1}^n (w_{n,i}(X_1,\dots,X_n))_j \cdot \one{X_j\in A} = 
\widetilde{P}_{n,i}(A).
\]
 (Note that \smash{$\widetilde{P}_{n,i}(A)$} is $\Ecal_n$-measurable
  by  definition.)
Next, fixing an arbitrary set $A\in\BX$, since $\Ecal_1 \supseteq \Ecal_2
\supseteq \dots$ with \smash{$\Ecal_\infty = \cap_{n=1}^\infty \Ecal_n$}, by
Levy's Downwards Theorem (e.g., Chapter 14.4 of
\citet{williams1991probability}), we have   
\[ 
\lim_{n\to\infty} \Ppst{P\circ\lambda}{X_i\in A}{\Ecal_n} \aseq 
\Ppst{P\circ\lambda}{X_i\in A}{\Ecal_\infty}, 
\]
or in other words,
\begin{equation}\label{eq:levy_downwards}
\lim_{n\to\infty} \widetilde{P}_{n,i}(A) \aseq 
\Ppst{P\circ\lambda}{X_i\in A}{\Ecal_\infty}.
\end{equation}
Let $Y_A$ be an $\Ecal_\infty$-measurable random variable such that \smash{$Y_A
  \aseq \Ppst{P\circ\lambda}{X_i\in A}{\Ecal_\infty}$}; the existence of such a
random variable is ensured  \citep{faden1985existence} since $\Xcal$, and hence 
$\Xcal^\infty$, is standard Borel. We thus have \smash{$\lim_{n\to\infty}
\widetilde{P}_{n,i}(A) \aseq Y_A$}. Now, recalling that we have assumed
\smash{$\lambda\in\Lambda_{\01}$}, we have   
\[
\Pp{P\circ\lambda}{Y_A \leq (P\circ\lambda_i)(A)} \in \{0,1\}. 
\] 
But by construction, \smash{$\Ep{P\circ\lambda}{Y_A} = (P\circ\lambda_i)(A)$},
so we conclude from the above display that in fact $\Pp{P\circ\lambda}{Y_A \leq
  (P\circ\lambda_i)(A)} = 1$. A similar argument leads to
$\Pp{P\circ\lambda}{Y_A \geq (P\circ\lambda_i)(A)} = 1$, and thus combining
these conclusions, we have shown that \smash{$Y_A \aseq (P\circ\lambda_i)(A)$}
and  
\[
\lim_{n\to\infty} \widetilde{P}_{n,i}(A) \aseq (P\circ\lambda_i)(A),
\]
as desired.
 
\subsubsection{Proof of $\Lambda_{\dF} \subseteq \Lambda_{\01}$}  
\label{sec:proof_main2}

Let \smash{$\lambda \in \Lambda_{\dF}$} and suppose for the sake of
contradiction that \smash{$\lambda \notin \Lambda_{\01}$}, so there is some
\smash{$P_* \in \Mcal_\Xcal(\lambda)$} and some $B_0 \in \Ecal_\infty$ such that
\smash{$p = (P_* \circ \lambda)(B_0) \in (0,1)$}. Now let $Q_0$ and $Q_1$ denote
the distribution of $X=(X_1,X_2,\dots)$ conditional on the event $B_0$ and on
the event $B_1 = B_0^c$, respectively, for \smash{$X \sim P_*\circ\lambda$}. 
Note that \smash{$P_*\circ\lambda$} can be written as a mixture,
\[
P_*\circ\lambda = p\cdot Q_0 +  (1-p)\cdot Q_1.
\]
Next we verify that $Q_\ell$ is $\lambda$-weighted exchangeable for each
$\ell=0,1$. Proposition \ref{prop:wtd_exch_equiv} tells us that to verify
$Q_\ell$ is $\lambda$-weighted exchangeable, we must only verify that
\[
\Ep{Q_\ell}{\frac{f(X)}{\lambda_i(X_i)\lambda_j(X_j)}} =
\Ep{Q_\ell}{\frac{f(X^{ij})}{\lambda_i(X_i)\lambda_j(X_j)}},
\]
for all $i \neq j$ and all measurable $f:\Xcal^\infty\to[0,\infty)$. But
since $Q_\ell$ is equal to \smash{$P_*\circ\lambda$} conditional on the event
$B_\ell$, and \smash{$(P_*\circ\lambda)(B_\ell)>0$}, it is equivalent to verify
that  
\[
\Ep{P_*\circ\lambda}{\frac{f(X)}
{\lambda_i(X_i)\lambda_j(X_j)} \cdot \one{X\in B_\ell}} =  
\Ep{P_*\circ\lambda}{\frac{f(X^{ij})}
{\lambda_i(X_i)\lambda_j(X_j)}\cdot\one{X\in B_\ell}}.
\]
Note that \smash{$X \in B_\ell$ if and only if $ X^{ij} \in B_\ell$}, because
$B_\ell\in\Ecal_\infty$, so it is equivalent to check that   
\[
\Ep{P_*\circ\lambda}{\frac{f(X)\cdot \one{X\in B_\ell}}
{\lambda_i(X_i)\lambda_j(X_j)}} = 
\Ep{P_*\circ\lambda}{\frac{f(X^{ij})\cdot\one{X^{ij}\in B_\ell}}
{\lambda_i(X_i)\lambda_j(X_j)}}.
\]
Lastly, another application of Proposition \ref{prop:wtd_exch_equiv}
 (with $f(X)\cdot \one{X\in B_\ell}$ in place of $f(X)$) tells us  
that the above display holds, since \smash{$P_*\circ\lambda$} itself is
$\lambda$-weighted exchangeable. 

Recalling that we have assumed \smash{$\lambda\in\Lambda_{\dF}$}, and having
just established that each $Q_\ell$ is $\lambda$-weighted exchangeable, we know
that there is a distribution $\mu_\ell$ on \smash{$\Mcal_\Xcal(\lambda)$} such
that \smash{$Q_\ell = (P\circ\lambda)_{\mu_\ell}$}, for each $\ell=0,1$. In
particular, writing the mixture of $\mu_0$ and $\mu_1$ as
\[
\mu = p \cdot \mu_0 + (1-p) \cdot \mu_1,
\]
we see that \smash{$P_* \circ \lambda = (P\circ\lambda)_\mu$}. We now need an 
additional lemma, whose proof is in Appendix \ref{app:mix_unique}.

\begin{lemma}\label{lem:mix_unique} 
Fix $\lambda\in\Lambda^\infty$ and let $\mu$ be a distribution on
\smash{$\Mcal_\Xcal(\lambda)$}. Suppose that \smash{$(P\circ\lambda)_\mu =
  P_*\circ\lambda$} for some \smash{$P_*\in\Mcal_\Xcal(\lambda)$}.  Then, under
$P\sim \mu$, for any $A\in\BXinf$ it holds that \smash{$(P\circ\lambda)(A)\aseq  
  (P_*\circ\lambda)(A)$}.
\end{lemma}

Therefore we know \smash{$(P_*\circ\lambda)(B_0) = (P\circ\lambda)(B_0)$} holds 
almost surely under $P\sim\mu$. We have now reached our desired contradiction,
because \smash{$(P_*\circ\lambda)(B_0)=p$} is nonrandom, whereas
$(P\circ\lambda)(B_0)$ is distributed as $\mathrm{Bernoulli}(p)$ under
$P\sim\mu$, by construction of $\mu = p \cdot \mu_0 + (1-p) \cdot \mu_1$.   

\subsection{Proof of Theorem \ref{thm:nec_condition}}
\label{sec:proof_nec}

Let \smash{$\lambda \in \Lambda_{\LLN}$}, and fix any \smash{$P_0 \in
  \Mcal_\Xcal(\lambda)$} and any $A \in \BX$ with $P_0(A) \in (0,1)$. We will
prove that the necessary condition \eqref{eq:nec_condition} must hold for
$P=P_0$. Fix any $c \in (0,1)$ and define a measure $P_1$ via  
\[
P_1(B) = c \cdot P_0(A\cap B) + P_0(A^c\cap B), \quad \textnormal{for $B \in
  \BX$}.   
\]
Observe that \smash{$P_1 \in \Mcal_\Xcal(\lambda)$} by construction. As 
\smash{$\lambda \in \Lambda_{\LLN}$}, by assumption we have
\[
\Pp{P_\ell\circ\lambda}{\lim_{n\to\infty} \sum_{j=1}^n 
(w_{n,1}(X_1,\dots,X_n))_j \cdot \one{X_j\in A} = 
(P_\ell\circ\lambda_1)(A)} = 1, 
\]
for each $\ell=0,1$, where recall \smash{$(w_{n,1}(X_1,\dots,X_n))_j $} is
defined as in \eqref{eq:wtd_empirical}. Next we let $E$ be the event that  
\smash{$\lim_{n\to\infty}\sum_{j=1}^n (w_{n,1}(X_1,\dots,X_n))_j \cdot
  \one{X_j\in A} = (P_0\circ\lambda_1)(A)$}. As $(P_0\circ\lambda_1)(A)
\neq (P_1\circ\lambda_1)(A)$ by construction, this implies 
\[
(P_0\circ\lambda)(E) = 1, \quad \textnormal{and} \quad (P_1\circ\lambda)(E) = 0,   
\]
and thus \smash{$\dtv(P_0\circ\lambda,P_1\circ\lambda)=1$}, where \smash{$\dtv$}
denotes total variation distance.

Next, for any $i \geq1$ and any $B\in\BX$, a straightforward calculation shows
that 
\[
(P_1\circ\lambda_i)(B) = \frac{
c\cdot (P_0\circ\lambda_i)(A\cap B) +
(P_0\circ\lambda_i)(A^c\cap B)}
{c\cdot (P_0\circ\lambda_i)(A) +
(P_0\circ\lambda_i)(A^c)},
\]
and thus
\begin{align*}
\dtv(P_0\circ\lambda_i, P_1\circ\lambda_i) 
&= \sup_{B\in\BX} \, |(P_0\circ\lambda_i)(B) - (P_1\circ\lambda_i)(B)| \\ 
&= (1-c) \cdot \frac{ (P_0\circ\lambda_i)(A) \cdot (P_0\circ\lambda_i)(A^c)}
{c\cdot (P_0\circ\lambda_i)(A) +(P_0\circ\lambda_i)(A^c)},
\end{align*}
where the last equality holds because the supremum is attained at $B=A$ (or
equivalently, at $B=A^c$). From this we can verify that
\smash{$\dtv(P_0\circ\lambda_i,P_1\circ\lambda_i)<1$}, and moreover,   
\begin{equation}\label{eq:dtv_bound}
\dtv(P_0\circ\lambda_i,P_1\circ\lambda_i)\leq (c^{-1}-1) \cdot
\min\{(P_0\circ\lambda_i)(A), (P_0\circ\lambda_i)(A^c)\}.
\end{equation}
We also know that
\[
0 = 1 - \dtv(P_0\circ\lambda,P_1\circ\lambda) = 
\prod_{i=1}^\infty (1 - \dtv(P_0\circ\lambda_i,P_1\circ\lambda_i)),
\] 
where the second equality holds by properties of the total variation
distance for product distributions. In general, for any sequence $a_1,a_2,\dots
\in [0,1)$ with finite sum \smash{$\sum_{i=1}^\infty a_i <\infty$}, it holds 
that \smash{$\prod_{i=1}^\infty (1-a_i) > 0$} (see, e.g., Theorem 4 in Section
28 of \citet{knopp1990theory}). Therefore, we must have 
\[
\sum_{i=1}^\infty \dtv(P_0\circ\lambda_i,P_1\circ\lambda_i) = \infty,
\]
and so combining this with \eqref{eq:dtv_bound}, we get
\[
\infty = \sum_{i=1}^\infty \dtv(P_0\circ\lambda_i,P_1\circ\lambda_i) 
\leq (c^{-1}-1)\cdot \sum_{i=1}^\infty \min\{(P_0\circ\lambda_i)(A) ,  
(P_0\circ\lambda_i)(A^c)\}.
\]
Since $(c^{-1}-1)$ is finite, the sum on the right-hand side must be infinite,
which completes the proof.

\subsection{Proof of Theorem \ref{thm:suff_condition}}
\label{sec:proof_suff}

\subsubsection{Outline of proof}

We begin by sketching the main ideas of the proof in order to build
intuition. In the exchangeable case, where $X\sim Q$ for an exchangeable 
distribution $Q$, de Finetti's theorem tells us that we can view $X_1,X_2,\dots$ 
as i.i.d.\ draws from some random distribution $P$ (that is, for some $P \sim 
\mu$) and we can recover this underlying distribution $P$ almost surely via the
limit \smash{$\lim_{n\to\infty} \widehat{P}_n$}, where \smash{$\widehat{P}_n$}
is the empirical distribution of $X_1,\dots,X_n$:
\[
P(A) \aseq \lim_{n\to\infty} \widehat{P}_n(A), \quad \textnormal{for all
  $A\in\BX$}.  
\]
In the weighted exchangeable case, where $X\sim Q$ for a $\lambda$-weighted  
exchangeable $Q$, we will work towards a similar conclusion. To recover the
underlying random distribution, we will take the limit of a \emph{weighted}
empirical distribution of $X_1,\dots,X_n$, denoted by \smash{$\widetilde{P}_n$},
and defined by 
\begin{equation}\label{eq:Ptilde}
\widetilde{P}_n(A) = \frac{
\sum_{i=1}^n\frac{\inf_{x\in\Xcal}\lambda_i(x)/\lambda_*(x)}
{\lambda_i(X_i)/\lambda_*(X_i)} \cdot
\one{X_i\in A}}{\sum_{i=1}^n \frac{\inf_{x\in\Xcal}\lambda_i(x)/\lambda_*(x)}
{\lambda_i(X_i)/\lambda_*(X_i)}}, \quad \textnormal{for $A\in\BX$},
\end{equation}
with $\lambda_*\in\Lambda$ chosen such that the sufficient condition
\eqref{eq:suff_condition} is satisfied. Note that by \eqref{eq:suff_condition},
for $n$ large enough the denominator in \eqref{eq:Ptilde} is positive, hence
\smash{$\widetilde{P}_n$} is well-defined. (Also, to avoid confusion, we note  
that \smash{$\widetilde{P}_n$} is different from the weighted empirical
distribution \smash{$\widetilde{P}_{n,i}$} in \eqref{eq:wtd_empirical} that is
used to define the weighted law of large numbers.)  

In the rest of the proof of Theorem \ref{thm:suff_condition}, we will first
construct a random distribution \smash{$\widetilde{P}$} that is an almost sure
limit of the weighted empirical distribution \smash{$\widetilde{P}_n$} in
\eqref{eq:Ptilde}. Then we will show that the given weighted exchangeable
distribution $Q$ is equal to the mixture distribution
\smash{$(P\circ\lambda)_\mu$}, where $\mu$ is the distribution of the random
measure \smash{$\widetilde{P}_*$} defined as
\begin{equation}\label{eqn:P_from_Ptilde}
\widetilde{P}_*(A) = \int_A \frac{\d\widetilde{P}(x)}{\lambda_*(x)}, \quad
\textnormal{for $A\in\BX$}. 
\end{equation}

\subsubsection{Constructing the limit $\widetilde{P}$ via an exchangeable
  subsequence}\label{sec:construct_Xcheck}

Our first task is to construct a distribution \smash{$\widetilde{P}$} that is an
almost sure limit of the weighted empirical distribution
\smash{$\widetilde{P}_n$} defined in \eqref{eq:Ptilde}. We begin by
approximating this weighted empirical distribution. Let \smash{$U_1,U_2,\dots
  \iidsim \mathrm{Unif}[0,1]$}, independently of $X$, and define for each
$i=1,2,\dots$,      
\begin{equation}\label{eq:Bi}
B_i = \one{U_i\leq p_i(X_i)}, \quad \textnormal{where} \quad 
p_i(x) = \frac{\inf_{x'\in\Xcal} \, \lambda_i(x')/\lambda_*(x')} 
{\lambda_i(x)/\lambda_*(x)} \in [0,1].
\end{equation}
We have
\[
\sum_{i=1}^\infty p_i(X_i) 
= \sum_{i=1}^\infty \frac{\inf_{x\in\Xcal} \, \lambda_i(x)/\lambda_*(x)}
{\lambda_i(X_i)/\lambda_*(X_i)}
\geq \sum_{i=1}^\infty\frac{\inf_{x\in\Xcal} \, \lambda_i(x)/\lambda_*(x)} 
{\sup_{x\in\Xcal} \, \lambda_i(x)/\lambda_*(x)} = \infty,
\]
where the last step holds by \eqref{eq:suff_condition}. Thus, defining
\[
M = \sum_{i=1}^\infty B_i,
\]
by the second Borel--Cantelli Lemma we see that $M = \infty$ almost surely. Now
define 
\begin{equation}\label{eq:Pbar_n_wtd}
\widebar{P}_n(A) = \frac{\sum_{i=1}^n B_i \cdot \one{X_i\in A}}{\sum_{i=1}^n 
  B_i}, \quad \textnormal{for $A \in \BX$},
\end{equation}
which, almost surely, is well-defined for sufficiently large $n$, because
\smash{$M \aseq \infty$}. This turns out to be a helpful approximation for the  
weighted empirical distribution in \eqref{eq:Ptilde}, with almost sure equality
in the limit, as we state next. The proof of this result is deferred to Appendix 
\ref{app:Pbar_equiv}. 

\begin{lemma}\label{lem:Pbar_equiv}
Under the notation and assumptions above, it holds that
\[
\limsup_{n\to\infty} \, \big|\widebar{P}_n(A)-\widetilde{P}_n(A)\big|
\aseq 0, \quad \textnormal{for all $A\in\BX$}. 
\]
\end{lemma}

Next we define an infinite subsequence \smash{$\check{X}$} of the original
sequence $X$. Define $I_0=0$, and 
\begin{equation}\label{eq:Im}
I_m = \min\{i > I_{m-1} : B_i = 1\}, \quad \textnormal{for $m=1,\dots,M$}.
\end{equation}
In other words, $I_1<I_2<\dots$ enumerates all indices $i$ for which $B_i=1$. 
Then define 
\begin{equation}\label{eq:Xcheck}
\check{X} = 
\begin{cases}
(X_{I_1},X_{I_2},X_{I_3},\dots) & \textnormal{if $M=\infty$}, \\
(X_{I_1},\dots,X_{I_M},x_0,x_0\dots) & \textnormal{otherwise},
\end{cases}
\end{equation}
where $x_0\in\Xcal$ is some fixed value. In other words, in the case $M<\infty$,
we augment the subsequence \smash{$(X_{I_m})_{m=1}^M$} with an infinite string 
$(x_0,x_0,\dots)$ (but of course, since $M=\infty$ almost surely, this occurs
with probability zero).  Critically, this construction results in
\smash{$\check{X}$} being exchangeable. The proof of the lemma below is   
deferred to Appendix \ref{app:exch_inf_subseq}.   

\begin{lemma}\label{lem:exch_inf_subseq}
Under the notation and assumptions above, the sequence \smash{$\check{X}$}
that is defined in \eqref{eq:Xcheck} has an exchangeable distribution. 
\end{lemma}

Why are we interested in \smash{$\check{X}$}? This will become more apparent
once we inspect the empirical distribution of its first $m$ elements: for $m \geq
1$, define 
\begin{equation}\label{eq:Pcheck_m}
\check{P}_m(A) = \frac{\sum_{i=1}^m\one{\check{X}_i \in A}}{m}, \quad
\textnormal{for $A \in \BX$}.
\end{equation}
Observe that, for any $n \geq 1$, denoting  \smash{$M_n=\sum_{i=1}^n B_i$}, if
$M_n>0$ then \smash{$\widebar{P}_n(A)= \check{P}_{M_n}(A)$}, where
\smash{$\widebar{P}_n$} is defined in \eqref{eq:Pbar_n_wtd} above.  
Combining this with Lemma \ref{lem:Pbar_equiv}, together with the fact that
\smash{$M=\lim_{n\to\infty} M_n$}, we see that
\begin{equation}\label{eq:Pcheck_limit}
\textnormal{if $M=\infty$ and $\lim_{m\to\infty}\check{P}_m(A) $
exists, then $\lim_{n\to\infty}\widetilde{P}_n(A)$ exists and is equal
to the same value} 
\end{equation}
holds almost surely, for all $A\in\BX$.

Exchangeability of the distribution of \smash{$\check{X}$} now allows us to
apply the original de Finetti--Hewitt--Savage theorem
(Theorem \ref{thm:definetti}) and the law of large numbers (Theorem
\ref{thm:lln}) to assert the existence of a random distribution
\smash{$\widetilde{P} \in \Pcal_\Xcal$} such that \smash{$\check{P}_m$} (and
thus also \smash{$\widetilde{P}_n$}) converges to \smash{$\widetilde{P}$}. To be
more specific, the de Finetti--Hewitt--Savage theorem tells us that we can
represent a draw from the distribution of \smash{$\check{X}$} as follows: we
first sample a random distribution \smash{$\widetilde{P}$}, then we sample
components \smash{$\check{X}_i$}, $i=1,2,\dots$ independently from 
\smash{$\widetilde{P}$}. The law of large numbers tells us that
\smash{$\widetilde{P}$} can be recovered by taking the limit of the empirical    
distribution \smash{$\check{P}_m$}. Thus, 
\[
\textnormal{there exists a random distribution $\widetilde{P}\in\Pcal_\Xcal$ 
such that $\lim_{m\to\infty}\check{P}_m(A) \aseq \widetilde{P}(A)$,
for all $A\in\BX$},  
\] 
and consequently, combining this with \eqref{eq:Pcheck_limit} along with the
fact that \smash{$M\aseq\infty$}, 
\begin{equation}\label{eqn:Ptilde_is_limit}
\textnormal{there exists a random distribution $\widetilde{P}\in\Pcal_\Xcal$
such that $\lim_{n\to\infty}\widetilde{P}_n(A) \aseq \widetilde{P}(A)$,
for all $A\in\BX$}. 
\end{equation}  
Next, for $n \geq 1$, we define $\Fcal_n = \sigma(X_{n+1},X_{n+2},\dots)$, then 
define \smash{$\Ftail = \cap_{n=0}^\infty \, \Fcal_n$}, which is called the
\emph{tail $\sigma$-algebra} corresponding to the infinite sequence $X$.   
The next lemma, proved in Appendix \ref{app:proof_lem:Ftail_measurable},
establishes that \smash{$\widetilde{P}$} can be recovered via
$\Ftail$-measurable random variables. This will be important later on. 

\begin{lemma}\label{lem:Ftail_measurable}
Under the notation and assumptions above, for any measurable function
$f:\Xcal\to[0,\infty)$, there exists an $\Ftail$-measurable function
\smash{$g:\Xcal^\infty\to[0,\infty]$} such that 
\[
g(X)\aseq  \Ep{X'\sim\widetilde{P}}{f(X')}.
\]
\end{lemma}

\subsubsection{Representing $Q$ as a mixture distribution}

In the last part of the proof, we construct a random measure
\smash{$\widetilde{P}_*\in\Mcal_\Xcal$} as specified in
\eqref{eqn:P_from_Ptilde}, with \smash{$\widetilde{P}$} the distribution
constructed in the last subsection, from \eqref{eqn:Ptilde_is_limit}. We will
show that the distribution $Q$ of interest is equal to the mixture
\smash{$\widetilde{P}_*\circ\lambda$}. However, we have not established that
\smash{$\widetilde{P}_*\circ\lambda$} is always well-defined, so we will need to
treat this carefully.   

We will need an additional lemma, which we prove in
Appendix \ref{app:proof_lem:check_P_circ_lambda_k}. 

\begin{lemma}\label{lem:check_P_circ_lambda_k}
Under the notation and assumptions above, for all $k\geq 1$ and all $A\in\BX$, 
\begin{equation}\label{eq:lambdak_defined}
\int_\Xcal\lambda_k(x) \, \d \widetilde{P}_*(x) \in (0,\infty) \quad 
\textnormal{almost surely},
\end{equation}
and also
\begin{equation}\label{eq:P_lambdak_conditional}
\Ppst{Q}{X_k\in A}{X_{-k}} \aseq (\widetilde{P}_*\circ\lambda_k)(A),
\end{equation}
where we note that \smash{$(\widetilde{P}_*\circ\lambda_k)$} is well-defined
almost surely, by the first claim \eqref{eq:lambdak_defined}. 
\end{lemma}

By \eqref{eq:lambdak_defined}, we see that
\smash{$\widetilde{P}_*\in\Mcal_\Xcal(\lambda)$} almost surely. Therefore, we
can define $\mu$ as the distribution of this random measure
\smash{$\widetilde{P}_*$} conditional on the event
\smash{$\widetilde{P}_*\in\Mcal_\Xcal(\lambda)$},  
so that $\mu$ is a distribution over \smash{$\Mcal_\Xcal(\lambda)$}.
Next we need to verify that \smash{$Q=(P\circ\lambda)_\mu$}.
As $\BXinf$ is the product $\sigma$-algebra, it suffices
to prove that, for all $n\geq 1$ and all $A_1,\dots,A_n\in\BX$,  
\[
\Pp{Q}{X_1\in A_1, \dots, X_n\in A_n} = \Ep{P\sim \mu}{\prod_{i=1}^n
  (P\circ\lambda_i)(A_i)}. 
\]

By Lemma \ref{lem:Ftail_measurable}, for each $i \geq 1$, and any $A\in\BX$, we
can construct an $\Ftail$-measurable function
\smash{$f_{A,i}:\Xcal^\infty\to [0,\infty]$} such that  
\[
f_{A,i}(X) \aseq \Ep{X'\sim \widetilde{P}}{\one{X'\in A} \cdot 
\frac{\lambda_i(X')}{\lambda_*(X')}} = \int_A \lambda_i(x) \,\d
\widetilde{P}_*.
\] 
Applying this with $A=A_i$ and again with $A=\Xcal$, we can define an
$\Ftail$-measurable function   
\[
f_i(x) = \begin{cases}f_{A_i,i}(x)/ f_{\Xcal,i}(x)
 &\textnormal{if $f_{A_i,i}(x)<\infty$ and $f_{\Xcal,i}(x)\in (0,\infty)$},\\
0 &\textnormal{otherwise}.
\end{cases}
\] 
Applying \eqref{eq:lambdak_defined}, this satisfies 
\begin{equation}\label{eqn:construct_f}
f_i(X) \aseq \frac{\int_{A_i} \lambda_i(x)\,\d \widetilde{P}_*}{\int_\Xcal
  \lambda_i(x) \, \d \widetilde{P}_*} = (\widetilde{P}_*\circ\lambda_i)(A_i). 
\end{equation} 
By the tower law, noting that each $f_i(X)$ is $\Ftail$-measurable and
thus, is measurable with respect to $\sigma(X_{-j})\supseteq \Ftail$ for 
any $j$, we can calculate 
\allowdisplaybreaks
\begin{align*}
&\Ep{Q}{\prod_{i=1}^{m-1} f_i(X) \cdot \prod_{i=m}^n \one{X_i\in A_i}} \\
&=\Ep{Q}{\prod_{i=1}^{m-1} f_i(X) \cdot \Ppst{Q}{X_m\in A_m}{X_{-m}} \cdot
  \prod_{i=m+1}^n \one{X_i\in A_i}} \;\;
  \textnormal{(as each $f_i(X)$ is $\sigma(X_{-m})$-measurable)} \\ 
&=\Ep{Q}{\prod_{i=1}^{m-1} f_i(X) \cdot (\widetilde{P}_*\circ\lambda_m)(A_m) 
\cdot \prod_{i=m+1}^n \one{X_i\in A_i}} \;\;
 \textnormal{(by the fact in \eqref{eq:P_lambdak_conditional})} \\
&=\Ep{Q}{\prod_{i=1}^{m-1} f_i(X) \cdot f_m(X)\cdot 
\prod_{i=m+1}^n \one{X_i\in A_i}} \;\;
\textnormal{(by construction of $f_m$)} \\
&=\Ep{Q}{\prod_{i=1}^m f_i(X) \cdot \prod_{i=m+1}^n \one{X_i\in A_i}} 
\end{align*}
for each $m=1,\dots,n$. Combining these calculations over all $i=1,\dots,n$, we
obtain 
\[
\Pp{Q}{X_1\in A_1, \dots, X_n\in A_n} 
= \Ep{Q}{\prod_{i=1}^n \one{X_i\in A_i}} 
= \Ep{Q}{\prod_{i=1}^n f_i(X)} = \EE{\prod_{i=1}^n
  (\widetilde{P}_*\circ\lambda_i)(A_i)}, 
\]
where the last step holds by construction of the functions $f_i$, $i=1,\dots,n$.  
Since $\mu$ is equal to the distribution of \smash{$\widetilde{P}_*$}
conditional on the almost sure event
\smash{$\widetilde{P}_*\in\Mcal_\Xcal(\lambda)$},  
we therefore see that
\[
\EE{\prod_{i=1}^n (\widetilde{P}_*\circ\lambda_i)(A_i)}
= \Ep{P\sim\mu}{\prod_{i=1}^n (P\circ\lambda_i)(A_i)},
\]
which completes the proof.

{\RaggedRight
\bibliographystyle{plainnat}
\bibliography{definetti}}

\clearpage\appendix 

\section{Proofs of properties of weighted exchangeability}
In this section, we prove Propositions \ref{prop:wtd_exch_equiv} and \ref{prop:Xi_condition_on_Ecalm}, which establish some key properties
of weighted exchangeable distributions. 

\subsection{A note on regularity conditions}

Before proceeding, we pause to comment on a technical point about conditional 
distributions. As $\Xcal$ (and thus also $\Xcal^n$ and $\Xcal^\infty$) is a
standard Borel space, this suffices (see, e.g., \citet{faden1985existence})
to ensure the existence of a regular conditional distribution (also called a
product regular conditional probability kernel) for, say, $X \mid \Ecal_\infty$, or
$(X_{n+1},X_{n+2},\dots) \mid (X_1,\dots,X_n)$, etc.  For example, if we
consider the conditional distribution of $X \mid \Ecal_\infty$, the standard
Borel property guarantees the existence of a kernel $\kappa :
\Xcal^\infty\times\BXinf\to[0,1]$ such that $x \mapsto \kappa(x,A)$ is
measurable for all $A\in\BXinf$, and such that $\kappa(X,\cdot)$ gives the
conditional distribution of $X \mid \Ecal_\infty$, i.e., \smash{$\kappa(X,A)
\aseq \PPst{X\in A}{\Ecal_\infty}$}. In what follows, in any part of the proof
where a conditional distribution is used, it should be understood that we have
constructed a regular conditional distribution, which is guaranteed to exist by
our standard Borel assumption.

\subsection{Proof of Proposition \ref{prop:wtd_exch_equiv}}
\label{app:wtd_exch_equiv}

We will present proofs for the infinite case, where $Q$ is a distribution on
$\Xcal^\infty$; the calculations for the finite case, where $Q$ is a
distribution on $\Xcal^n$, are similar and are omitted for brevity. 

First we prove (a)$\implies$(b). Suppose that $Q$ is $\lambda$-weighted
exchangeable.
 Note that it suffices to consider functions of the form $f(x) = 
  \one{x\in A}$, because any measurable function can be generated by such 
  functions. That is, we want to prove  
\[
\Ep{Q}{\frac{\one{X\in A}}{\lambda_i(X_i)\lambda_j(X_j)}} = 
\Ep{Q}{\frac{\one{X^{ij}\in A}}{\lambda_i(X_i)\lambda_j(X_j)}},
\]
for any $A\in\BXinf$.
As the $\sigma$-algebra on $\Xcal^\infty$ is generated by sets of the form
$A=A_1\times\dots\times A_n\times\Xcal\times\Xcal\times\dots$ where $n\geq 1$
and $A_1,\dots,A_n\in\BX$, it suffices to check that the expected values are 
equal for sets of this form. Without loss of generality it suffices to consider
$n\geq \max\{i,j\}$. Define the function 
\[
h(x_1,\dots,x_n) = \one{(x_1,\dots,x_n)\in A_1\times\dots\times A_n}
\cdot \prod_{k\in[n]\backslash\{i,j\}}\lambda_k(x_k).
\]
We then have   
\begin{align*}
\Ep{Q}{\frac{\one{X\in A}}{\lambda_i(X_i)\lambda_j(X_j)}}
&=\Ep{Q}{\frac{\one{(X_1,\dots,X_n)\in A_1\times\dots\times A_n}}
{\lambda_i(X_i)\lambda_j(X_j)}}\\
&=\Ep{Q_n}{\frac{\one{(X_1,\dots,X_n)\in A_1\times\dots\times A_n}
\cdot\prod_{k\in[n]\backslash\{i,j\}}\lambda_k(X_k)}
{\prod_{k=1}^n \lambda_k(X_k)}}\\
&=\Ep{Q_n}{\frac{h(X_1,\dots,X_n)}{\prod_{k=1}^n \lambda_k(X_k)}}\\
&=\int_{\Xcal^n}h(x_1,\dots,x_n)\,\frac{\d{Q}_n(x_1,\dots,x_n)
}{\prod_{k=1}^n \lambda_k(x_k)},
\end{align*}
and similarly,
\begin{align*}
\Ep{Q}{\frac{\one{X^{ij}\in A}}{\lambda_i(X_i)\lambda_j(X_j)}}
&=\Ep{Q}{\frac{\one{((X^{ij})_1,\dots,(X^{ij})_n)\in 
A_1\times\dots\times A_n}}{\lambda_i(X_i)\lambda_j(X_j)}}\\
&=\Ep{Q_n}{\frac{\one{((X^{ij})_1,\dots,(X^{ij})_n)\in
A_1\times\dots\times A_n}\cdot\prod_{k\in[n]\backslash\{i,j\}}
\lambda_k(X_k)}{\prod_{k=1}^n \lambda_k(X_k)}}\\
&=\Ep{Q_n}{\frac{h((X^{ij})_1,\dots,(X^{ij})_n)}
{\prod_{k=1}^n \lambda_k(X_k)}}\\
&=\int_{\Xcal^n}h((x^{ij})_1,\dots,(x^{ij})_n)\,
\frac{\d{Q}_n(x_1,\dots,x_n)}{\prod_{k=1}^n \lambda_k(x_k)}.
\end{align*}
Finally, since $Q_n$ is $(\lambda_1,\dots,\lambda_n)$-weighted exchangeable, 
\[
\int_{\Xcal^n}h(x_1,\dots,x_n)\,\frac{\d{Q}_n(x_1,\dots,x_n)}
{\prod_{k=1}^n \lambda_k(x_k)} = 
\int_{\Xcal^n}h((x^{ij})_1,\dots,(x^{ij})_n)\,\frac{\d{Q}_n(x_1,\dots,x_n)}
{\prod_{k=1}^n \lambda_k(x_k)}.
\] 

Second we prove (b)$\implies$(a). For each $n\geq 1$, we need to verify 
that the measure induced by \smash{$\frac{\d{Q}_n(x_1,\dots,x_n)}{\prod_{k=1}^n 
\lambda_k(x_k)}$} is finite exchangeable, i.e., for any $A\in\BXn$ and any
permutation $\sigma$ on $[n]$, 
\[
\int_{\Xcal^n}\one{(x_1,\dots,x_n)\in A}\frac{\d{Q}_n(x_1,\dots,x_n)}
{\prod_{k=1}^n \lambda_k(x_k)}
= \int_{\Xcal^n}\one{(x_{\sigma(1)},\dots,x_{\sigma(n)})\in A}
\frac{\d{Q}_n(x_1,\dots,x_n)}{\prod_{k=1}^n \lambda_k(x_k)}.
\]
Since the permutations on $[n]$ can be generated by pairwise swaps, it is
sufficient to show that 
\[
\int_{\Xcal^n}\one{(x_1,\dots,x_n)\in A}\frac{\d{Q}_n(x_1,\dots,x_n)}
{\prod_{k=1}^n \lambda_k(x_k)}
=\int_{\Xcal^n}\one{((x^{ij})_1,\dots,(x^{ij})_n)\in A}
\frac{\d{Q}_n(x_1,\dots,x_n)}{\prod_{k=1}^n \lambda_k(x_k)}
\]
for all $1\leq i<j\leq n$. Equivalently we need to show that
\[
\Ep{Q}{\frac{\one{(X_1,\dots,X_n)\in A}
}{\prod_{k=1}^n \lambda_k(X_k)}} = 
\Ep{Q}{\frac{ \one{((X^{ij})_1,\dots,(X^{ij})_n)\in A}}
{\prod_{k=1}^n \lambda_k(X_k)}}.
\]
Define
\[
f(x) = \frac{\one{(x_1,\dots,x_n) \in A}}
{\prod_{k\in[n]\backslash\{i,j\}}\lambda_k(x_k)}.
\]
Then since we have assumed (b) holds, we have
\smash{$\Ep{Q}{\frac{f(X)}{\lambda_i(X_i)\lambda_j(X_j)}} =
\Ep{Q}{\frac{f(X^{ij})}{\lambda_i(X_i)\lambda_j(X_j)}}$},
which completes the proof.

\subsection{Proof of Proposition \ref{prop:Xi_condition_on_Ecalm}}
\label{app:Xi_condition_on_Ecalm}

First consider the case that $Q$ is a $\lambda$-weighted exchangeable
distribution on $\Xcal^n$. We need to show that, for any $A\in\BX$, 
\[
\Ppst{Q}{X_i\in A}{\Ecal_m} \aseq
\sum_{j=1}^m  \frac{\sum_{\sigma\in\Scal_m : \sigma(i)=j} \prod_{k=1}^m \lambda_k(X_{\sigma(k)})}{\sum_{\sigma\in\Scal_m} \prod_{k=1}^m \lambda_k(X_{\sigma(k)}) } \cdot  \one{X_j\in A}.
\]
Equivalently, it is sufficient to check that
\[
\Ep{Q}{\sum_{j=1}^m  \frac{\sum_{\sigma\in\Scal_m : \sigma(i)=j} \prod_{k=1}^m \lambda_k(X_{\sigma(k)})}{\sum_{\sigma\in\Scal_m} \prod_{k=1}^m \lambda_k(X_{\sigma(k)}) } \cdot  \one{X_j\in A}\cdot\one{X\in C}} = 
\Ep{Q}{\one{X_i\in A}\cdot\one{X\in C}}
\]
holds for any $C\in\Ecal_m$. Define \smash{$g(x) = \frac{\one{x\in 
    C}}{\sum_{\sigma\in\Scal_m} \prod_{k=1}^m\lambda_k(x_{\sigma(k)})}$}, and 
note that as $C\in\Ecal_m$, we see that $g$ is $\Ecal_m$-measurable, i.e., 
$g((x_{\sigma(1)},\dots,x_{\sigma(m)},x_{m+1},\dots,x_n)) = g(x)$
for all $x\in\Xcal^n$ and all $\sigma\in\R^m$. We have
\begin{align*}
&\Ep{Q}{\sum_{j=1}^m  \frac{\sum_{\sigma\in\Scal_m : \sigma(i)=j} \prod_{k=1}^m \lambda_k(X_{\sigma(k)})}{\sum_{\sigma\in\Scal_m} \prod_{k=1}^m \lambda_k(X_{\sigma(k)}) } \cdot  \one{X_j\in A}\cdot\one{X\in C}}\\
&=\Ep{Q}{ \frac{\sum_{\sigma\in\Scal_m } \prod_{k=1}^m \lambda_k(X_{\sigma(k)})\cdot  \one{X_{\sigma(i)}\in A} }{\sum_{\sigma\in\Scal_m} \prod_{k=1}^m \lambda_k(X_{\sigma(k)}) } \cdot\one{X\in C}}\\
&=\Ep{Q}{ \sum_{\sigma\in\Scal_m } \prod_{k=1}^m \lambda_k(X_{\sigma(k)})\cdot  \one{X_{\sigma(i)}\in A}  \cdot g(X)}\\
&= \int_{\Xcal^n}\sum_{\sigma\in\Scal_m}\prod_{k=1}^m\lambda_k(x_{\sigma(k)})\cdot \one{x_{\sigma(i)}\in A}\cdot g(x)\,\d{Q}(x)\\
&=\sum_{\sigma\in\Scal_m} \int_{\Xcal^n}\prod_{k=1}^m\lambda_k(x_{\sigma(k)})\cdot \one{x_{\sigma(i)}\in A}\cdot g(x)\cdot \prod_{k=1}^m\lambda_k(x_k)\cdot \prod_{k=m+1}^n\lambda_k(x_k)\,\frac{\d{Q}(x)}{\prod_{k=1}^n\lambda_k(x_k)}\\
&=\sum_{\sigma\in\Scal_m} \int_{\Xcal^n}\prod_{k=1}^m\lambda_k(x_k)\cdot \one{x_i\in A}\cdot g(x)\cdot \prod_{k=1}^m\lambda_k(x_{\sigma^{-1}(k)})\cdot \prod_{k=m+1}^n\lambda_k(x_k)\,\frac{\d{Q}(x)}{\prod_{k=1}^n\lambda_k(x_k)}\\
&=\sum_{\sigma\in\Scal_m} \int_{\Xcal^n} \one{x_i\in A}\cdot g(x)\cdot \prod_{k=1}^m\lambda_k(x_{\sigma^{-1}(k)})\,\d{Q}(x)\\
&=\int_{\Xcal^n} \one{x_i\in A}\cdot g(x)\cdot  \bigg(\sum_{\sigma\in\Scal_m}  \prod_{k=1}^m\lambda_k(x_{\sigma^{-1}(k)})\bigg)\,\d{Q}(x)\\
&=\int_{\Xcal^n} \one{x_i\in A}\cdot \one{x \in C}\,\d{Q}(x)\\
&=\Ep{Q}{\one{X_i\in A}\cdot\one{X\in C}},
\end{align*}
where the fifth equality holds by replacing $x$ with
$(x_{\sigma^{-1}(1)},\dots,x_{\sigma^{-1}(m)},x_{m+1},\dots,x_n)$, and recalling
that $Q$ is $\lambda$-weighted exchangeable while $g$ is $\Ecal_m$-measurable.
This verifies the desired equality.

Now suppose $Q$ is a $\lambda$-weighted exchangeable distribution on
$\Xcal^\infty$.  For any $n\geq m$, define
\[
\Ecal_{m,n} = \Big\{A\in\BXn : \one{(x_1,\dots,x_m)\in A} = 
\one{(x_{\sigma(1)},\dots,x_{\sigma(m)},x_{m+1},\dots,x_n)\in A},
\;\textnormal{for all $x\in\Xcal^n$ and $\sigma\in\Scal_m$}\Big\}.
\]
We can then verify that $\Ecal_m\subseteq\BXinf$ is the minimal $\sigma$-algebra 
generated by $\cup_{n\geq m} \, \Ecal_{m,n}$.  By Levy's Upwards Theorem (e.g.,
Chapter 14.2 of \citet{williams1991probability}), 
\[
\Ppst{Q}{X_i\in A}{\Ecal_m} \aseq \lim_{n\to\infty} \Ppst{Q}{X_i\in A}{\Ecal_{m,n}}.
\]
Applying our work above to the finite $(\lambda_1,\dots,\lambda_n)$-weighted
exchangeable distribution $Q_n$, we have
\[
\Ppst{Q}{X_i\in A}{\Ecal_{m,n}} \aseq 
\sum_{j=1}^m \frac{\sum_{\sigma\in\Scal_m : \sigma(i)=j} \prod_{k=1}^m \lambda_k(X_{\sigma(k)})}{\sum_{\sigma\in\Scal_m} \prod_{k=1}^m \lambda_k(X_{\sigma(k)}) } \cdot  \one{X_j\in A}
\]
for each $n\geq m$, and therefore, as desired, we conclude that
\[
\Ppst{Q}{X_i\in A}{\Ecal_m} \aseq 
\sum_{j=1}^m \frac{\sum_{\sigma\in\Scal_m : \sigma(i)=j} \prod_{k=1}^m \lambda_k(X_{\sigma(k)})}{\sum_{\sigma\in\Scal_m} \prod_{k=1}^m \lambda_k(X_{\sigma(k)}) } \cdot  \one{X_j\in A}.
\]

\section{Proofs of supporting lemmas for main results}

\subsection{Proof of Lemma \ref{lem:mix_unique}}
\label{app:mix_unique}

Fix any $i\neq j$ and define $\lambda_{ij}(x) =
\min\{\lambda_i(x),\lambda_j(x)\}$. For any $P\in\Mcal_\Xcal(\lambda)$, a
direct calculation shows that
\[
(P\circ\lambda_i)(A) = 
\frac{\Ep{P\circ\lambda_{ij}}{\one{X\in A}\cdot \frac{\lambda_i(X)}{\lambda_{ij}(X)}}}
{\Ep{P\circ\lambda_{ij}}{ \frac{\lambda_i(X)}{\lambda_{ij}(X)}}}
\]
for all $A\in\BX$. Therefore, if \smash{$(P\circ\lambda_{ij})(A)\aseq
(P_*\circ\lambda_{ij})(A)$} for any $i\neq j$ and for all $A\in\BX$, then this
implies that $(P\circ\lambda_i)(A)\aseq (P_*\circ\lambda_i)(A)$ for any $i$ 
and for all $A\in\BX$.  Therefore, it suffices to show that
\smash{$(P\circ\lambda_{ij})(A)\aseq (P_*\circ\lambda_{ij})(A)$} for any
$A\in\BX$. 

Next define a new distribution \smash{$\tilde\mu$} on $\Mcal_\Xcal(\lambda)$ as 
\[
\tilde\mu(B) =
\frac{\int_B \frac{\int_\Xcal\lambda_{ij}(x)\,\d{P}(x)}{\int_\Xcal\lambda_i(x)\,\d{P}(x)}\cdot \frac{\int_\Xcal\lambda_{ij}(x)\,\d{P}(x)}{\int_\Xcal\lambda_j(x)\,\d{P}(x)}\,\d\mu(P)}{
\int_{\Mcal_\Xcal(\lambda)} \frac{\int_\Xcal\lambda_{ij}(x)\,\d{P}'(x)}{\int_\Xcal\lambda_i(x)\,\d{P}'(x)}\cdot \frac{\int_\Xcal\lambda_{ij}(x)\,\d{P}'(x)}{\int_\Xcal\lambda_j(x)\,\d{P}'(x)}\,\d\mu(P')}
\] 
for all measurable $B\subseteq\Mcal_\Xcal(\lambda)$. (This is well-defined,
because the integrand in the denominator is positive and is bounded by 1,
surely, by construction.)  Since $\mu$ and \smash{$\tilde\mu$} are absolutely
continuous with respect to each other, with Radon--Nikodym derivative  
\[
\frac{\d\tilde\mu(P)}{\d\mu(P)}= \frac{ \frac{\int_\Xcal\lambda_{ij}(x)\,\d{P}(x)}{\int_\Xcal\lambda_i(x)\,\d{P}(x)}\cdot \frac{\int_\Xcal\lambda_{ij}(x)\,\d{P}(x)}{\int_\Xcal\lambda_j(x)\,\d{P}(x)}}{
\int_{\Mcal_\Xcal(\lambda)} \frac{\int_\Xcal\lambda_{ij}(x)\,\d{P}'(x)}{\int_\Xcal\lambda_i(x)\,\d{P}'(x)}\cdot \frac{\int_\Xcal\lambda_{ij}(x)\,\d{P}'(x)}{\int_\Xcal\lambda_j(x)\,\d{P}'(x)}\,\d\mu(P')}>0,
\] 
we see that $(P\circ\lambda_{ij})(A)= (P_*\circ\lambda_{ij})(A)$ holds almost
surely under $P\sim\mu$ if and only if it holds almost surely under
\smash{$P\sim\tilde\mu$}, so now we will verify the statement for
\smash{$P\sim\tilde\mu$}. This is the same as verifying that
\[
\Ep{P\sim\tilde\mu}{(P\circ\lambda_{ij})(A)} = (P_*\circ\lambda_{ij})(A) \quad \textnormal{and} \quad
\textnormal{Var}_{P\sim\tilde\mu}({(P\circ\lambda_{ij})(A)}) = 0,
\]
or equivalently, 
\[
\Ep{P\sim\tilde\mu}{(P\circ\lambda_{ij})(A)} = (P_*\circ\lambda_{ij})(A) \quad \textnormal{and} \quad
\Ep{P\sim\tilde\mu}{\big((P\circ\lambda_{ij})(A)\big)^2} = \big((P_*\circ\lambda_{ij})(A)\big)^2.
\]
For any $B_1,B_2\in\BX$,
\begin{align*}
&\Ep{P\sim \tilde\mu}{(P\circ\lambda_{ij})^2(B_1\times B_2)} \\
&=\Ep{P\sim\mu}{\int (P\circ\lambda_{ij})^2(B_1\times B_2)\frac{ \frac{\int_\Xcal\lambda_{ij}(x)\,\d{P}(x)}{\int_\Xcal\lambda_i(x)\,\d{P}(x)}\cdot \frac{\int_\Xcal\lambda_{ij}(x)\,\d{P}(x)}{\int_\Xcal\lambda_j(x)\,\d{P}(x)}}{
\int_{\Mcal_\Xcal(\lambda)} \frac{\int_\Xcal\lambda_{ij}(x)\,\d{P}(x)}{\int_\Xcal\lambda_i(x)\,\d{P}'(x)}\cdot \frac{\int_\Xcal\lambda_{ij}(x)\,\d{P}'(x)}{\int_\Xcal\lambda_j(x)\,\d{P}'(x)}\,\d\mu(P')}} \;\; \textnormal{(by definition of $\tilde\mu$)} \\
&=
\frac{\Ep{P\sim \mu}{\Ep{P\circ\lambda}{\one{(X_i,X_j)\in B_1\times B_2}\cdot\frac{\lambda_{ij}(X_i)}{\lambda_i(X_i)}\cdot\frac{\lambda_{ij}(X_j)}{\lambda_j(X_j)}}}}{
\Ep{P\sim\mu}{\Ep{P\circ\lambda}{\frac{\lambda_{ij}(X_i)}{\lambda_i(X_i)}\cdot\frac{\lambda_{ij}(X_j)}{\lambda_j(X_j)}}}}
\;\; \textnormal{(by definition of $P\circ\lambda$ and of $P\circ\lambda_{ij}$)}\\
&= 
\frac{\Ep{P_*\circ\lambda}{\one{(X_i,X_j)\in B_1\times B_2}\cdot\frac{\lambda_{ij}(X_i)}{\lambda_i(X_i)}\cdot\frac{\lambda_{ij}(X_j)}{\lambda_j(X_j)}}
}{
\Ep{P_*\circ\lambda}{\frac{\lambda_{ij}(X_i)}{\lambda_i(X_i)}\cdot\frac{\lambda_{ij}(X_j)}{\lambda_j(X_j)}}} \;\; \textnormal{(since $P_*\circ\lambda = Q = (P\circ\lambda)_\mu$)}\\
&= 
\frac{\Ep{P_*\circ\lambda_i}{\one{X\in B_1}\cdot\frac{\lambda_{ij}(X)}{\lambda_i(X)}}
}{
\Ep{P_*\circ\lambda_i}{\frac{\lambda_{ij}(X)}{\lambda_i(X)}}}\cdot\frac{\Ep{P_*\circ\lambda_j}{\one{X\in B_2}\cdot\frac{\lambda_{ij}(X)}{\lambda_j(X)}}}{\Ep{P_*\circ\lambda_j}{\frac{\lambda_{ij}(X)}{\lambda_j(X)}}} \;\; \textnormal{(by definition of $P_*\circ\lambda$)}\\
&=(P_*\circ\lambda_{ij})(B_1)\cdot (P_*\circ\lambda_{ij})(B_2) \;\; \textnormal{(by definition of $P_*\circ\lambda_i$, $P_*\circ\lambda_j$, $P_*\circ\lambda_{ij}$)}.
\end{align*}
Applying this calculation with $B_1=A$ and $B_2=\Xcal$ we obtain
\[
\Ep{P\sim\tilde\mu}{(P\circ\lambda_{ij})(A)}
= \Ep{P\sim\tilde\mu}{(P\circ\lambda_{ij})(A)\cdot(P\circ\lambda_{ij})(\Xcal)} = 
(P_*\circ\lambda_{ij})(A)\cdot(P_*\circ\lambda_{ij})(\Xcal)=(P_*\circ\lambda_{ij})(A).
\]
Applying the calculation again with $B_1=B_2=A$, we obtain
\[
\Ep{P\sim\tilde\mu}{\big((P\circ\lambda_{ij})(A)\big)^2} =\big((P_*\circ\lambda_{ij})(A)\big)^2,
\]
which completes the proof. 

\subsection{Proof of Lemma \ref{lem:Pbar_equiv}}\label{app:Pbar_equiv}

First consider the case that \smash{$\sum_{i=1}^\infty p_i(X_i)\cdot \one{X_i\in
    A}=\infty$}. For $n$ sufficiently large so that \smash{$\sum_{i=1}^n B_i>0$} 
  and \smash{$\sum_{i=1}^n p_i(X_i)\cdot \one{X_i\in A}>0$}, we calculate 
\begin{align*}
\widebar{P}_n(A) &= \frac{\sum_{i=1}^n B_i\cdot \one{X_i\in A}}{\sum_{i=1}^n B_i}\\
&=  \frac{\sum_{i=1}^n p_i(X_i)\cdot \one{X_i\in A}}{\sum_{i=1}^n p_i(X_i)}\cdot\frac{\sum_{i=1}^np_i(X_i)}{\sum_{i=1}^n B_i}\cdot\frac{\sum_{i=1}^n B_i\cdot\one{X_i\in A}}{\sum_{i=1}^n p_i(X_i)\cdot\one{X_i\in A}}\\
&=\widetilde{P}_n(A)\cdot\frac{\sum_{i=1}^np_i(X_i)}{\sum_{i=1}^n B_i}\cdot\frac{\sum_{i=1}^n B_i\cdot\one{X_i\in A}}{\sum_{i=1}^n p_i(X_i)\cdot\one{X_i\in A}},
\end{align*}
where the last step holds by the definition of \smash{$\widetilde{P}_n$} and of 
$p_i(X_i)$ given in \eqref{eq:Ptilde} and \eqref{eq:Bi}.  Thus,
\[
\Big|\widebar{P}_n(A)  - \widetilde{P}_n(A)\Big| = \widetilde{P}_n(A) \cdot\left|\frac{\sum_{i=1}^np_i(X_i)}{\sum_{i=1}^n B_i}\cdot\frac{\sum_{i=1}^n B_i\cdot\one{X_i\in A}}{\sum_{i=1}^n p_i(X_i)\cdot\one{X_i\in A}} - 1\right|,
\]
and so, since \smash{$\widetilde{P}_n(A)\leq 1$} because
\smash{$\widetilde{P}_n$} is a distribution,
\[
\limsup_{n\to \infty} \Big|\widebar{P}_n(A)  - \widetilde{P}_n(A)\Big| 
\leq\limsup_{n\to \infty}\left|\frac{\sum_{i=1}^np_i(X_i)}{\sum_{i=1}^n B_i}\cdot\frac{\sum_{i=1}^n B_i\cdot\one{X_i\in A}}{\sum_{i=1}^n p_i(X_i)\cdot\one{X_i\in A}} - 1\right|.
\]

Next, since \smash{$\sum_{i=1}^\infty p_i(X_i)=\infty$} and
\smash{$\sum_{i=1}^\infty p_i(X_i)\one{X_i\in A}=\infty$} in this case, we have 
\[
\lim_{n\to\infty}\frac{\sum_{i=1}^n B_i}{\sum_{i=1}^np_i(X_i)} = 1, \quad \textnormal{almost surely (conditional on $X$)},
\]
by Chapter IX, Section \S3, Theorem 12 of \citet{petrov1975sums} (applied by
checking the variance condition \smash{$\sum_{i\geq i_*}
    \frac{\textnormal{{Var}}(B_i\mid X_i)}{(\sum_{j\leq i}p_j(X_j))^2} \leq
    \sum_{i\geq i_*}\frac{p_i(X_i)}{(\sum_{j\leq i}p_j(X_j))^2} < \infty$, for 
  $i_* = \min\{i: p_i(X_i)>0\}$}). Thus, similarly, 
\[
\lim_{n\to\infty}\frac{\sum_{i=1}^n B_i\cdot\one{X_i\in A}}{\sum_{i=1}^n p_i(X_i)\cdot\one{X_i\in A}}=1, \quad \textnormal{almost surely (conditional on $X$)}.
\]
Thus, on the event \smash{$\sum_{i=1}^\infty p_i(X_i)\cdot\one{X_i\in A}
  =\infty$}, we have
\[
\limsup_{n\to \infty}\left|\frac{\sum_{i=1}^np_i(X_i)}{\sum_{i=1}^n B_i}\cdot\frac{\sum_{i=1}^n B_i\cdot\one{X_i\in A}}{\sum_{i=1}^n p_i(X_i)\cdot\one{X_i\in A}} - 1\right| = 0, \quad \textnormal{almost surely (conditional on $X$)},
\]
and so
\begin{equation}\label{eqn:toshow_lem:Pbar_equiv}
\limsup_{n\to \infty}
\Big|\widebar{P}_n(A)  - \widetilde{P}_n(A)\Big| = 0, \quad \textnormal{almost surely (conditional on $X$)}.
\end{equation}
 
Next, consider the case that \smash{$\sum_{i=1}^\infty
  p_i(X_i)\cdot\one{X_i\in A} <\infty$} while \smash{$\sum_{i=1}^\infty
  p_i(X_i)=\infty$}. Then in this case, conditional on $X$, it holds almost
surely that \smash{$\sum_{i=1}^\infty B_i\cdot\one{X_i\in A}<\infty$} while  
\smash{$\sum_{i=1}^\infty B_i=\infty$} (by the first and second Borel--Cantelli
Lemmas, respectively); on these events, we have
\[
\lim_{n\to\infty}\widetilde{P}_n(A)=\lim_{n\to\infty}\widebar{P}_n(A)=0.
\]
This implies \eqref{eqn:toshow_lem:Pbar_equiv} again holds conditional on $X$. 

Finally, we have shown that \smash{$\sum_{i=1}^\infty p_i(X_i)=\infty$} almost 
surely. Combining the cases considered above, we see that the claim 
\eqref{eqn:toshow_lem:Pbar_equiv} holds marginally, i.e., with respect to the
distribution of $X$ as well as $B_1,B_2,\dots$, which completes the proof. 

\subsection{Proof of Lemma \ref{lem:exch_inf_subseq}}
\label{app:exch_inf_subseq}

By Proposition \ref{prop:exch_equiv}, it suffices to show that, 
 for any $j\neq k$ and any measurable
  $f:\Xcal^\infty\to[0,\infty)$, 
\[
\Ep{Q}{f(X)} = \Ep{Q}{f(X^{jk})}.
\]
It is equivalent to verify this for functions of the form $f(x) = \one{x\in A}$,
where $A\in\BXinf$, i.e., to show that for any $j\neq k$ and for any
$A\in\BXinf$,  
\[
\PP{\check{X}\in A} = \PP{\check{X}^{jk}\in A}.
\] 
Since $\BXinf$ is the product $\sigma$-algebra, it suffices to show that, for
all $n\geq \max\{j,k\}$ and all $A\in\BXn$,
\[
\PP{(\check{X}_1,\dots,\check{X}_n)\in A} 
= \PP{((\check{X}^{jk})_1,\dots,(\check{X}^{jk})_n)\in A}.
\]
We will prove the stronger statement
\[
\PPst{(\check{X}_1,\dots,\check{X}_n)\in A}{I_1,\dots,I_n} 
\aseq \PPst{((\check{X}^{jk})_1,\dots,(\check{X}^{jk})_n)\in A}{I_1,\dots,I_n}.
\]
Equivalently, for any indices $1\leq i_1<\dots<i_n < \infty$ such that the event
$E = \{I_1 = i_1,\dots,I_n = i_n\}$ has positive probability, we need to verify
that
\[
\PPst{(\check{X}_1,\dots,\check{X}_n)\in A}{I_1=i_1,\dots,I_n=i_n} 
= \PPst{((\check{X}^{jk})_1,\dots,(\check{X}^{jk})_n)\in A}{I_1=i_1,\dots,I_n=i_n},
\]
or equivalently,
\[
\EE{\one{(\check{X}_1,\dots,\check{X}_n)\in A}\one{E}}
= \EE{\one{((\check{X}^{jk})_1,\dots,(\check{X}^{jk})_n)\in A}\one{E}}.
\]
Now let us rewrite this to be more interpretable. Define
\[
A' = \{x\in\Xcal^\infty : (x_{i_1},\dots,x_{i_n})\in A\},
\]
thus $\one{(X_{i_1},\dots,X_{i_n})\in A} = \one{X\in A'}$ holds surely,
meaning that on $E$, it holds that
\smash{$\one{(\check{X}_1,\dots,\check{X}_n)\in A}=\one{X\in A'}$}.  
Moreover, by construction,
\[
(X_{i_1},\dots,X_{i_n})^{jk} = ((X^{i_ji_k})_{i_1},\dots,(X^{i_ji_k})_{i_n})
\]
and thus \smash{$\one{(X_{i_1},\dots,X_{i_n})^{jk}\in A} = \one{X^{i_ji_k}\in
    A'}$}, so that on $E$, it holds that
\smash{$\one{(\check{X}_1,\dots,\check{X}_n)^{jk}\in 
    A}=\one{X^{i_ji_k}\in A'}$}. Therefore, we now just need to show 
\[
\EE{\one{X\in A'}\one{E}}= \EE{\one{X^{i_ji_k}\in A}\one{E}}.
\]
Next, we calculate
\begin{align*}
\PPst{E}{X} &= \PPst{B_\ell = \one{\ell\in\{i_1,\dots,i_n\}}, \ell = 1,\dots,i_n}{X}\\
&= \prod_{r=1}^n \PPst{B_{i_r}=1}{X} \cdot\prod_{\ell\in[i_n]\backslash\{i_1,\dots,i_n\}}\PPst{B_{i_r}=0}{X}\\
&= \prod_{r=1}^n p_{i_r}(X_{i_r}) \cdot\prod_{\ell\in[i_n]\backslash\{i_1,\dots,i_n\}} (1-p_\ell(X_\ell))\\
&= \prod_{r=1}^n \frac{\inf_{x\in\Xcal}\lambda_{i_r}(x)/\lambda_*(x)}{\lambda_{i_r}(X_{i_r})/\lambda_*(X_{i_r})} \cdot\prod_{\ell\in[i_n]\backslash\{i_1,\dots,i_n\}} \left(1-\frac{\inf_{x\in\Xcal}\lambda_\ell(x)/\lambda_*(x)}{\lambda_\ell(X_\ell)/\lambda_*(X_\ell)}\right)\\
&=  \frac{h(X)}{\lambda_{i_j}(X_{i_j})\lambda_{i_k}(X_{i_k})},
\end{align*}
where
\begin{multline*}
h(x) =  \frac{\inf_{x'\in\Xcal}\lambda_{i_j}(x')/\lambda_*(x')}{1/\lambda_*(x_{i_j})}\cdot   \frac{\inf_{x'\in\Xcal}\lambda_{i_k}(x')/\lambda_*(x')}{1/\lambda_*(x_{i_k})} \cdot{} \\ \prod_{r\in[n]\backslash\{j,k\}} \frac{\inf_{x'\in\Xcal}\lambda_{i_r}(x')/\lambda_*(x')}{\lambda_{i_r}(x_{i_r})/\lambda_*(x_{i_r})} \cdot\prod_{\ell\in[i_n]\backslash\{i_1,\dots,i_n\}} \left(1-\frac{\inf_{x'\in\Xcal}\lambda_\ell(x')/\lambda_*(x')}{\lambda_\ell(x_\ell)/\lambda_*(x_\ell)}\right).
\end{multline*}
We can therefore write
\[
\EE{\one{X\in A'}\one{E}} = \EE{\one{X\in A'}\cdot\PPst{E}{X}}
= \EE{\one{X\in A'}\cdot  \frac{h(X)}{\lambda_{i_j}(X_{i_j})\lambda_{i_k}(X_{i_k})}},
\]
and similarly
\[
\EE{\one{X^{i_ji_k}\in A'}\one{E}} = \EE{\one{X^{i_ji_k}\in A'}\cdot  \frac{h(X)}{\lambda_{i_j}(X_{i_j})\lambda_{i_k}(X_{i_k})}}
= \EE{\one{X^{i_ji_k}\in A'}\cdot  \frac{h(X^{i_ji_k})}{\lambda_{i_j}(X_{i_j})\lambda_{i_k}(X_{i_k})}},
\]
where the last step holds because \smash{$h(x) = h(x^{i_ji_k})$} for all $x$, by
definition. By Proposition \ref{prop:wtd_exch_equiv}, we have
\[
\EE{\one{X^{i_ji_k}\in A'}\cdot  \frac{h(X^{i_ji_k})}{\lambda_{i_j}(X_{i_j})\lambda_{i_k}(X_{i_k})}} 
= \EE{\one{X\in A'}\cdot  \frac{h(X)}{\lambda_{i_j}(X_{i_j})\lambda_{i_k}(X_{i_k})}},
\]
which proves that
\[
\EE{\one{X\in A'}\one{E}} =\EE{\one{X^{i_ji_k}\in A'}\one{E}},
\]
and thus completes the proof.

\subsection{Proof of Lemma \ref{lem:Ftail_measurable}}
\label{app:proof_lem:Ftail_measurable}

First, we calculate
\begin{align*}
 \Ep{X'\sim\widetilde{P}}{f(X')}
&=\int_\Xcal f(x)\,\d\widetilde{P}(x)\\
&=\int_\Xcal \int_{t\geq 0}\one{f(x)\geq t}\,\d{t}\,\d\widetilde{P}(x)\\
&= \int_{t\geq 0}\int_\Xcal \one{f(x) \geq t}\,\d\widetilde{P}(x)\,\d{t} \;\; \textnormal{(by Tonelli's theorem)}\\
&= \int_{t\geq 0}\widetilde{P}(L_t)\,\d{t},
\end{align*}
where $L_t$ is the nested family of sets given by $L_t = \{x\in\Xcal: f(x)\geq
t\}\in\BX$.

Simplifying further, it is therefore sufficient to show that, for any $C\in\BX$,  
there exists an $\Ftail$-measurable function $h_C:\Xcal^\infty\to[0,1]$
such that \smash{$h_C(X)\aseq \widetilde{P}(C)$}. Taking
\[
h_C(X)= \limsup_{n\to\infty}\widetilde{P}_n(C),
\]
by \eqref{eqn:Ptilde_is_limit}, we see that \smash{$h_C(X)\aseq
  \widetilde{P}(C)$} must hold.  It remains to show that the random variable
$h_C(X)$ is $\Ftail$-measurable. 

For any $n> m\geq 0$ define
\[
T_{m,n} = 
\frac{
\sum_{i=m+1}^n\frac{\inf_{x\in\Xcal}\lambda_i(x)/\lambda_*(x)}{\lambda_i(X_i)/\lambda_*(X_i)}
 \cdot \one{X_i\in C}}{\sum_{i=m+1}^n \frac{\inf_{x\in\Xcal}\lambda_i(x)/\lambda_*(x)}{\lambda_i(X_i)/\lambda_*(X_i)}},
\]
where, for any fixed $m$, $T_{m,n}$ is well-defined (i.e., the denominator is
positive) for sufficiently large $n$ by \eqref{eq:suff_condition}.  Note that
$T_{m,n}$ is $\Fcal_m$-measurable, and so \smash{$\limsup_{n\to\infty} 
  T_{m,n}$} is also $\Fcal_m$-measurable.  To complete the proof, we will verify
that 
\[
\limsup_{n\to\infty}\widetilde{P}_n(C) = \limsup_{n\to\infty} T_{m,n}
\]
holds surely, for any $m$; if this is the case, then $h_C(X)$ is
$\Fcal_m$-measurable for all $m\geq 0$, and thus is $\Ftail$-measurable. It is 
sufficient to verify that
\begin{equation}\label{eq:check_limsup} 
\limsup_{n\to\infty}\Big|T_{m,n} - \widetilde{P}_n(C)\Big| = 0
\end{equation}
holds surely.  We calculate (for any $n>m$ sufficiently large, so that $T_{m,n}$
is well-defined),
\begin{align*}
\Big|T_{m,n} - \widetilde{P}_n(C)\Big|
&= \left| \frac{
\sum_{i=m+1}^n\frac{\inf_{x\in\Xcal}\lambda_i(x)/\lambda_*(x)}{\lambda_i(X_i)/\lambda_*(X_i)}
 \cdot
\one{X_i\in C}}{\sum_{i=m+1}^n \frac{\inf_{x\in\Xcal}\lambda_i(x)/\lambda_*(x)}{\lambda_i(X_i)/\lambda_*(X_i)}}
 -\frac{
\sum_{i=1}^n\frac{\inf_{x\in\Xcal}\lambda_i(x)/\lambda_*(x)}{\lambda_i(X_i)/\lambda_*(X_i)}
 \cdot
\one{X_i\in C}}{\sum_{i=1}^n \frac{\inf_{x\in\Xcal}\lambda_i(x)/\lambda_*(x)}{\lambda_i(X_i)/\lambda_*(X_i)}}
\right|\\
&\leq \frac{ \sum_{i=1}^m  \frac{\inf_{x\in\Xcal}\lambda_i(x)/\lambda_*(x)}{\lambda_i(X_i)/\lambda_*(X_i)} }{ \sum_{i=1}^n \frac{\inf_{x\in\Xcal}\lambda_i(x)/\lambda_*(x)}{\lambda_i(X_i)/\lambda_*(X_i)}}
\leq \frac{m}{\sum_{i=1}^n \frac{\inf_{x\in\Xcal}\lambda_i(x)/\lambda_*(x)}{\sup_{x\in\Xcal}\lambda_i(x)/\lambda_*(x)}},
\end{align*}
and note that the limit of the denominator is $\infty$ by
\eqref{eq:suff_condition}. Therefore, \eqref{eq:check_limsup} holds surely, and
thus we have shown that $h_C(X)$ is $\Ftail$-measurable, which completes the
proof.

\subsection{Proof of Lemma \ref{lem:check_P_circ_lambda_k}}
\label{app:proof_lem:check_P_circ_lambda_k}

We split the proof into two parts. First, we verify that the two claims hold for
a particular choice of $k$.  Then we extend to all $k\geq 1$.

\subsubsection{Special case: proof for a specific $k$}

First we fix any $k \geq 1$ satisfying 
\[
\frac{\inf_{x\in\Xcal} \, \lambda_k(x)/\lambda_*(x)}
{\sup_{x\in\Xcal} \, \lambda_k(x)/\lambda_*(x)} > 0.
\]
Note that there must be infinitely many such $k$, by \eqref{eq:suff_condition}.
We will first show that the two claims hold for this particular choice of $k$.

First, we check \eqref{eq:lambdak_defined}: since \smash{$\widetilde{P}$} is a 
distribution, we have
\[
\int_\Xcal\lambda_k(x) \, \d\widetilde{P}_*(x) =
\int_\Xcal \frac{\lambda_k(x)}{\lambda_*(x)}\,\d\widetilde{P}(x) \in 
\left[\inf_{x\in\Xcal} \, \frac{\lambda_k(x)}{\lambda_*(x)}, \;
\sup_{x\in\Xcal} \, \frac{\lambda_k(x)}{\lambda_*(x)}\right]
\subseteq (0,\infty),
\]
where the last step holds by our choice of $k$. Therefore,
\eqref{eq:lambdak_defined} holds surely for our choice of $k$.

Next, we need to prove that \eqref{eq:P_lambdak_conditional} holds for our
choice of $k$. It is equivalent to show that, for $X\sim Q$,
\[
\PP{X_k \in A, X_{-k}\in B} = \EE{(\widetilde{P}_*\circ\lambda_k)(A) \cdot \one{X_{-k}\in
    B}}, \quad \textnormal{for all $A\in\BX$, $B \in \BXinf$}.
\]
Note that we have shown that \smash{$(\widetilde{P}_*\circ\lambda_k)$} is
well-defined surely, in our proof of \eqref{eq:lambdak_defined}, above, and thus
this expected value is well-defined.  As $\BXinf$ is a product $\sigma$-algebra,
it is sufficient to show that, for all $\ell \geq k$,
\begin{equation}\label{eq:cond_kl}
\PP{X_k \in A, X_{[\ell]\backslash k}\in B} = \EE{(\widetilde{P}_*\circ\lambda_k)(A) \cdot 
  \one{X_{[\ell]\backslash k} \in B}}, \quad \textnormal{for all $A\in\BX$, $B \in
  \Bcal(\Xcal^{\ell-1})$},
\end{equation}
where \smash{$X_{[\ell ]\backslash k} = (X_1,\dots,X_{k-1},X_{k+1},\dots,X_\ell)
  \in\Xcal^{\ell -1}$}. 

The next lemma, proved in Appendix \ref{app:cond_Ptilde}, reduces the above
condition to one involving a limit of the weighted empirical distribution.   

\begin{lemma}\label{lem:cond_Ptilde} 
Under the notation and assumptions above, it holds that
\[
\lim_{n\to\infty} \left| \PP{X_k \in A, X_{[\ell]\backslash k}\in B} - 
\EE{(\widetilde{P}_n\circ(\lambda_k/\lambda_*))(A)\cdot 
\one{X_{[\ell]\backslash k}\in B}}\right| = 0,
\]
where $(\lambda_k/\lambda_*)\in\Lambda$ should be interpreted elementwise,
i.e., $(\lambda_k/\lambda_*)(x) = \lambda_k(x)/\lambda_*(x)$.
\end{lemma}

Thus to show \eqref{eq:cond_kl}, we only need to show that     
\[
\lim_{n\to\infty} \left| 
\EE{(\widetilde{P}_n\circ(\lambda_k/\lambda_*))(A) \cdot 
\one{X_{[\ell]\backslash k}\in B}} - 
\EE{(\widetilde{P}_*\circ\lambda_k)(A) \cdot 
\one{X_{[\ell ]\backslash k}\in B}} \right| = 0,
\]
for which it is sufficient to verify
\[
\lim_{n\to\infty}
\EE{ \left| (\widetilde{P}_n\circ(\lambda_k/\lambda_*))(A) - 
(\widetilde{P}_*\circ\lambda_k)(A) \right|} = 0.
\]  
Since the term inside the expectation above is at most 1 in absolute value, by
the dominated convergence theorem, it suffices to verify that
\begin{equation}\label{eqn:Ptilde_Ptilden_step}
\left|(\widetilde{P}_n\circ(\lambda_k/\lambda_*))(A)  - 
(\widetilde{P}_*\circ\lambda_k)(A)\right| \asto 0.
\end{equation}
By construction, we have \smash{$\widetilde{P}_*\circ\lambda_k =
  \widetilde{P}\circ(\lambda_k/\lambda_*)$}, so we calculate
\[
\left| (\widetilde{P}_n\circ(\lambda_k/\lambda_*))(A) - 
(\widetilde{P}_*\circ\lambda_k)(A) \right|
= \left| \frac{\int_{A} \frac{\lambda_k(x)}{\lambda_*(x)}\,\d\widetilde{P}_n(x)}
{\int_\Xcal \frac{\lambda_k(x)}{\lambda_*(x)}\,\d\widetilde{P}_n(x)}  -
\frac{\int_{A} \frac{\lambda_k(x)}{\lambda_*(x)}\,\d\widetilde{P}(x)}
{\int_\Xcal \frac{\lambda_k(x)}{\lambda_*(x)}\,\d\widetilde{P}(x)} \right|.
\]
Since $\lambda_k(x)/\lambda_*(x)$ is positive and bounded by our choice of $k$,
it suffices to show that
\[
\int_{A'} \frac{\lambda_k(x)}{\lambda_*(x)} \, \d\widetilde{P}_n(x) \asto
\int_{A'} \frac{\lambda_k(x)}{\lambda_*(x)} \, \d\widetilde{P}(x).
\]
holds for any $A'\in\BX$, which is true because we have constructed
\smash{$\widetilde{P}$} to be the almost sure weak limit of
\smash{$\widetilde{P}_n$}, as in \eqref{eqn:Ptilde_is_limit}.  We have therefore
established that \eqref{eqn:Ptilde_Ptilden_step} holds, and so we have shown
that \eqref{eq:P_lambdak_conditional} holds for our choice of $k$.

\subsubsection{General case}

Now that the two claims are established for our particular choice of $k$, we
will next show that both claims hold with $\ell$ in place of $k$, for any
$\ell\geq 1$ with $\ell\neq k$.

First, we need to show, for any $\ell\neq k$,
\[
0 < \int_\Xcal \lambda_\ell(x)\,\d\widetilde{P}_*(x)<\infty
\] holds almost
surely, to verify \eqref{eq:lambdak_defined} (with $\ell$ in place of $k$).  By
definition of \smash{$\widetilde{P}_*$,} we have
\[ 
\int_\Xcal \lambda_\ell(x)\,\d\widetilde{P}_*(x)
=  \int_\Xcal \frac{\lambda_\ell(x)}{\lambda_*(x)}\,\d\widetilde{P}(x) >0,
\]
surely, as \smash{$\frac{\lambda_\ell(x)}{\lambda_*(x)}$} is always positive,
and \smash{$\widetilde{P}\in\Pcal_\Xcal$}. Therefore the lower bound holds
surely. Next we address the upper bound.  We calculate
\begin{align*}
1 
&= \Ep{Q}{\frac{\lambda_k(X_k)\lambda_\ell(X_\ell)}{\lambda_k(X_k)\lambda_\ell(X_\ell)}}\\
&= \Ep{Q}{\frac{\lambda_k(X_\ell)\lambda_\ell(X_k)}{\lambda_k(X_k)\lambda_\ell(X_\ell)}} \;\; \textnormal{(by Proposition \ref{prop:wtd_exch_equiv})}\\
&= \Ep{Q}{\Epst{Q}{\frac{\lambda_\ell(X_k)}{\lambda_k(X_k)}}{X_{-k}} \frac{\lambda_k(X_\ell)}{\lambda_\ell(X_\ell)}}\\
&= \Ep{Q}{\int_\Xcal\frac{\lambda_\ell(x)}{\lambda_k(x)}\,\d(\widetilde{P}_*\circ\lambda_k)(x) \cdot \frac{\lambda_k(X_\ell)}{\lambda_\ell(X_\ell)}},
\end{align*}
where the last step holds since we have proved \eqref{eq:P_lambdak_conditional}
for $k$ (and recall that we have shown that
\smash{$\widetilde{P}_*\circ\lambda_k$} is well-defined, surely). Next we have 
\[
\int_\Xcal\frac{\lambda_\ell(x)}{\lambda_k(x)}\,\d(\widetilde{P}_*\circ\lambda_k)
= \frac{\int_\Xcal\lambda_\ell(x)\,\d\widetilde{P}_*(x)}{\int_\Xcal\lambda_k(x)\,\d\widetilde{P}_*(x)},
\]
and therefore, 
\[
1 = \Ep{Q}{\int_\Xcal\lambda_\ell(x)\,\d\widetilde{P}_*(x) \cdot
  \frac{\lambda_k(X_\ell)/\lambda_\ell(X_\ell)}{\int_\Xcal\lambda_k(x)\,\d\widetilde{P}_*(x)}}.
\]
Note that 
\smash{$\lambda_k(X_\ell)/\lambda_\ell(X_\ell) /
  \int_\Xcal\lambda_k(x)\,\d\widetilde{P}_*(x)>0$}  
almost surely---in the numerator, $\lambda_k,\lambda_\ell$ both take positive
finite values, while in the denominator, we apply \eqref{eq:lambdak_defined}
with our choice of $k$.  Therefore,
\smash{$\int_\Xcal\lambda_\ell(x)\,\d\widetilde{P}_*(x) <\infty$} almost surely,
which verifies that \eqref{eq:lambdak_defined} holds with $\ell$ in place of
$k$. 

Finally, we need to show, for any $\ell\neq k$ and for $X\sim Q$,
\[
\PPst{X_\ell\in A}{X_{-\ell}} \aseq (\widetilde{P}_* \circ\lambda_\ell)(A), \quad \textnormal{for all $A\in\BX$},
\]
to verify that \eqref{eq:P_lambdak_conditional} holds with $\ell$ in place of
$k$.  Equivalently, we need to show that, for any $A\in\BX$ and $B\in\BXinf$,
\[
\PP{X_\ell\in A,X_{-\ell}\in B} =
\EE{(\widetilde{P}_*\circ\lambda_\ell)(A)\cdot\one{X_{-\ell}\in B}},
\] 
since \smash{$\widetilde{P}_*\circ\lambda_\ell$} is well-defined almost surely,
but not necessarily surely, we should interpret the expected value as being
computed only over the almost sure event that
\smash{$\widetilde{P}_*\circ\lambda_\ell$} is well-defined. Since $\BXinf$ is
the product $\sigma$-algebra, it suffices to show that, for all $n\geq \ell$ and
all $B\in\mathcal{B}(\Xcal^{n-1})$, 
\[
\PP{X_\ell\in A,X_{[n]\backslash\ell}\in B} =
\EE{(\widetilde{P}_*\circ\lambda_\ell)(A)\cdot\one{X_{[n]\backslash\ell}\in B}}.
\]

We have
\begin{align*}
\PP{X_\ell\in A,X_{[n]\backslash\ell}\in B} 
&=\EE{\one{X_\ell\in A}\cdot \one{X_{[n]\backslash\ell}\in B} \cdot \frac{\lambda_k(X_k)\lambda_\ell(X_\ell)}{\lambda_k(X_k)\lambda_\ell(X_\ell)}}\\
&=\EE{\one{(X^{k\ell})_\ell\in A}\cdot \one{(X^{k\ell})_{[n]\backslash\ell}\in B} \cdot \frac{\lambda_k((X^{k\ell})_k)\lambda_\ell((X^{k\ell})_\ell)}{\lambda_k(X_k)\lambda_\ell(X_\ell)}}\\
&=\EE{\one{X_k\in A}\cdot \one{(X^{k\ell})_{[n]\backslash\ell}\in B} \cdot \frac{\lambda_k(X_\ell)\lambda_\ell(X_k)}{\lambda_k(X_k)\lambda_\ell(X_\ell)}}\\
&=\EE{\EEst{\one{X_k\in A}\cdot \frac{\lambda_\ell(X_k)}{\lambda_k(X_k)}}{X_{-k}}\cdot  \one{(X^{k\ell})_{[n]\backslash\ell}\in B} \cdot \frac{\lambda_k(X_\ell)}{\lambda_\ell(X_\ell)}}\\
&=\EE{\Ep{X'\sim \widetilde{P}_*\circ\lambda_k}{ \one{X'\in A}\cdot \frac{\lambda_\ell(X')}{\lambda_k(X')}}\cdot  \one{(X^{k\ell})_{[n]\backslash\ell}\in B} \cdot \frac{\lambda_k(X_\ell)}{\lambda_\ell(X_\ell)}},
\end{align*}
where the second step holds by Proposition \ref{prop:wtd_exch_equiv} since
$X\sim Q$ where $Q$ is $\lambda$-weighted exchangeable, while the fourth step
holds since \smash{$(X^{k\ell})_{[n]\backslash\ell}$} does not depend on $X_k$
by construction, and the final step holds by \eqref{eq:P_lambdak_conditional}
with our choice of $k$.  Next, 
\begin{align*}
\Ep{X'\sim \widetilde{P}_*\circ\lambda_k}{ \one{X'\in A}\cdot \frac{\lambda_\ell(X')}{\lambda_k(X')}}
&=\frac{\int_\Xcal \one{x\in A}\cdot \lambda_\ell(x)\,\d\widetilde{P}_*(x)}{\int_\Xcal \lambda_k(x)\,\d\widetilde{P}_*(x)}\\
&=\frac{\int_\Xcal \one{x\in A}\cdot \lambda_\ell(x)\,\d\widetilde{P}_*(x)}{\int_\Xcal \lambda_\ell(x)\,\d\widetilde{P}_*(x)}\cdot\frac{\int_\Xcal \frac{\lambda_\ell(x)}{\lambda_k(x)}\cdot \lambda_k(x)\,\d\widetilde{P}_*(x)}{\int_\Xcal \lambda_k(x)\,\d\widetilde{P}_*(x)}\\
&=(\widetilde{P}_*\circ\lambda_\ell)(A) \cdot \Ep{X'\sim \widetilde{P}_*\circ\lambda_k}{\frac{\lambda_\ell(X')}{\lambda_k(X')}}
\end{align*}
almost surely, where we use the fact that \smash{$\int_\Xcal
\lambda_\ell(x)\,\d\widetilde{P}_*(x)\in(0,\infty)$} and thus
\smash{$\widetilde{P}_*\circ\lambda_\ell$} is well-defined, almost surely, by 
\eqref{eq:lambdak_defined} applied with $\ell$ in place of $k$.  Returning to
the calculations above, then,
\begin{align*}
\PP{X_\ell\in A,X_{[n]\backslash\ell}\in B} 
&=\EE{(\widetilde{P}_*\circ\lambda_\ell)(A) \cdot  \Ep{X'\sim \widetilde{P}_*\circ\lambda_k}{ \frac{\lambda_\ell(X')}{\lambda_k(X')}}\cdot  \one{(X^{k\ell})_{[n]\backslash\ell}\in B} \cdot \frac{\lambda_k(X_\ell)}{\lambda_\ell(X_\ell)}}\\
&=\EE{(\widetilde{P}_*\circ\lambda_\ell)(A) \cdot \EEst{ \frac{\lambda_\ell(X_k)}{\lambda_k(X_k)}}{X_{-k}}\cdot  \one{(X^{k\ell})_{[n]\backslash\ell}\in B} \cdot \frac{\lambda_k(X_\ell)}{\lambda_\ell(X_\ell)}},
\end{align*}
where the second step again applies \eqref{eq:P_lambdak_conditional} with our
choice of $k$.  Next, by Lemma \ref{lem:Ftail_measurable}, as in
\eqref{eqn:construct_f} we can construct an $\Ftail$-measurable function
$f:\Xcal^\infty\to[0,1]$, such that \smash{$f(X)\aseq
(\widetilde{P}_*\circ\lambda_\ell)(A)$}. We then have
\begin{align*}
\PP{X_\ell\in A,X_{[n]\backslash\ell}\in B} 
&=\EE{f(X) \cdot \EEst{ \frac{\lambda_\ell(X_k)}{\lambda_k(X_k)}}{X_{-k}}\cdot  \one{(X^{k\ell})_{[n]\backslash\ell}\in B} \cdot \frac{\lambda_k(X_\ell)}{\lambda_\ell(X_\ell)}}\\
&=\EE{f(X) \cdot   \one{(X^{k\ell})_{[n]\backslash\ell}\in B} \cdot  \frac{\lambda_k(X_\ell)\lambda_\ell(X_k)}{\lambda_k(X_k)\lambda_\ell(X_\ell)}}\\
&=\EE{f(X^{k\ell}) \cdot   \one{X_{[n]\backslash\ell}\in B} \cdot  \frac{\lambda_k(X_k)\lambda_\ell(X_\ell)}{\lambda_k(X_k)\lambda_\ell(X_\ell)}}\\
&=\EE{f(X) \cdot   \one{X_{[n]\backslash\ell}\in B}}\\
&=\EE{(\widetilde{P}_*\circ\lambda_\ell)(A) \cdot   \one{X_{[n]\backslash\ell}\in B}},
\end{align*}
where the third step holds by Proposition \ref{prop:wtd_exch_equiv}, and the
second and fourth steps hold since $f$ is $\Ftail$-measurable (and thus also
measurable with respect to $\sigma(X_{-k})\supseteq\Ftail$, for the second step
where we apply the tower law, and with respect to $\Ecal_\infty\supseteq\Ftail$
so that $f(X)=f(X^{k\ell})$, for the fourth step).  This completes the proof
that \eqref{eq:P_lambdak_conditional} holds with $\ell$ in place of $k$.
 
\subsection{Proof of Lemma \ref{lem:cond_Ptilde}}\label{app:cond_Ptilde}

First, we verify that, for sufficiently large $n$, it must hold that
\smash{$\sum_{j=\ell+1}^n \frac {\inf_{x\in\Xcal} \,
\lambda_j(x)/\lambda_*(x)}{\lambda_j(x')/\lambda_k(x')}>0$} for all $x'\in\Xcal$.
Indeed, we have
\begin{align*}
\sum_{j=\ell+1}^n \frac{\inf_{x\in\Xcal} \, \lambda_j(x)/\lambda_*(x)}{\lambda_j(x')/\lambda_k(x')}
&\geq \sum_{j=\ell+1}^n \frac{\inf_{x\in\Xcal} \, \lambda_j(x)/\lambda_*(x)}
{\sup_{x\in\Xcal} \, \lambda_j(x)/\lambda_k(x)}\\
&\geq \inf_{x\in\Xcal} \, \lambda_k(x)/\lambda_*(x)\cdot \sum_{j=\ell+1}^n \frac{\inf_{x\in\Xcal} \, \lambda_j(x)/\lambda_*(x)}
{\sup_{x\in\Xcal} \, \lambda_j(x)/\lambda_*(x)}.
\end{align*}
The first term is positive by our choice of $k$, while the second term is
positive for sufficiently large $n$ by our assumption that the sufficient
condition \eqref{eq:suff_condition} holds.

Next, fixing any sufficiently large $n>\ell$, we calculate 
\begin{align*}
&\PP{X_k \in A, X_{[\ell]\backslash k}\in B} 
= \EE{\one{X_k \in A}\cdot\one{X_{[\ell]\backslash k}\in B}}\\
&=\EE{\one{X_k \in A} \cdot \one{X_{[\ell]\backslash k}\in B} \cdot
\sum_{i=\ell+1}^n \frac{\frac
{\inf_{x\in\Xcal} \, \lambda_i(x)/\lambda_*(x)}{\lambda_i(X_i)/\lambda_k(X_i)}}
{\sum_{j=\ell+1}^n \frac
{\inf_{x\in\Xcal} \, \lambda_j(x)/\lambda_*(x)}{\lambda_j(X_j)/\lambda_k(X_j)}}} \\
&= \sum_{i=\ell+1}^n \EE{\one{X_k \in A} \cdot \one{X_{[\ell]\backslash k}\in B} 
\cdot \frac{\frac
{\inf_{x\in\Xcal} \, \lambda_i(x)/\lambda_*(x)}{\lambda_i(X_i)/\lambda_k(X_i)}}
{\sum_{j=\ell+1}^n \frac
{\inf_{x\in\Xcal} \, \lambda_j(x)/\lambda_*(x)}{\lambda_j(X_j)/\lambda_k(X_j)}} \cdot 
\frac{\lambda_i(X_i)\lambda_k(X_k)}{\lambda_i(X_i)\lambda_k(X_k)}} \\ 
&=\sum_{i=\ell+1}^n\EE{\one{(X^{ik})_k \in A} \cdot 
\one{(X^{ik})_{[\ell]\backslash k}\in B} \cdot 
\frac{\frac{\inf_{x\in\Xcal} \, \lambda_i(x)/\lambda_*(x)}{\lambda_i((X^{ik})_i)/\lambda_k((X^{ik})_i)}}
{\sum_{j=\ell+1}^n \frac
{\inf_{x\in\Xcal} \, \lambda_j(x)/\lambda_*(x)}{\lambda_j((X^{ik})_j)/\lambda_k((X^{ik})_j)}} \cdot 
\frac{\lambda_i((X^{ik})_i)\lambda_k((X^{ik})_k)}
{\lambda_i(X_i)\lambda_k(X_k)}},
\end{align*}
where in the last step we apply Proposition \ref{prop:wtd_exch_equiv}. 
As \smash{$(X^{ik})_{[\ell]\backslash k}=X_{[\ell]\backslash k}$} for any
$i>\ell$ by definition, we can simplify the above to
\begin{align*}
&\PP{X_k \in A, X_{[\ell]\backslash k}\in B} \\
&=\EE{\one{X_{[\ell]\backslash k}\in B} \cdot \sum_{i=\ell+1}^n   
\frac{\frac
{\inf_{x\in\Xcal} \, \lambda_i(x)/\lambda_*(x)}{\lambda_i(X_k)/\lambda_k(X_k)}\cdot \one{X_i \in A}}
{\left( \sum_{j=\ell+1}^n \frac
{\inf_{x\in\Xcal} \, \lambda_j(x)/\lambda_*(x)}{\lambda_j(X_j)/\lambda_k(X_j)} \right) + \Delta_i} \cdot 
\frac{\lambda_i(X_k)\lambda_k(X_i)}{\lambda_i(X_i)\lambda_k(X_k)} }\\ 
&= \EE{\one{X_{[\ell]\backslash k}\in B} \cdot \sum_{i=\ell+1}^n 
\frac{\frac
{\inf_{x\in\Xcal} \, \lambda_i(x)/\lambda_*(x)}{\lambda_i(X_i)/\lambda_k(X_i)} \cdot 
\one{X_i \in A}}{\left(\sum_{j=\ell+1}^n \frac
{\inf_{x\in\Xcal} \, \lambda_j(x)/\lambda_*(x)}{\lambda_j(X_j)/\lambda_k(X_j)} \right) + \Delta_i}},
\end{align*}
where 
\[
\Delta_i = \frac
{\inf_{x\in\Xcal} \, \lambda_i(x)/\lambda_*(x)}{\lambda_i(X_k)/\lambda_k(X_k)}  - \frac
{\inf_{x\in\Xcal} \, \lambda_i(x)/\lambda_*(x)}{\lambda_i(X_i)/\lambda_k(X_i)}.
\]
Recall the weighted empirical distribution $\widetilde{P}_n$ defined in
\eqref{eq:Ptilde}. We can calculate
\[
(\widetilde{P}_n\circ(\lambda_k/\lambda_*))(A)  
= \frac{ \sum_{i=1}^n \frac
{\inf_{x\in\Xcal} \, \lambda_i(x)/\lambda_*(x)}{ \lambda_i(X_i)/\lambda_k(X_i)} \cdot \one{X_i\in A}} 
{\sum_{j=1}^n \frac
{\inf_{x\in\Xcal} \, \lambda_j(x)/\lambda_*(x)}{ \lambda_j(X_j)/\lambda_k(X_j)}},
\]
and so 
\begin{multline*}
\left| (\widetilde{P}_n\circ(\lambda_k/\lambda_*))(A) - 
\sum_{i=\ell+1}^n \frac{\frac
{\inf_{x\in\Xcal} \, \lambda_i(x)/\lambda_*(x)}{\lambda_i(X_i)/\lambda_k(X_i)}\cdot \one{X_i \in A}}
{\left( \sum_{j=\ell+1}^n \frac
{\inf_{x\in\Xcal} \, \lambda_j(x)/\lambda_*(x)}{\lambda_j(X_j)/\lambda_k(X_j)} \right) +\Delta_i} \right| \\ 
=\left|\sum_{i=1}^n \frac{\frac
{\inf_{x\in\Xcal} \, \lambda_i(x)/\lambda_*(x)}{\lambda_i(X_i)/\lambda_k(X_i)} \cdot \one{X_i \in A}} 
{\sum_{j=1}^n \frac
{\inf_{x\in\Xcal} \, \lambda_j(x)/\lambda_*(x)}{\lambda_j(X_j)/\lambda_k(X_j)}} -
\sum_{i=\ell+1}^n \frac{\frac
{\inf_{x\in\Xcal} \, \lambda_i(x)/\lambda_*(x)}{\lambda_i(X_i)/\lambda_k(X_i)} \cdot \one{X_i \in A}}
{\left(\sum_{j=\ell+1}^n \frac
{\inf_{x\in\Xcal} \, \lambda_j(x)/\lambda_*(x)}{\lambda_j(X_j)/\lambda_k(X_j)} \right) +\Delta_i} \right| \\ 
\leq \frac{\ell+1}{\left(\frac{\inf_{x\in\Xcal} \, \lambda_k(x)/\lambda_*(x)}
{\sup_{x\in\Xcal} \, \lambda_k(x)/\lambda_*(x)} \cdot
\sum_{j=\ell+1}^n \frac{\inf_{x\in\Xcal} \, \lambda_j(x)/\lambda_*(x)}
{\sup_{x\in\Xcal} \, \lambda_j(x)/\lambda_*(x)} -1\right)_+},
\end{multline*}
where the last step follows from some direct calculations, making use of the
fact that
\[
|\Delta_i| \leq \sup_{x\in\Xcal} \, \lambda_k(x)/\lambda_*(x)
\]
holds surely for every $i$. Therefore,
\begin{align*}
&\left| \PP{X_k \in A, X_{[\ell]\backslash k}\in B} - 
\EE{(\widetilde{P}_n\circ(\lambda_k/\lambda_*))(A) \cdot 
\one{X_{[\ell]\backslash k}\in B}} \right|\\
&= \left|\EE{\one{X_{[\ell]\backslash k}\in B} \cdot  
\left(\sum_{i=\ell+1}^n \frac{\frac
{\inf_{x\in\Xcal} \, \lambda_i(x)/\lambda_*(x)}{\lambda_i(X_i)/\lambda_k(X_i)} \cdot \one{X_i \in A}}
{\sum_{j=\ell+1}^n \frac
{\inf_{x\in\Xcal} \, \lambda_j(x)/\lambda_*(x)}{\lambda_j(X_j)/\lambda_k(X_j)} +\Delta_i}  -
(\widetilde{P}_n\circ(\lambda_k/\lambda_*))(A)\right)} \right| \\
&\leq \EE{\one{X_{[\ell]\backslash k}\in B} \cdot  
\left|\sum_{i=\ell+1}^n \frac{\frac
{\inf_{x\in\Xcal} \, \lambda_i(x)/\lambda_*(x)}{\lambda_i(X_i)/\lambda_k(X_i)} \cdot \one{X_i \in A}}
{\sum_{j=\ell+1}^n \frac
{\inf_{x\in\Xcal} \, \lambda_j(x)/\lambda_*(x)}{\lambda_j(X_j)/\lambda_k(X_j)} +\Delta_i} -
(\widetilde{P}_n\circ(\lambda_k/\lambda_*))(A)\right|} \\
&\leq \frac{\ell+1}{\left(\frac{\inf_{x\in\Xcal}\lambda_k(x)/\lambda_*(x)}
{\sup_{x\in\Xcal} \, \lambda_k(x)/\lambda_*(x)}  \cdot 
\sum_{j=\ell+1}^n \frac{\inf_{x\in\Xcal}\,\lambda_*(x)/\lambda_j(x)}
{\sup_{x\in\Xcal} \, \lambda_*(x)/\lambda_j(x)} -1\right)_+}.
\end{align*}
By our choice of $k$, we know that the first factor in the denominator above is
positive, and by \eqref{eq:suff_condition},
\[
\lim_{n\to\infty} \sum_{j=\ell+1}^n 
\frac{\inf_{x\in\Xcal} \, \lambda_*(x)/\lambda_j(x)}
{\sup_{x\in\Xcal} \, \lambda_*(x)/\lambda_j(x)}
=\infty.
\]
Therefore,
\[
\lim_{n\to\infty} \left| \PP{X_k \in A, X_{[\ell]\backslash k}\in B} - 
\EE{(\widetilde{P}_n\circ(\lambda_k/\lambda_*))(A)\cdot 
\one{X_{[\ell]\backslash k}\in B}}\right| = 0,
\]
which completes the proof.

\section{Proofs supporting auxiliary results}

\subsection{Proof of Proposition \ref{prop:Xcal_size_at_least_3}
}\label{sec:proof_prop:Xcal_size_at_least_3}

Since $|\Xcal|\geq 3$, we can choose a partition $\Xcal = \Xcal_0\cup
\Xcal_1\cup\Xcal_2$ for some nonempty $\Xcal_0,\Xcal_1,\Xcal_2\in\BX$. Now
define 
\[
\lambda_i(x) = \begin{cases} 
e^{-i} & x\in \Xcal_{\textnormal{mod}(i,3)}, \\ 
1 & x \not\in \Xcal_{\textnormal{mod}(i,3)},
\end{cases} \quad i \geq 1.
\]
First we examine the sufficient condition \eqref{eq:suff_condition}.
Let $\lambda_*:\Xcal\to\R_+$ be any measurable function. 
Fix any $x_\ell\in\Xcal_\ell$ for each $\ell = 0,1,2$. Then we have
\begin{multline*}
\sum_{i=1}^\infty \frac{\inf_{x\in\Xcal}\lambda_i(x)/\lambda_*(x)}
{\sup_{x\in\Xcal}\lambda_i(x)/\lambda_*(x)}
\leq \sum_{i=1}^\infty \frac{
\min_{\ell=0,1,2}\lambda_i(x_\ell)/\lambda_*(x_\ell)}
{\max_{\ell=0,1,2}\lambda_i(x_\ell)/\lambda_*(x_\ell)}\\
\leq \frac{\max_{\ell=0,1,2}\lambda_*(x_\ell)}
{\min_{\ell=0,1,2}\lambda_*(x_\ell)}
\cdot \sum_{i=1}^\infty \frac{\min_{\ell=0,1,2}\lambda_i(x_\ell)}
{\max_{\ell=0,1,2}\lambda_i(x_\ell)}
= \frac{\max_{\ell=0,1,2}\lambda_*(x_\ell)}
{\min_{\ell=0,1,2}\lambda_*(x_\ell)}
\cdot \sum_{i=1}^\infty e^{-i} < \infty,
\end{multline*}
and therefore the sufficient condition \eqref{eq:suff_condition} is not
satisfied. Now we check the necessary condition \eqref{eq:nec_condition}. By
construction, we can see that $\Mcal_\Xcal(\lambda) = \{P\in\Mcal(\Xcal):
0<P(\Xcal)<\infty\}$. Fix any $A,A^c\in\BX$ with $P(A),P(A^c)>0$. Since $\Xcal =
\Xcal_0\cup \Xcal_1\cup\Xcal_2$ is a partition, we can choose some
$\ell_A\in\{0,1,2\}$ such that \smash{$P(A\cap \Xcal_{\ell_A}) >0$}, and
similarly some \smash{$\ell_{A^c}\in\{0,1,2\}$} such that \smash{$P(A^c\cap
  \Xcal_{\ell_{A^c}})>0$}.  

Now let \smash{$\ell\in\{0,1,2\}\backslash\{\ell_A,\ell_{A^c}\}$}, so that for
any $i$ with $\textnormal{mod}(i,3) = \ell$, we have $\lambda_i(x) = 1$ for 
\smash{$x\in\Xcal_{\ell_A}$} and also for \smash{$x\in\Xcal_{\ell_{A^c}}$}. Then
for any such $i$, we calculate
\[
(P\circ\lambda_i)(A)
= \frac{\int_A \lambda_i(x)\,\d{P}(x)}{\int_{\Xcal} \lambda_i(x)\,\d{P}(x)} 
\geq \frac{\int_{A\cap \Xcal_{\ell_A}} \lambda_i(x)\,\d{P}(x)}
{\int_{\Xcal}\lambda_i(x)  \,\d{P}(x)} 
\geq \frac{\int_{A\cap\Xcal_{\ell_A}} \,\d{P}(x)}{\int_{\Xcal} \,\d{P}(x)} 
= \frac{P(A\cap \Xcal_{\ell_A})}{P(\Xcal)}>0,
\]
where the third step holds since $\lambda_i(x)=1$ for
\smash{$x\in\Xcal_{\ell_A}$} by the choice of $i$, and $\lambda_i(x)\leq 1$ for
all $x$. Similarly, for any $i$ with $\textnormal{mod}(i,3) = \ell$, 
\[
(P\circ\lambda_i)(A^c) \geq  \frac{P(A^c\cap \Xcal_{\ell_{A^c})}}{P(\Xcal)}>0.
\]
Therefore, we have
\[
\sum_{i=1}^\infty\min\{(P\circ\lambda_i)(A),(P\circ\lambda_i)(A^c\}
\geq \sum_{i=1}^\infty\one{\textnormal{mod}(i,3)=\ell}\cdot 
\min\left\{ \frac{P(A\cap \Xcal_{\ell_A)}}{P(\Xcal)}, 
\frac{P(A^c\cap \Xcal_{\ell_{A^c})}}{P(\Xcal)}\right\} = \infty,
\]
proving that the condition \eqref{eq:nec_condition} is satisfied.

\subsection{Proof of Theorem \ref{thm:finite_case}}
\label{sec:proof_thm:finite_case}

Once we apply the results of Theorems \ref{thm:main} and \ref{thm:nec_condition}, 
we see that we only need to verify that
\[
\Lambda_{\dF}\supseteq\{\textnormal{$\lambda$
  satisfying \eqref{eq:nec_condition}}\},
\]
i.e., that the necessary condition \eqref{eq:nec_condition} is in fact
sufficient for proving $\lambda\in\Lambda_{\dF}$, for the finite case. 

\subsubsection{An equivalent characterization via a graph} 

First, we give an equivalent characterization of the
necessary condition for this finite case. For any nonempty subset
$S\subseteq\Xcal$, define an undirected graph on vertices $S$, denoted by
$G_S=(S,E_S)$, where the set of edges is given by 
\begin{equation}\label{eqn:edges_graph}
E_S = \left\{ (x_0,x_1)\in S\times S \, : \, \sum_{i=1}^\infty \frac{\min\{\lambda_i(x_0),\lambda_i(x_1)\}}{\max_{x\in S} \lambda_i(x)} = \infty\right\}.\end{equation} 
\begin{lemma}\label{lem:nec_equiv_finite_case}
If $|\Xcal|<\infty$ then, for any $\lambda\in\Lambda^\infty$, 
$\lambda$ satisfies the  condition \eqref{eq:nec_condition} if and only if, for
every nonempty $S\subseteq\Xcal$, $G_S$ is a connected graph.
\end{lemma}

From this point onward,
we assume that $\lambda$ satisfies \eqref{eq:nec_condition}, so that by the above lemma, $G_S$ is connected for each nonempty $S\subseteq\Xcal$.
For each  nonempty $S\subseteq\Xcal$, we fix a \emph{directed} and \emph{rooted}
spanning tree \smash{$E^*_S\subseteq E_S$}, and will work with the rooted
directed acyclic graph (DAG) \smash{$G^*_S = (S,E^*_S)$}. 

\subsubsection{Writing $Q$ via a mixture of supports} 

Next for any $x\in\Xcal^\infty$ define its support
\[
S(x) = \left\{x\in\Xcal : \sum_{i=1}^\infty\one{x_i=x} >0 \right\},
\]
and for each nonempty $S\subseteq\Xcal$ define
\[
p_S = \Pp{Q}{S(X) = S},
\]
the probability that $S(X)=S$ when we sample $X\sim Q$. We can write $Q$ as a
mixture, 
\[
Q = \sum_{S : p_S>0} p_S \cdot Q_S,
\] 
where $Q_S$ is the distribution $Q$ conditional on the event $S(X)=S$.

For any $S$ with $p_S>0$, as the event $\{S(X)=S\}$ is in $\Ecal_\infty$, $Q_S$
is also $\lambda$-weighted exchangeable (this can be verified exactly as in the
proof of Theorem \ref{thm:main}---specifically, the proof that
$\Lambda_{\dF}\subseteq\Lambda_{\01}$). In what follows, we will show that, 
for each $S$, $Q_S$ has a (weighted) de Finetti representation. As $Q$ is
ultimately a mixture of distributions $Q_S$ over subsets $S$, this will be
sufficient: if each $Q_S$ can be written as a mixture of $\lambda$-weighted
i.i.d.\ distributions, then the same is also true for $Q$.   

\subsubsection{Defining the probability measure} 

From this point on, we will work with $Q_S$ for a specific subset
$S\subseteq\Xcal$, and all probabilities and expectations should be interpreted
as taken with respect to $Q_S$, unless otherwise specified. Next we need another
lemma.

\begin{lemma}\label{lem:infinite_sum_for_finite_case}
Under the notation and assumptions above, for any \smash{$(x,x')\in E^*_S$}, 
\[
\sum_{i\geq
  1}\frac{\min\{\lambda_i(x),\lambda_i(x')\}}{\lambda_i(X_i)}\cdot\one{X_i\in\{x,x'\}}
\, = \, \infty
\]
almost surely under $Q_S$.
\end{lemma}

Next, \smash{for each $(x,x')\in E^*_S$}, define
\[
\widetilde{P}_n(x;x') = \frac{\sum_{i=1}^n
\frac{\min\{\lambda_i(x),\lambda_i(x')\}}{\lambda_i(X_i)}\cdot\one{X_i=x}}{\sum_{i=1}^n 
\frac{\min\{\lambda_i(x),\lambda_i(x')\}}{\lambda_i(X_i)}\cdot\one{X_i \in\{x,x'\}}}.
\]
By Lemma \ref{lem:infinite_sum_for_finite_case}, the denominator tends  to
infinity almost surely as $n\to\infty$. In particular, it holds almost surely
that the denominator is positive for sufficiently large $n$. In other words, for
sufficiently large $n$, \smash{$\widetilde{P}_n(x;x')$} is well-defined.  

Next we need to verify that this quantity converges. As in the proof of Theorem
\ref{thm:suff_condition}, the following lemma will be proved by extracting
an exchangeable subsequence.

\begin{lemma}\label{lem:convergence_for_finite_case}
Under the notation and assumptions above, for any \smash{$(x,x')\in E^*_S$},
there is a $\Ftail$-measurable random variable \smash{$\widetilde{P}(x;x')\in
  (0,1)$} such that 
\[
\widetilde{P}(x;x') = \lim_{n\to\infty} \widetilde{P}_n(x;x')
\]
almost surely under $Q_S$.
\end{lemma}

Next, we will convert these values to a probability distribution
\smash{$\{\widetilde{P}(x)\}_{x\in\Xcal}$}, which again is $\Ftail$-measurable. 
Let $x_*$ be the root of the rooted DAG \smash{$G^*_S$}. For $x\in S$, consider
the unique path in \smash{$G^*_S$}, from the root $x_*$ to $x$: this path can be
expressed as  a finite sequence $y_0(x) = x_*, y_1(x),\dots, y_{k(x)}(x) = x$
such that \smash{$(y_{i-1}(x),y_i(x))\in E^*_S$} for each $i$. 
Then define 
\[
\widetilde{P}'(x) =  \prod_{i=1}^{k(x)}
\frac{1-\widetilde{P}(y_{i-1}(x);y_i(x))}{\widetilde{P}(y_{i-1}(x);y_i(x))},
\]
where for the case $x=x_*$ (and so $k(x)=k(x_*)=0$), the empty product should be
interpreted as $1$, i.e., \smash{$\widetilde{P}'(x_*)=1$}. Note that, since
\smash{$\widetilde{P}(x;x')\in(0,1)$} for all \smash{$(x,x')\in E^*_S$},
therefore the above product is positive and finite for all $x\in S$. 
Then define
\[
\widetilde{P}(x) = \begin{cases}\frac{\widetilde{P}'(x)}{\sum_{x'\in
      S}\widetilde{P}'(x')} & x\in S,\\ 0 & x\not\in S.
\end{cases}\]

Now we will see the motivation for this construction: for any \smash{$(x,x')\in
  E^*_S$}, we now verify that  
\begin{equation}\label{eqn:finite_case_ratio_property}
\frac{\widetilde{P}(x)}{\widetilde{P}(x)+\widetilde{P}(x')} =
\widetilde{P}(x;x').
\end{equation}
In other words, \smash{$\widetilde{P}(x;x')$} defines the conditional
probability \smash{$\Ppst{\widetilde{P}}{X=x}{X\in\{x,x'\}}$}, for each edge  
\smash{$(x,x')\in E^*_S$} in the rooted DAG. To see why this holds, first note
that since \smash{$G^*_S$} is a rooted DAG, the path from the root to $x$, and
the path from the root to $x'$, are each unique. Recall that
\[
y_0(x) = x_* , y_1(x), \dots, y_{k(x)}(x) = x
\]
is the path from $x_*$ to $x$. Then the unique path from the root to $x'$ must
therefore be equal to 
\[y_0(x) = x_* , y_1(x), \dots, y_{k(x)}(x) = x,  y_{k(x)+1}(x) = x'.\]
In other words, we have $k(x') = k(x)+1$, and
\[
y_i(x') = \begin{cases} y_i(x) & i =1,\dots, k(x),\\ x' & i = k(x)+1 =
  k(x').\end{cases}
\]
We can therefore calculate
\begin{multline*}
\widetilde{P}'(x') =  \prod_{i=1}^{k(x')} \frac{1-\widetilde{P}(y_{i-1}(x');y_i(x'))}{\widetilde{P}(y_{i-1}(x');y_i(x'))}\\
=  \left[\prod_{i=1}^{k(x')-1 } \frac{1-\widetilde{P}(y_{i-1}(x');y_i(x'))}{\widetilde{P}(y_{i-1}(x');y_i(x'))}\right] \cdot \frac{1-\widetilde{P}(y_{k(x')-1}(x');y_{k(x')}(x'))}{\widetilde{P}(y_{k(x')-1}(x');y_{k(x')}(x'))} \\
=  \left[\prod_{i=1}^{k(x)} \frac{1-\widetilde{P}(y_{i-1}(x);y_i(x))}{\widetilde{P}(y_{i-1}(x);y_i(x))}\right] \cdot \frac{1-\widetilde{P}(x;x')}{\widetilde{P}(x;x')}
= \widetilde{P}'(x)\cdot  \frac{1-\widetilde{P}(x;x')}{\widetilde{P}(x;x')},
\end{multline*}
which proves the desired equality as we then have
\[
\frac{\widetilde{P}(x)}{\widetilde{P}(x) + \widetilde{P}(x')}=\frac{\widetilde{P}'(x)}{\widetilde{P}'(x) + \widetilde{P}'(x')}
= \frac{\widetilde{P}'(x)}{\widetilde{P}'(x) +\widetilde{P}'(x)\cdot  \frac{1-\widetilde{P}(x;x')}{\widetilde{P}(x;x')}} = \widetilde{P}(x;x').\]

\subsubsection{Representing $Q_S$ as a mixture distribution}

Finally, we need to check that it holds that $Q_S$ is equal to the mixture
defined by \smash{$\widetilde{P}\circ\lambda$}, for the random measure
\smash{$\widetilde{P}$}.  In fact, as in the proof of Theorem
\ref{thm:suff_condition}, it is sufficient to verify that 
\[
\PPst{X_k = x}{X_{-k}} \aseq (\widetilde{P}\circ\lambda_k)(x)
\]
for all $k$ and all $x\in S$ (this claim is analogous to the claim
\eqref{eq:P_lambdak_conditional} in Lemma \ref{lem:check_P_circ_lambda_k}; the
argument that appears after  Lemma \ref{lem:check_P_circ_lambda_k}, in the proof
of Theorem \ref{thm:suff_condition}, explains why verifying this claim is
sufficient). Since $G_S$ is a connected graph, it is sufficient to check that 
\[
\frac{\PPst{X_k = x}{X_{-k}}}{\PPst{X_k = x'}{X_{-k}}} \aseq
\frac{\widetilde{P}(x) \lambda_k(x)}{\widetilde{P}(x') \lambda_k(x')},
\]
for all $k$ and for all \smash{$(x,x')\in E^*_S$}, or equivalently, 
\[
\frac{\PPst{X_k = x}{X_{-k}}}{\PPst{X_k = x'}{X_{-k}}} \overset{\textnormal{a.s.}}{\geq}
\frac{\widetilde{P}(x) \lambda_k(x)}{\widetilde{P}(x') \lambda_k(x')},
\textnormal{ \ and \ \ }\frac{\PPst{X_k = x}{X_{-k}}}{\PPst{X_k = x'}{X_{-k}}} \overset{\textnormal{a.s.}}{\leq}
\frac{\widetilde{P}(x) \lambda_k(x)}{\widetilde{P}(x') \lambda_k(x')}.\]
Since the proofs of these two inequalities are essentially identical, we will prove only the first.
By the property
\eqref{eqn:finite_case_ratio_property} calculated above, equivalently, we need
to check that 
\[
\PPst{X_k = x}{X_{-k}}  \overset{\textnormal{a.s.}}{\geq} \frac{\widetilde{P}(x;x')}{1-\widetilde{P}(x;x') }
\cdot \frac{\lambda_k(x)}{\lambda_k(x')} \cdot \PPst{X_k = x'}{X_{-k}},
\]
for all $k$ and for all \smash{$(x,x')\in E^*_S$}. (Note that
\smash{$\widetilde{P}(x;x')$} is $\Ftail$-measurable, and is therefore 
measurable with respect to $\sigma(X_{-k})$.) By definition of conditional
probability, this is equivalent to checking that 
\[
\PP{X_k = x,X_{-k} \in A} \geq\frac{\lambda_k(x)}{\lambda_k(x')} \cdot  \EE{
  \frac{\widetilde{P}(x;x')}{1-\widetilde{P}(x;x') }\cdot \one{X_k=x'}
  \one{X_{-k} \in A}}
\]
for all $A\in\BXinf$. By construction of the product $\sigma$-algebra, it is
sufficient to check that, for every $\ell\geq k$ and any
$A\in\Bcal(\Xcal^{\ell-1})$, 
\begin{equation}\label{eqn:a_s_ineq}
\PP{X_k = x,X_{[\ell]\backslash k} \in A} \geq \frac{\lambda_k(x)}{\lambda_k(x')}
\cdot \EE{ \frac{\widetilde{P}(x;x')}{1-\widetilde{P}(x;x')} \cdot\one{X_k=x'}
  \one{X_{[\ell]\backslash k} \in A}}.
\end{equation}
We can derive that
\[
\frac{\widetilde{P}(x;x')}{1-\widetilde{P}(x;x')} =
\lim_{n\to\infty}\frac{\sum_{i=1}^n
  \frac{\min\{\lambda_i(x),\lambda_i(x')\}}{\lambda_i(X_i)}\cdot\one{X_i=x}}{\sum_{i=1}^n
  \frac{\min\{\lambda_i(x),\lambda_i(x')\}}{\lambda_i(X_i)}\cdot\one{X_i=x'}}=
\lim_{n\to\infty}\frac{\sum_{i=\ell+1}^n
  \frac{\min\{\lambda_i(x),\lambda_i(x')\}}{\lambda_i(X_i)}\cdot\one{X_i=x}}{1 +
  \sum_{i=\ell+1}^n
  \frac{\min\{\lambda_i(x),\lambda_i(x')\}}{\lambda_i(X_i)}\cdot\one{X_i=x'}},
\]
 almost surely, where the first step holds from Lemma \ref{lem:convergence_for_finite_case},
while the second step is true since, by Lemmas
\ref{lem:infinite_sum_for_finite_case} and
\ref{lem:convergence_for_finite_case}, the sums in the numerator and denominator
both diverge. By Fatou's lemma,
\begin{multline*}
\EE{ \frac{\widetilde{P}(x;x')}{1-\widetilde{P}(x;x') } \cdot\one{X_k=x'} \one{X_{[\ell]\backslash k} \in A}}
\\\leq \lim\inf_{n\to\infty}\EE{\frac{\sum_{i=\ell+1}^n
    \frac{\min\{\lambda_i(x),\lambda_i(x')\}}{\lambda_i(X_i)}\cdot\one{X_i=x}}{1
    + \sum_{i=\ell+1}^n
    \frac{\min\{\lambda_i(x),\lambda_i(x')\}}{\lambda_i(X_i)}\cdot\one{X_i=x'}}\cdot\one{X_k=x'}
  \one{X_{[\ell]\backslash k} \in A}}.
\end{multline*}
For each $n$, we calculate
\begin{align*}
&\EE{\frac{\sum_{i=\ell+1}^n  \frac{\min\{\lambda_i(x),\lambda_i(x')\}}{\lambda_i(X_i)}\cdot\one{X_i=x}}{1 + \sum_{i=\ell+1}^n  \frac{\min\{\lambda_i(x),\lambda_i(x')\}}{\lambda_i(X_i)}\cdot\one{X_i=x'}}\cdot\one{X_k=x'} \one{X_{[\ell]\backslash k} \in A}}\\
&=\sum_{i=\ell+1}^n\EE{\frac{\frac{\min\{\lambda_i(x),\lambda_i(x')\}}{\lambda_i(X_i)}\cdot\one{X_i=x}}{1 + \sum_{i'=\ell+1}^n  \frac{\min\{\lambda_{i'}(x),\lambda_{i'}(x')\}}{\lambda_{i'}(X_{i'})}\cdot\one{X_{i'}=x'}}\cdot\one{X_k=x'} \one{X_{[\ell]\backslash k} \in A} \cdot \frac{\lambda_i(X_i)\lambda_k(X_k)}{\lambda_i(X_i)\lambda_k(X_k)}}\\
&=\sum_{i=\ell+1}^n\EE{\frac{\frac{\min\{\lambda_i(x),\lambda_i(x')\}}{\lambda_i(X_k)}\cdot\one{X_k=x}}{1 + \sum_{i'=\ell+1}^n  \frac{\min\{\lambda_{i'}(x),\lambda_{i'}(x')\}}{\lambda_{i'}((X^{ik})_{i'})}\cdot\one{(X^{ik})_{i'}=x'}}\cdot\one{X_i=x'} \one{(X^{ik})_{[\ell]\backslash k} \in A}\frac{\lambda_i(X_k)\lambda_k(X_i)}{\lambda_i(X_i)\lambda_k(X_k)}}\\
&=\sum_{i=\ell+1}^n\EE{\frac{\frac{\min\{\lambda_i(x),\lambda_i(x')\}}{\lambda_i(X_i)}\cdot\one{X_k=x}}{1 + \sum_{i'=\ell+1}^n  \frac{\min\{\lambda_{i'}(x),\lambda_{i'}(x')\}}{\lambda_{i'}((X^{ik})_{i'})}\cdot\one{(X^{ik})_{i'}=x'}}\cdot\one{X_i=x'} \one{X_{[\ell]\backslash k} \in A}} \cdot \frac{\lambda_k(x')}{\lambda_k(x)}\\
&=\sum_{i=\ell+1}^n\EE{\frac{\frac{\min\{\lambda_i(x),\lambda_i(x')\}}{\lambda_i(X_i)}\cdot\one{X_i=x'}}{1 + \sum_{i'=\ell+1}^n  \frac{\min\{\lambda_{i'}(x),\lambda_{i'}(x')\}}{\lambda_{i'}(X_{i'})}\cdot\one{X_{i'}=x'} + \Delta_i}\cdot\one{X_k=x} \one{X_{[\ell]\backslash k} \in A}}\cdot \frac{\lambda_k(x')}{\lambda_k(x)},
\end{align*}
where the second step holds by Proposition \ref{prop:wtd_exch_equiv}, and where 
\[
\Delta_i =
\frac{\min\{\lambda_i(x),\lambda_i(x')\}}{\lambda_i(x')}\cdot(\one{X_k=x'} -
\one{X_i = x'}) \in [-1,1].
\]
Therefore,
\begin{align*}
&\left|\frac{\lambda_k(x)}{\lambda_k(x')}\EE{\frac{\sum_{i=\ell+1}^n  \frac{\min\{\lambda_i(x),\lambda_i(x')\}}{\lambda_i(X_i)}\cdot\one{X_i=x}}{1 + \sum_{i=\ell+1}^n  \frac{\min\{\lambda_i(x),\lambda_i(x')\}}{\lambda_i(X_i)}\cdot\one{X_i=x'}}\cdot\one{X_k=x'} \one{X_{[\ell]\backslash k} \in A}} - \PP{X_k=x,X_{[\ell]\backslash k} \in A}\right|\\
&=\left| \sum_{i=\ell+1}^n\EE{\frac{\frac{\min\{\lambda_i(x),\lambda_i(x')\}}{\lambda_i(X_i)}\cdot\one{X_i=x'}}{1 + \sum_{i'=\ell+1}^n  \frac{\min\{\lambda_{i'}(x),\lambda_{i'}(x')\}}{\lambda_{i'}(X_{i'})}\cdot\one{X_{i'}=x'} + \Delta_i}\cdot\one{X_k=x} \one{X_{[\ell]\backslash k} \in A}} - \PP{X_k=x,X_{[\ell]\backslash k} \in A}\right|\\
&\leq \EE{\left| \sum_{i=\ell+1}^n\frac{\frac{\min\{\lambda_i(x),\lambda_i(x')\}}{\lambda_i(X_i)}\cdot\one{X_i=x'}}{1 + \sum_{i'=\ell+1}^n  \frac{\min\{\lambda_{i'}(x),\lambda_{i'}(x')\}}{\lambda_{i'}(X_{i'})}\cdot\one{X_{i'}=x'} + \Delta_i}\cdot\one{X_k=x} \one{X_{[\ell]\backslash k} \in A} - \one{X_k=x,X_{[\ell]\backslash k} \in A}\right|}\\
&\leq \EE{\left| \sum_{i=\ell+1}^n\frac{\frac{\min\{\lambda_i(x),\lambda_i(x')\}}{\lambda_i(X_i)}\cdot\one{X_i=x'}}{1 + \sum_{i'=\ell+1}^n  \frac{\min\{\lambda_{i'}(x),\lambda_{i'}(x')\}}{\lambda_{i'}(X_{i'})}\cdot\one{X_{i'}=x'} + \Delta_i}-1\right|}.
\end{align*}
Now we bound the sum inside the expectation. If
$ \sum_{i'=\ell+1}^n  \frac{\min\{\lambda_{i'}(x),\lambda_{i'}(x')\}}{\lambda_{i'}(X_{i'})}\cdot\one{X_{i'}=x'}=0$ then the sum is zero.
If not, then since $\Delta_i \in[-1,1]$ by construction, the sum is upper bounded as
\begin{multline*}\sum_{i=\ell+1}^n\frac{\frac{\min\{\lambda_i(x),\lambda_i(x')\}}{\lambda_i(X_i)}\cdot\one{X_i=x'}}{1 + \sum_{i'=\ell+1}^n  \frac{\min\{\lambda_{i'}(x),\lambda_{i'}(x')\}}{\lambda_{i'}(X_{i'})}\cdot\one{X_{i'}=x'} + \Delta_i} \\\leq \sum_{i=\ell+1}^n\frac{\frac{\min\{\lambda_i(x),\lambda_i(x')\}}{\lambda_i(X_i)}\cdot\one{X_i=x'}}{1 + \sum_{i'=\ell+1}^n  \frac{\min\{\lambda_{i'}(x),\lambda_{i'}(x')\}}{\lambda_{i'}(X_{i'})}\cdot\one{X_{i'}=x'} + (-1)}  = 1,\end{multline*}
and is lower bounded as
\begin{multline*}\sum_{i=\ell+1}^n\frac{\frac{\min\{\lambda_i(x),\lambda_i(x')\}}{\lambda_i(X_i)}\cdot\one{X_i=x'}}{1 + \sum_{i'=\ell+1}^n  \frac{\min\{\lambda_{i'}(x),\lambda_{i'}(x')\}}{\lambda_{i'}(X_{i'})}\cdot\one{X_{i'}=x'} + \Delta_i} \geq \sum_{i=\ell+1}^n\frac{\frac{\min\{\lambda_i(x),\lambda_i(x')\}}{\lambda_i(X_i)}\cdot\one{X_i=x'}}{1 + \sum_{i'=\ell+1}^n  \frac{\min\{\lambda_{i'}(x),\lambda_{i'}(x')\}}{\lambda_{i'}(X_{i'})}\cdot\one{X_{i'}=x'} + 1}\\  = 
\frac{\sum_{i=\ell+1}^n\frac{\min\{\lambda_i(x),\lambda_i(x')\}}{\lambda_i(X_i)}\cdot\one{X_i=x'}}{2 + \sum_{i'=\ell+1}^n  \frac{\min\{\lambda_{i'}(x),\lambda_{i'}(x')\}}{\lambda_{i'}(X_{i'})}\cdot\one{X_{i'}=x'} } = 1 - \frac{2}{2 + \sum_{i'=\ell+1}^n  \frac{\min\{\lambda_{i'}(x),\lambda_{i'}(x')\}}{\lambda_{i'}(X_{i'})}\cdot\one{X_{i'}=x'} }.\end{multline*}
In all cases, then, we have
\[\left| \sum_{i=\ell+1}^n\frac{\frac{\min\{\lambda_i(x),\lambda_i(x')\}}{\lambda_i(X_i)}\cdot\one{X_i=x'}}{1 + \sum_{i'=\ell+1}^n  \frac{\min\{\lambda_{i'}(x),\lambda_{i'}(x')\}}{\lambda_{i'}(X_{i'})}\cdot\one{X_{i'}=x'} + \Delta_i}-1\right| \leq \frac{2}{2 + \sum_{i'=\ell+1}^n  \frac{\min\{\lambda_{i'}(x),\lambda_{i'}(x')\}}{\lambda_{i'}(X_{i'})}\cdot\one{X_{i'}=x'} }.\]
Therefore,
\begin{align*}
&\left|\frac{\lambda_k(x)}{\lambda_k(x')} \lim_{n\to\infty}\EE{\frac{\sum_{i=\ell+1}^n  \frac{\min\{\lambda_i(x),\lambda_i(x')\}}{\lambda_i(X_i)}\cdot\one{X_i=x}}{1 + \sum_{i=\ell+1}^n  \frac{\min\{\lambda_i(x),\lambda_i(x')\}}{\lambda_i(X_i)}\cdot\one{X_i=x'}}}
- \PP{X_k=x,X_{[\ell]\backslash k} \in A}\right|\\
&=\lim_{n\to\infty}\left|\frac{\lambda_k(x)}{\lambda_k(x')}\EE{\frac{\sum_{i=\ell+1}^n  \frac{\min\{\lambda_i(x),\lambda_i(x')\}}{\lambda_i(X_i)}\cdot\one{X_i=x}}{1 + \sum_{i=\ell+1}^n  \frac{\min\{\lambda_i(x),\lambda_i(x')\}}{\lambda_i(X_i)}\cdot\one{X_i=x'}}\cdot\one{X_k=x'} \one{X_{[\ell]\backslash k} \in A}} - \PP{X_k=x,X_{[\ell]\backslash k} \in A}\right|\\
&\leq \lim_{n\to\infty} \EE{ \frac{2}{2 + \sum_{i'=\ell+1}^n  \frac{\min\{\lambda_{i'}(x),\lambda_{i'}(x')\}}{\lambda_{i'}(X_{i'})}\cdot\one{X_{i'}=x'}}}\\
&=0,
\end{align*}
where the last step holds since the denominator tends to infinity almost surely,
by Lemmas \ref{lem:infinite_sum_for_finite_case} and
\ref{lem:convergence_for_finite_case}. This verifies \eqref{eqn:a_s_ineq},
as desired.

\subsection{Proof of Lemma \ref{lem:nec_equiv_finite_case}}

First suppose the graph $G_S$ is disconnected for some nonempty $S\subseteq
\Xcal$. Then we can find a subset $A\subseteq S$ with $A,S\backslash A$ both
nonempty, such that $(x_0,x_1)\not\in E_S$ for all $x_0\in A,x_1\in S\backslash
A$. Let $P$ be the uniform distribution on $S$. We will now show that the
necessary condition \eqref{eq:nec_condition} fails for this $P$ and this
$A$. For any $i$ we have 
\begin{align*}
\min\{(P\circ\lambda_i)(A), (P\circ\lambda_i)(A^c)\}
&=\min\left\{\sum_{x_0\in A}(P\circ\lambda_i)(\{x_0\}),\sum_{x_1\in S\backslash A} (P\circ\lambda_i)(\{x_1\})\right\}\\
&=\min\left\{\sum_{x_0\in A} \frac{\lambda_i(x_0)}{\sum_{x\in S}\lambda_i(x)},\sum_{x_1\in S\backslash A} \frac{\lambda_i(x_1)}{\sum_{x\in S}\lambda_i(x)}\right\}\\
&\leq \min\left\{\frac{|A|\cdot\max_{x_0\in A}\lambda_i(x_0)}{\max_{x\in S}\lambda_i(x)},\frac{|S\backslash A|\cdot\max_{x_1\in S\backslash A} \lambda_i(x_1)}{\max_{x\in S}\lambda_i(x)}\right\}\\
&\leq |S| \min\left\{\max_{x_0\in A}\frac{\lambda_i(x_0)}{\max_{x\in S}\lambda_i(x)},\max_{x_1\in S\backslash A}\frac{\lambda_i(x_1)}{\max_{x\in S}\lambda_i(x)}\right\}\\
&= |S|\max_{x_0\in A,x_1\in S\backslash A} \min\left\{\frac{\lambda_i(x_0)}{\max_{x\in S}\lambda_i(x)},\frac{\lambda_i(x_1)}{\max_{x\in S}\lambda_i(x)}\right\}\\
&\leq  |S|\sum_{x_0\in A,x_1\in S\backslash A} \min\left\{\frac{\lambda_i(x_0)}{\max_{x\in S}\lambda_i(x)},\frac{\lambda_i(x_1)}{\max_{x\in S}\lambda_i(x)}\right\}.
\end{align*}
Therefore,
\[
\sum_{i=1}^\infty \min\{(P\circ\lambda_i)(A), (P\circ\lambda_i)(A^c)\}
\leq |S| \sum_{x_0\in A,x_1\in S\backslash A}\sum_{i\geq
  1}\min\left\{\frac{\lambda_i(x_0)}{\max_{x\in
      S}\lambda_i(x)},\frac{\lambda_i(x_1)}{\max_{x\in S}\lambda_i(x)}\right\}
<\infty,
\]
where the last step holds since, for all $x_0\in A$ and $x_1\in S\backslash A$,
we have $(x_0,x_1)\not\in E_S$ and therefore \smash{$\sum_{i=1}^\infty
\frac{\min\{\lambda_i(x_0),\lambda_i(x_1)\}}{\max_{x\in S} \lambda_i(x)}
<\infty$}. 

Next suppose that $G_S$ is connected for all $S$. We will show that  the
necessary condition \eqref{eq:nec_condition} must hold. Fix any
$A\subseteq\Xcal$ and any $P\in\Mcal_\Xcal(\lambda)$ with $P(A),P(A^c)>0$. Note
that, since $\lambda(x)>0$ for all $x$, we must have $P(\Xcal)<\infty$ in order
to have $P\in\Mcal_\Xcal(\lambda)$. Now let $S = \{x\in\Xcal: P(x)>0\}$ be the
support of $P$. Since $P(A),P(A^c)>0$ we must have $S\cap A\neq \emptyset$ and $S\backslash A\neq \emptyset$.
 Since $G_S$ is a connected graph we must have some $x^*_0\in  S\cap A$, 
$x^*_1\in S\backslash A$ with \smash{$\sum_{i=1}^\infty
  \frac{\min\{\lambda_i(x^*_0),\lambda_i(x^*_1)\}}{\max_{x\in S} \lambda_i(x)}
  =\infty$}.  We now calculate 
\begin{align*}
&\min\{(P\circ\lambda_i)(A), (P\circ\lambda_i)(A^c)\}\\
&=\min\left\{\sum_{x_0\in S\cap A}(P\circ\lambda_i)(\{x_0\}),\sum_{x_1\in S\backslash A} (P\circ\lambda_i)(\{x_1\})\right\}\\
&=\min\left\{\sum_{x_0\in S\cap A} \frac{P(\{x_0\}) \cdot \lambda_i(x_0)}{\sum_{x\in S}P(\{x\}) \cdot  \lambda_i(x)},\sum_{x_1\in S\backslash A} \frac{P(\{x_1\}) \cdot \lambda_i(x_1)}{\sum_{x\in S}P(\{x_1\}) \cdot \lambda_i(x)}\right\}\\
&\geq \min\left\{\sum_{x_0\in S\cap A} \frac{P(\{x_0\}) \cdot \lambda_i(x_0)}{\sum_{x\in S}P(\{x\}) \cdot  \max_{x\in S}\lambda_i(x)},\sum_{x_1\in S\backslash A} \frac{P(\{x_1\}) \cdot \lambda_i(x_1)}{\sum_{x\in S}P(\{x\}) \cdot  \max_{x\in S}\lambda_i(x)}\right\}\\
&\geq \min\left\{\frac{P(\{x^*_0\}) \cdot \lambda_i(x^*_0)}{\sum_{x\in S}P(\{x\}) \cdot  \max_{x\in S}\lambda_i(x)},\frac{P(\{x^*_1\}) \cdot \lambda_i(x^*_1)}{\sum_{x\in S}P(\{x\}) \cdot  \max_{x\in S}\lambda_i(x)}\right\}\\
&\geq \min\left\{\frac{P(\{x^*_0\})}{P(\Xcal)},\frac{P(\{x^*_1\})}{P(\Xcal)},\right\}\cdot  \min\left\{\frac{\lambda_i(x^*_0)}{\max_{x\in S}\lambda_i(x)},\frac{\lambda_i(x^*_1)}{  \max_{x\in S}\lambda_i(x)}\right\},
\end{align*}
and therefore,
\begin{multline*}\sum_{i=1}^\infty\min\{(P\circ\lambda_i)(A), (P\circ\lambda_i)(A^c)\}\\
\geq \min\left\{\frac{P(\{x^*_0\})}{P(\Xcal)},\frac{P(\{x^*_1\})}{P(\Xcal)} ,\right\}\cdot  \sum_{i=1}^\infty \min\left\{\frac{\lambda_i(x^*_0)}{\max_{x\in S}\lambda_i(x)},\frac{\lambda_i(x^*_1)}{  \max_{x\in S}\lambda_i(x)}\right\} =\infty,\end{multline*}
where the last step holds because the first term is positive (as $x^*_0,x^*_1\in
S$, and $S$ is the support of $P$), while the second term is infinite by our
choice of $x^*_0,x^*_1$. This completes the proof.

\subsection{Proof of Lemma \ref{lem:infinite_sum_for_finite_case}}

Since
\smash{$(x,x')\in E^*_S$}, we have
\[
\infty =\sum_{i=1}^\infty \frac{\min\{\lambda_i(x),\lambda_i(x')\}}{\max_{x''\in S} \lambda_i(x'')}
\leq  \sum_{i : \lambda_i(x) \leq \lambda_i(x')}\frac{\lambda_i(x)}{\max_{x''\in S} \lambda_i(x'')}
+ \sum_{i: \lambda_i(x)\geq \lambda_i(x')}\frac{\lambda_i(x')}{\max_{x''\in S} \lambda_i(x'')} .
\]
 Then at least one of the two terms
on the right-hand side must be infinite;
without loss of generality, suppose the first sum above is infinite, i.e., $\sum_{i\in\mathcal{I}} \frac{\lambda_i(x')}{\max_{x''\in S} \lambda_i(x'')} = \infty$ where 
  $\mathcal{I}= \{i\geq 1: \lambda_i(x)\leq\lambda_i(x')\}$.
In this case, we will show that \smash{$\sum_{i\in\mathcal{I}} \one{X_i = x} =
  \infty$}. Note that, since
\smash{$\frac{\min\{\lambda_i(x),\lambda_i(x')\}}{\lambda_i(X_i)}\cdot\one{X_i=x}
  = \one{X_i=x}$} for all $i\in\mathcal{I}$, this will be sufficient to
establish the lemma. 

Now consider the subsequence \smash{$X_{\mathcal{I}}$}. By definition of
weighted exchangeability, this subsequence is
\smash{$\lambda_{\mathcal{I}}$}-weighted exchangeable, and the weight sequence
\smash{$\lambda_{\mathcal{I}}$} satisfies  
\[
\sum_{j=1}^\infty\frac{(\lambda_{\mathcal{I}})_j(x)}{\max_{x''\in\Xcal}(\lambda_{\mathcal{I}})_j(x'')}
= \sum_{i\in\mathcal{I}}\frac{\lambda_i(x)}{\max_{x''\in S} \lambda_i(x'')} = \infty.
\]
Next we need another lemma.

\begin{lemma}\label{lem:zero_or_inf}
Suppose that $Q$ is a $\lambda$-weighted exchangeable distribution on a  
finite set $\Xcal$. Fix any $x\in\Xcal$, and suppose also that
\[
\sum_{i=1}^\infty\frac{\lambda_i(x)}{\max_{x'\in\Xcal}\lambda_i(x)}= \infty.
\]
Then
\[
\Pp{Q}{0< \sum_{i=1}^\infty \one{X_i = x} <\infty} = 0.
\]
That is, it holds almost surely that $X_i=x$ occurs either zero times or
infinitely many times. 
\end{lemma}

This result, applied to the \smash{$\lambda_{\mathcal{I}}$}-weighted
exchangeable sequence \smash{$X_{\mathcal{I}}$}, implies that 
\[
\PP{0 < \sum_{j=1}^\infty \one{(X_{\mathcal{I}})_j=x}<\infty} = 0,
\]
or equivalently,
\[
\PP{0 < \sum_{i\in\mathcal{I}}\one{X_i=x}<\infty} = 0.
\]
To complete the proof, defining
\[
p = \PP{\sum_{i\in\mathcal{I}}\one{X_i=x}=0},
\]
we need to prove that $p=0$. Suppose instead that $p>0$.
Recall that by definition of $Q_S$, we know that \smash{$\Pp{Q_S}{
    \sum_{i\geq1}\one{X_i=x}=0} =0$}, and so for sufficiently large $n\geq 1$ we
have 
\[
\PP{\sum_{i=1}^n \one{X_i = x} = 0}  < p.
\]
Fix any such $n$, and define \smash{$\mathcal{I}' = \mathcal{I}\cup
  \{1,\dots,n\}$}. Then applying Lemma \ref{lem:zero_or_inf} to the subsequence
\smash{$X_{\mathcal{I}'}$}, which is \smash{$\lambda_{\mathcal{I}'}$}-weighted
exchangeable, we have 
\[
\PP{0 < \sum_{i\in\mathcal{I}'}\one{X_i=x}<\infty} = 0.
\]
We can also calculate
\begin{multline*}
\PP{0 < \sum_{i\in\mathcal{I}'}\one{X_i=x}<\infty} 
\geq \PP{\sum_{i=1}^n \one{X_i=x}>0, \sum_{i\in\mathcal{I}}\one{X_i =x}=0}\\
\geq \PP{\sum_{i\in\mathcal{I}}\one{X_i =x}=0}- \PP{\sum_{i=1}^n \one{X_i=x}=0}
= p - \PP{\sum_{i=1}^n \one{X_i=x}=0} >0.
\end{multline*}
We have reached a contradiction, as desired.

\subsection{Proof of Lemma \ref{lem:convergence_for_finite_case}}

Fix \smash{$(x,x')\in E^*_S$}. For any $n> \ell \geq 0$, define
\[
\widetilde{P}_{\ell,n}(x;x') = \frac{\sum_{i=\ell+1}^n
  \frac{\min\{\lambda_i(x),\lambda_i(x')\}}{\lambda_i(X_i)}\cdot\one{X_i=x}}{\sum_{i=\ell+1}^n
  \frac{\min\{\lambda_i(x),\lambda_i(x')\}}{\lambda_i(X_i)}\cdot\one{X_i \in\{x,x'\}}}.
\]
Note that
\[
\left|\widetilde{P}_{\ell,n}(x;x')  - \widetilde{P}_n(x;x') \right| \leq \frac{\ell}{\sum_{i=1}^n
  \frac{\min\{\lambda_i(x),\lambda_i(x')\}}{\lambda_i(X_i)}\cdot\one{X_i \in\{x,x'\}}}
\]
by construction, and this denominator tends to infinity almost surely by Lemma
\ref{lem:infinite_sum_for_finite_case}, so we therefore have 
\[
\limsup_{n\to\infty} \widetilde{P}_n(x;x') =\limsup_{n\to\infty}
\widetilde{P}_{\ell,n}(x;x') 
\]
holding almost surely, for all $\ell$. Therefore,
\[
\limsup_{n\to\infty} \widetilde{P}_n(x;x') =\limsup_{\ell\to\infty}
\left(\limsup_{n\to\infty} \widetilde{P}_{\ell,n}(x;x')\right) 
\]
holds almost surely. Next, recall that, 
 for $\ell \geq 1$, we define the $\sigma$-algebra $\Fcal_\ell = \sigma(X_{\ell+1},X_{\ell+2},\dots)$.   
 Since \smash{$\widetilde{P}_{\ell,n}(x;x')$} is
$\Fcal_\ell$-measurable by construction, \smash{$\limsup_{n\to\infty}
  \widetilde{P}_{\ell,n}(x;x')$} is also $\Fcal_\ell$-measurable. Therefore, 
\smash{$\limsup_{\ell\to\infty} (\limsup_{n\to\infty}
  \widetilde{P}_{\ell,n}(x;x'))$} is $\Fcal_\ell$-measurable for every $\ell$,
and is therefore $\Ftail$-measurable.  

Now let
\[
\widetilde{P}'(x;x') =\limsup_{\ell\to\infty} \left(\limsup_{n\to\infty} \widetilde{P}_{\ell,n}(x;x')\right),
\]
and let
\[
\widetilde{P}(x;x') = \begin{cases} \widetilde{P}'(x;x') & \textnormal{if $\widetilde{P}'(x;x') \in(0,1)$}, \\
0.5 &\textnormal{otherwise}.\end{cases}
\]
By construction, \smash{$\widetilde{P}(x;x')$} is $\Ftail$-measurable.

From the work above, we see that \smash{$\widetilde{P}'(x;x') =
  \limsup_{n\to\infty} \widetilde{P}_n(x;x')$} holds almost surely. We next
need to show that \smash{$\lim_{n\to\infty}\widetilde{P}_n(x;x')$} exists almost
surely, to verify that \smash{$\widetilde{P}'(x;x') = \lim_{n\to\infty}
  \widetilde{P}_n(x;x')$} holds almost surely. Define
\[
p_i = p_i(X_i) =
\frac{\min\{\lambda_i(x),\lambda_i(x')\}}{\lambda_i(X_i)}\one{X_i\in\{x,x'\}}
\in [0,1].
\]
Similar to our earlier construction in the proof of Theorem \ref{thm:suff_condition}, draw \smash{$U_1,U_2,\dots \iidsim
  \mathrm{Unif}[0,1]$}, independently of $X$, and define for each $i=1,2,\dots$,       
\[
B_i = \one{U_i\leq p_i(X_i)}.
\]
Define, as before, \smash{$M = \sum_{i=1}^\infty B_i$}, and 
\[
I_m = \min\{i > I_{m-1} : B_i = 1\}, \quad \textnormal{for $m=1,\dots,M$}. 
\]
i.e., $I_1<I_2<\dots$ enumerates all indices $i$ for which $B_i=1$. Then define  
\[
\check{X} = 
\begin{cases}
(X_{I_1},X_{I_2},X_{I_3},\dots) & \textnormal{if $M=\infty$}, \\
(X_{I_1},\dots,X_{I_M},x,x,\dots) & \textnormal{otherwise}.
\end{cases}
\]
Note that \smash{$\check{X}\in\{x,x'\}^{\infty}$} by construction. 
By Lemma \ref{lem:infinite_sum_for_finite_case}, together with the second 
Borel--Cantelli Lemma, we have $M=\infty$ almost surely.
Next, we can verify that \smash{$\check{X}$} is exchangeable---the proof of
this claim is essentially identical to the proof of Lemma
\ref{lem:exch_inf_subseq} (where we establish the analogous result for
establishing Theorem \ref{thm:suff_condition}). Therefore, by de Finetti's
theorem,  
\[
\lim_{m\to\infty}\frac{\one{\check{X}_1 =x} + \dots + \one{\check{X}_m
    =x}}{m}
\]
converges almost surely; as in the proof of Theorem \ref{thm:suff_condition} we
can see that this is equal (almost surely) to 
\[
\lim_{n\to\infty}\frac{\sum_{i=1}^n B_i \cdot \one{X_i = x}}{\sum_{i=1}^n B_i},
\]
and again following the same steps as in the proof of Theorem
\ref{thm:suff_condition}, this  is equal (almost surely) to 
\[
\lim_{n\to\infty}\frac{\sum_{i=1}^n p_i \cdot \one{X_i = x}}{\sum_{i=1}^n p_i} = 
\lim_{n\to\infty}\frac{\sum_{i=1}^n
  \frac{\min\{\lambda_i(x),\lambda_i(x')\}}{\lambda_i(X_i)}\cdot \one{X_i = x}}{\sum_{i=1}^n
  \frac{\min\{\lambda_i(x),\lambda_i(x')\}}{\lambda_i(X_i)}\cdot \one{X_i\in\{x,x'\}}} 
= \lim_{n\to\infty}\widetilde{P}_n(x;x').
\]
In particular this verifies that
\smash{$\lim_{n\to\infty}\widetilde{P}_n(x;x')$} exists almost surely, and 
returning to de Finetti's theorem for the exchangeable subsequence, conditional
on the random value \smash{$\widetilde{P}'(x;x')$} to which the above limit
converges, \smash{$\check{X}_{I_i}$}, $i=1,2,\dots$ are i.i.d.\ draws from
\smash{$\widetilde{P}'(x;x')\cdot\delta_x + (1-\widetilde{P}'(x;x'))\cdot 
  \delta_{x'}$}.    

To complete the proof we need to check that
\smash{$\widetilde{P}'(x;x')\in(0,1)$} almost surely, so that we have
\smash{$\widetilde{P}(x;x') \aseq \lim_{n\to\infty} \widetilde{P}_n(x;x')$}. 
Suppose instead that \smash{$\widetilde{P}'(x;x')=0$} with positive
probability. Then, on this event, we have 
\[
\PPst{\sum_{m=1}^\infty\one{\check{X}_{I_m}=x} =0}{X}=1
\]
almost surely.
In order to have \smash{$\sum_{m=1}^\infty\one{\check{X}_{I_m}=x} =0$}, we
need to have $B_i=0$ for all $i$ with $X_i = x$, that is, 
\[
1 = \PPst{\sum_{m=1}^\infty\one{\check{X}_{I_m}=x} =0}{X} 
= \PPst{B_i =0 \textnormal{ for all $i$ with $X_i=x$}}{X} = \prod_{i: X_i =
  x}(1-p_i).  
\]
Since $p_i>0$ for any $i$ with $X_i=x$ (because $\lambda_i$ is positive-valued
for each $i$), this implies that we must have
\smash{$\sum_{i=1}^\infty\one{X_i=x}=0$} in order for this equality to hold; in 
other words, if \smash{$\widetilde{P}'(x;x')=0$} with positive probability, then
\smash{$\sum_{i=1}^\infty\one{X_i=x}=0$} with positive probability. 
But by definition of $Q_S$ we have \smash{$\PP{\sum_{i=1}^\infty \one{X_i
      =x}>0}=1$}. This is a contradiction, so \smash{$\widetilde{P}'(x;x')=0$}
can only hold with probability zero. We can similarly show that
\smash{$\widetilde{P}'(x;x')=1$} can only hold with probability zero, which
completes the proof. 

\subsection{Proof of Lemma \ref{lem:zero_or_inf}}

Assume for the sake of contradiction that
\[
\Pp{Q}{\sum_{i=1}^\infty \one{X_i = x} =k} > 0,
\]
for some finite and positive $k$. The event that \smash{$\sum_{i=1}^\infty
  \one{X_i = x} =k$}, i.e., that $x$ is observed exactly $k$ many times, is
$\Ecal_\infty$-measurable. Let $Q'$ be the distribution of $X$ conditional on
this event. Then $Q'$ is also $\lambda$-weighted exchangeable (this can be
verified exactly as in the proof of Theorem \ref{thm:main}---specifically, the
proof that $\Lambda_{\dF}\subseteq\Lambda_{\01}$). Now define \smash{$p_i =
  \Pp{Q'}{X_i = x}$}. We must have  
\[
\sum_{i=1}^\infty p_i = \sum_{i=1}^\infty\Pp{Q'}{X_i=x} =
\Ep{Q'}{\sum_{i=1}^\infty\one{X_i =x}} = \Ep{Q'}{k} = k.
\]
We also have, for any $i\neq j$,
\begin{align*}
p_i &=\Ep{Q'}{\one{X_i = x}}\\
&=\Ep{Q'}{\one{X_i = x}\cdot \frac{\lambda_i(X_i)\lambda_j(X_j)}{\lambda_i(X_i)\lambda_j(X_j)}}\\
&=\Ep{Q'}{\one{X_j = x}\cdot \frac{\lambda_i(X_j)\lambda_j(X_i)}{\lambda_i(X_i)\lambda_j(X_j)}}\\
&\leq \Ep{Q'}{\one{X_j = x}} \cdot \frac{\lambda_i(x)/ \min_{x'\in\Xcal}\lambda_i(x')}{\lambda_j(x)/\max_{x'\in\Xcal}\lambda_j(x')}\\
&= p_j \cdot \frac{\lambda_i(x)/ \min_{x'\in\Xcal}\lambda_i(x')}{\lambda_j(x)/\max_{x'\in\Xcal}\lambda_j(x')},
\end{align*}
where the third equality applies Proposition \ref{prop:wtd_exch_equiv}.
In particular, this implies that each $p_i$ is positive (since these
inequalities can be satisfied either if each $p_i$ is positive, or each $p_i$ is
zero---and they cannot all be zero when we choose $k>0$). 

Therefore, for any $i$,
\[
\sum_{j=1}^\infty p_j \geq p_i \cdot \left(\frac{\lambda_i(x)/ \min_{x'\in\Xcal}\lambda_i(x')}{\lambda_j(x)/\max_{x'\in\Xcal}\lambda_j(x')}\right)^{-1}
= p_i \cdot \frac{\min_{x'\in\Xcal}\lambda_i(x')}{\lambda_i(x)} \cdot\sum_{j=1}^\infty\frac{\lambda_j(x)}{\max_{x'\in\Xcal}\lambda_j(x')}.\]
But this is a contradiction, because on the left-hand side we have
\smash{$\sum_{j=1}^\infty p_j=k$} by construction, while on the right-hand side, 
\smash{$p_i \cdot \frac{\min_{x'\in\Xcal}\lambda_i(x')}{\lambda_i(x)} >0$}, and 
\smash{$\sum_{j=1}^\infty\frac{\lambda_j(x)}{\max_{x'\in\Xcal}\lambda_j(x')}
  =\infty$} by the assumption in the lemma.

\end{document}